\documentclass[11pt,a4paper,reqno]{amsart}
\usepackage[T1]{fontenc}
\usepackage[utf8]{inputenc}

\usepackage{stmaryrd}

\title[Relative (T), simplices of invariant measures, and e.c.\ models]{Relative Property (T), simplices of invariant measures, and existentially closed models}

\author{Tomás Ibarlucía}
\address{Universit\'e Paris Cit\'e \\
  CNRS \\
  IMJ-PRG \\
  F-75006 Paris \\
  France.}
\urladdr{\url{https://webusers.imj-prg.fr/~tomas.ibarlucia}}

\usepackage{eucal} 

\usepackage[
  left=33mm,
  right=33mm,
  top=31mm,
  bottom=30mm
]{geometry}

\usepackage{hyperref}

\usepackage{my-macros}

\usepackage{enumitem}
\setlist[enumerate,1]{label=(\roman*), font=\normalfont}
\setlist[enumerate,2]{label=(\arabic*), font=\normalfont}

\usepackage{tikz-cd}

\ifx\thmnum\undefined
  \newtheorem{thmnum}{}[section]
\fi

\usepackage{aliascnt}

\newcounter{quotethmcnt}

\newcommand{\mynewthm}[3][]{
  \newaliascnt{#2}{thmnum}
  \newtheorem{#2}[#2]{#3}
  \aliascntresetthe{#2}
  \newtheorem*{#2*}{#3}
  \expandafter\newcommand\csname #2autorefname\endcsname{#3}
  \expandafter\renewcommand\csname the#2\endcsname{\thethmnum}
}

\def\equationautorefname~#1\null{(#1)}
\def\itemautorefname~#1\null{#1}

\mynewthm{thm}{Theorem}

\mynewthm{cor}{Corollary}
\mynewthm{lem}{Lemma}
\mynewthm{prop}{Proposition}
\mynewthm{fact}{Fact}
\mynewthm{conjecture}{Conjecture}

\theoremstyle{definition}
\mynewthm{dfn}{Definition}
\mynewthm{ntn}{Notation}

\theoremstyle{remark}
\mynewthm{question}{Question}
\mynewthm{rmk}{Remark}
\mynewthm{example}{Example}
\mynewthm{claim}{Claim}

\numberwithin{equation}{section}

\newcommand{\qf}{\mathrm{qf}}

\newcommand{\aff}{\mathrm{aff}}
\newcommand{\ext}{\mathrm{ext}}
\newcommand{\full}{\mathrm{full}}
\newcommand{\ec}{\mathrm{ec}}
\newcommand{\aec}{\mathrm{aec}}
\newcommand{\crc}{\mathrm{cr}}
\newcommand{\lsc}{\mathrm{lsc}}
\newcommand{\fM}{\mathfrak{M}}
\newcommand{\fs}{\mathfrak{s}}
\newcommand{\cont}{\mathrm{cont}}
\newcommand{\inv}{\mathrm{inv}}

\newcommand{\Spec}{\mathrm{Spec}}
\newcommand{\fI}{\mathfrak{I}}

\newcommand{\PMP}{\mathrm{PMP}}
\newcommand{\PrA}{\mathrm{PrA}}

\DeclareFontFamily{U}{mathx}{}
\DeclareFontShape{U}{mathx}{m}{n}{<-> mathx10}{}
\DeclareSymbolFont{mathx}{U}{mathx}{m}{n}
\DeclareMathAccent{\widehat}{0}{mathx}{"70}
\DeclareMathAccent{\widecheck}{0}{mathx}{"71}

\begin{document}

\begin{abstract}
We prove a Bauer--Poulsen dichotomy theorem for simplices of invariant measures associated with permutation groups.

More precisely, let $G$ be a transitive group of permutations of a countable set $\cS$, and let $H$ be the stabilizer of a point of $\cS$.
Let $\cl{G}$ and $\cl{H}$ denote their closures in the topology of pointwise convergence.
Assume the Polish group $\cl{H}$ has relative Property (T) in $\cl{G}$.
Then the simplex $\cM_\inv(2^\cS)$ of invariant probability measures for the induced action $G\actson 2^\cS$ is Bauer if and only if $\cl{G}$ has Property (T), and is Poulsen otherwise.
This addresses some examples and questions considered by Austin.

We deduce this result from a more general model-theoretic statement of independent interest.
To this end, we initiate the study of existentially closed models in affine logic.
\end{abstract}

\maketitle

\vspace{-1em}

\setcounter{tocdepth}{1}
\tableofcontents

\vspace{-2em}

\section{Introduction}

In \cite{Glasner1997}, Glasner and Weiss famously proved the following dichotomy theorem on the geometry of simplices of invariant measures, which is also a characterization of Kazhdan groups:

\begin{thm*}[{\cite[Thm.~2]{Glasner1997}}]
Let $G$ be a countable group. Then $\cM_\inv(2^G)$, the collection of invariant probability measures of the topological Bernoulli shift, is either a Bauer simplex or the Poulsen simplex.

Moreover, $\cM_\inv(2^G)$ is a Bauer simplex if and only if $G$ has Property (T).
\end{thm*}

We recall here that a Choquet simplex is \emph{Bauer} if its extreme points (in this case, the ergodic measures) form a closed set. The \emph{Poulsen simplex} is the unique, up to affine homeomorphism, metrizable Choquet simplex with at least two elements whose extreme points are dense \cite{Poulsen1961,Lindenstrauss1978}.

By contrast, the simplices $\cM_\inv(X)$ arising from general group actions $G\actson X$ by homeomorphisms on compact metrizable spaces can take any possible form.
Indeed, Downarowicz \cite{Downarowicz} had previously shown that every metrizable Choquet simplex is affinely homeomorphic to $\cM_\inv(X)$ for some subshift $X\subseteq 2^\bZ$ of the group of integers.
In turn, a classical result of Choquet and Haydon \cite{Haydon1975} states that every Polish space can be obtained, up to homeomorphism, as the set of extreme points of some metrizable Choquet simplex.

An interesting intermediate situation is given by the simplices of invariant measures $\cM_\inv(X)$ of $G$-systems of the form $X = 2^\cS$, and more generally,
\begin{equation*}
X = K^\cS,
\end{equation*}
where the action is induced by a permutation group $G\actson \cS$ on a countable set, and $K$ is a compact metrizable space.
Simplices of this form have been studied principally in the theory of \emph{exchangeable random variables}.
A foundational result of the area, proved in the 1930s by de Finetti, yields for instance that for the group $G=\Sym_0(\cS)$ of finitely supported permutations of $\cS$, one can identify:
\begin{equation*}
\cM_\inv\big(K^\cS\big)\cong \cM\big(\cM(K)\big).
\end{equation*}
In other words, $\Sym_0(\cS)$-invariant measures on $K^\cS$ can be represented by probability measures on $\cM(K)$.
In particular, in this case, $\cM_\inv(K^\cS)$ is a Bauer simplex.
A rich theory of representation theorems for more complex \guillemotleft exchangeability contexts\guillemotright\ was developed by Aldous, Hoover and Kallenberg, among others \cite{Kallenberg2005}.

Most concrete examples of simplices $\cM_\inv(K^\cS)$ arising in exchangeability theory are Bauer simplices.
One reason for this was given by Austin in an unpublished note \cite{Austin2008}.
He proved there that if the invariant measures of a system of the form $G\actson K^\cS$ are \emph{representable} in a precise, general sense introduced in that note, then $\cM_\inv(K^\cS)$ must be a Bauer simplex.
Another powerful reason, related to Property (T) for Polish groups, will be recalled later in \autoref{sec:property-T} (see \autoref{rmk:oligomorphic-groups-exchangeability-theory}).

However, in \cite{Austin2008}, Austin also showed that in the case of \emph{cube-exchangeable} measures (an example introduced by Aldous \cite{Aldous1985}), the convex set $\cM_\inv(K^\cS)$ is the Poulsen simplex.
This exchangeability context corresponds to $\cS = \bF_2^{\oplus\bN}$, the infinite-dimensional vector space over $\bF_2$, with the canonical action of the countable group $G = \bF_2^{\oplus\bN}\rtimes\Sym_0(\bN)$ by affine transformations.
A representation theorem à la Aldous--Hoover--Kallenberg is therefore impossible in this case.

At the end of \cite{Austin2008}, Austin asks for a general result that would subsume this particular Poulsen example.
He goes further to conjecture that for arbitrary exchangeability contexts $G\actson \cS$ with no finite orbits, the Poulsen property should in fact be typical, i.e., only fail under very special circumstances.
He also raises the question of whether the simplex $\cM_\inv(K^\cS)$ can ever be neither Bauer nor Poulsen.

In the present paper we prove the following dichotomy theorem, which is a partial answer to these questions.

\begin{thm}[See \autoref{thm:main-final}]
\label{thm:intro:main}
Let $G\actson\cS$ be a transitive permutation group on a countable set, and let $H\leq G$ be the stabilizer of a point of $\cS$.
Let $\cl{G}$ and $\cl{H}$ denote the closures of $G$ and $H$ inside the full symmetric group $\Sym(\cS)$, for the topology of pointwise convergence.
Assume that the pair of Polish groups $(\cl{G},\cl{H})$ has Property (T) (see \autoref{dfn:relative-(T)}).
Then for every non-empty compact metrizable space $K$, the Choquet simplex $\cM_\inv\big(K^\cS\big)$ is either Bauer or Poulsen.

Moreover, under the same assumptions and provided that $|K|>1$, $\cM_\inv\big(K^\cS\big)$ is a Bauer simplex if and only if $\cl{G}$ has Property (T).
\end{thm}

We note first that this generalizes the dichotomy theorem of Glasner and Weiss, since the pair $(\cl{G},1)$ trivially has Property (T), for any group $G$.

We also note that Aldous's cube-exchangeability context falls within the scope of our result. Indeed, in that example, the closure $\cl{H}$ of the stabilizer of $0\in \cS = \bF_2^{\oplus\bN}$ is the Polish group $\Sym(\bN)$, which has Property (T) as first shown by Tsankov \cite{TsankovUnitary}.
In particular, the pair $(\cl{G},\cl{H})$ associated to this example has Property (T).

The same holds, more generally, for Austin's \guillemotleft finer-grained cube\guillemotright-exchangeability contexts of the form $\cS = (\bZ/m\bZ)^{\oplus\bN}$ and $G = (\bZ/m\bZ)^{\oplus\bN}\rtimes \Sym_0(\bN)$, whose analysis is left open in \cite[\textsection5.5]{Austin2008}.
Moreover, in all these cases, the corresponding groups $\cl{G}$ fail to have Property (T).
Our theorem therefore shows that these exchangeability contexts fall in the Poulsen side of the dichotomy, confirming Austin's suspicion. See \autoref{cor:finer-grained-Poulsen}.

It is worth emphasizing that if one starts with a countable group $G$, Property (T) for the pair $(\cl{G},\cl{H})$ is a weaker and more applicable hypothesis than Property (T) for the pair $(G,H)$, as illustrated by the previous examples.
In this sense, our result can be seen as a demonstration of how Polish group theory can be relevant for problems of more classical origin.

The proof that we provide, on the other hand, is based on model theory.
Part of the motivation for the present work is in fact to exhibit an application of the ideas developed recently in joint work with I. Ben Yaacov and T. Tsankov \cite{BITaffine}, on simplicial theories in affine logic.
As we explain below, we will deduce \autoref{thm:intro:main} from a general, model-theoretic Bauer--Poulsen dichotomy, which I hope may find further applications in ergodic theory.

Before discussing these ideas, however, let us say a word about the possibility of proving the dichotomy part of \autoref{thm:intro:main} with the method of \cite{Glasner1997}.\footnote{The moreover part of \autoref{thm:intro:main} will be proved beforehand in \autoref{prop:prop-T-iff-Bauer}, by a direct adaptation of the argument in \cite[Prop.~2]{Glasner1997}.}
The original argument of \cite{Glasner1997} can be adapted to establish the weaker form of our result where one assumes $G$ is countable and the pair $(G,H)$ has Property (T) as a pair of discrete groups.
In order to make the same argument work for $(\cl{G},\cl{H})$, it would be necessary to adapt a result of Bekka--Valette \cite{BekkaValette} to the setting of non-Archimedean Polish groups.
More precisely, one should show that if $\cl{G}$ does not have Property (T), then $\cl{G}$ admits a continuous unitary representation with almost invariant vectors but without finite-dimensional subrepresentations.
The argument in \cite{BekkaValette}, in turn, relies on Guichardet's theorem that Property (FH) implies Property (T).
However, it is an open question whether this result holds for non-Archimedean Polish groups.

\subsection*{Model-theoretic dichotomies}

The Glasner--Weiss Theorem and its proof have been used by various authors to obtain other interesting Bauer--Poulsen dichotomies \cite{GlasnerYair2024,BowenLamplighter2015}.
It is also worth mentioning that many remarkable recent results continue to show that simplices of invariant measures arising in natural situations (as well as simplices of traces on groups and $C^*$-algebras) tend to be Bauer \cite{Slutsky2026}\footnote{We note that in this recent preprint, relative Property (T) also plays a crucial role.} or Poulsen \cite{BowenIRS2015,IoanaTrace2025,GaoCharacters2026}.
However, to my knowledge, Glasner--Weiss's result remained in essence, for a long time, the only known general dichotomy theorem of this kind.

Recently in \cite{BITaffine}, while studying affine first-order theories whose type spaces form Choquet simplices, we unexpectedly established the following:

\begin{thm*}[{\cite[Thm.~20.8]{BITaffine}}, assuming a countable language]
Let $T$ be an affinely axiomatizable, affinely complete theory in continuous logic.
Assume $T$ is \emph{simplicial}, meaning that the affine type spaces $\tS^\aff_n(T)$, $n\in\bN$, are all Choquet simplices.
Then the simplices $\tS^\aff_n(T)$ are either Bauer or the Poulsen simplex.
\end{thm*}

We recall briefly that continuous logic is a generalization of classical first-order logic in which the predicates are real-valued rather than $2$-valued \cite{BenYaacov2010, BBHU}.
The probability measure-preserving (\emph{pmp}) systems of a given group $G$ can be axiomatized in this formalism, and this viewpoint allows for interesting interactions between ergodic theory and model theory \cite{Berenstein2018p, Giraud2019p, IbarluciaTsankov, BerHenIba, GST2025}.

The continuous first-order theory (i.e., the axiomatization) of the pmp $G$-systems, which we denote by $\PMP_G$, is in fact \emph{affine}, in that it only requires linear combinations and constants, rather than more general continuous connectives such as $\max$ and $\min$.
The fragment of continuous logic where the connectives are restricted to affine functions was first considered and studied by Bagheri \cite{Bagheri2014,Bagheri2014a}, and is now referred to as \emph{affine logic} \cite{BITaffine,Bagheri2024p}.
Despite being less expressive, affine logic has richer structural features than full continuous logic, which stem from the fact that affine type spaces are \emph{compact convex sets}.
In particular, affine logic allows for a model-theoretic treatment of the notion of ergodicity \cite[\textsection28]{BITaffine}.

It turns out that the Glasner--Weiss theorem can be derived from our results in \cite{BITaffine}, albeit in a slightly indirect manner which I only realized recently.\footnote{More precisely, the theorem stated above implies a Bauer--Poulsen dichotomy for the extreme completions of $\PMP_G$ (see \cite[Cor.~28.5]{BITaffine}), which is in fact a strong version of a theorem of Abért--Weiss \cite[Thm.~3]{AbertWeiss}.
Their result, in turn, implies the Glasner--Weiss theorem, as explained in their paper.

Furthermore, if one cares only about the dichotomy part of \cite[Thm.~2]{Glasner1997}, then its derivation from \cite[Cor.~28.5]{BITaffine} can also be done in a purely model-theoretic manner.}
However, this is not the approach I would like to develop for proving \autoref{thm:intro:main}.

Our goal is to formulate a more general model-theoretic statement that subsumes both \cite[Thm.~20.8]{BITaffine} and the dichotomy part of \autoref{thm:intro:main} (and hence of \cite[Thm.~2]{Glasner1997}) as direct special cases, while paving the way for other applications of these ideas.
To this end, we will recast much of the theory developed in \cite{BITaffine} in the setting of \emph{Robinson theories},\footnote{The name is due to Hrushovski \cite{HruPatterns}.} namely universal theories with the amalgamation property, and the study of their \emph{existentially closed models}.
This provides the appropriate framework for \guillemotleft low-quantification model theory\guillemotright, which is naturally suited to the present and related applications.

To motivate this shift, we observe that $\PMP_G$ is an affine Robinson theory, and that its \emph{quantifier-free} type spaces are precisely the simplices of invariant measures we are interested in.
Indeed, as was observed in \cite[Prop.~28.2]{BITaffine}, we have canonical affine homeomorphisms
\begin{equation}
\label{eq:M-inv-Sqf-PMPG}\tag{*}
\cM_\inv\big((2^n)^G\big) \cong \tS^\qf_n(\PMP_G)
\end{equation}
between the simplices of invariant measures of the Bernoulli shifts $G\actson (2^n)^G$ and the spaces of quantifier-free types of the theory $\PMP_G$.

As we show in \autoref{sec:homogeneous}, given a subgroup $H\leq G$, one may also define an affine Robinson theory $\PMP_{G/H}$ so that we have, similarly,
\begin{equation*}
\cM_\inv\big((2^n)^{G/H}\big) \cong \tS^\qf_n(\PMP_{G/H}).\footnote{In fact, one could provide a construction that works for the set of invariant measures of any metrizable flow $G\actson X$ and its powers, but the one that we provide in \autoref{sec:homogeneous} is tailored to the case of $G\actson 2^{G/H}$.}
\end{equation*}
The models of $\PMP_{G/H}$ can be seen as pmp $G$-systems generated by a distinguished $\sigma$-algebra of $H$-invariant sets.
The theory $\PMP_{G/H}$ is in addition \emph{irreducible}, meaning its models have the joint embedding property, and \emph{face-preserving}, in the sense we explain next.

\subsection*{Face-preserving Robinson theories}
Robinson theories and their existentially closed models are a well-known generalization of first-order logic.
Indeed, the models of any classical or continuous logic theory can be identified with the existentially closed models of the universal part of its Morleyization, which is a well-behaved Robinson theory.
Conversely, a continuous Robinson theory $T$ is well-behaved (and equivalent to a standard one) when any of the following equivalent conditions holds:
\begin{enumerate}
\item $T$ admits a \emph{model companion/completion};
\item the class of existentially closed models of $T$ is axiomatizable;
\item the variable restriction maps $\tS^\qf_{n+m}(T)\to \tS^\qf_n(T)$ are open.
\end{enumerate}

At this point, it is worth making a short digression to mention that for the theories of the form $\PMP_G$, the existence of a model completion is an intriguing open question.
It was proved to exist for amenable groups early on by Berenstein and Henson \cite{Berenstein2018p}, and for non-abelian free groups only more recently in \cite{BerHenIba}.
In a remarkable recent work \cite{GST2025}, Goldbring, Seward and Tucker-Drob extended the positive answer to all \emph{approximately treeable} groups, precisely by showing that the pushforward maps $\cM_\inv\big((2^{n+m})^G\big) \to \cM_\inv\big((2^n)^G\big)$ are open in that case.

In \autoref{sec:appendix:open-and-face-preserving}, we will show that in the category of compact convex sets, the role of open maps is played by those maps that are both open and \emph{face-preserving}, meaning they send faces to faces.
Correspondingly, as we observe in \autoref{sec:appendix:model-completions}, the following properties are equivalent for an affine Robinson theory $T$:
\begin{enumerate}
\item $T$ admits an \emph{affine model companion/completion};
\item the class of affinely existentially closed models of $T$ is axiomatizable;
\item the canonical maps $\tS^\qf_{n+m}(T)\to \tS^\qf_n(T)$ are open and face-preserving.
\end{enumerate}
The fact that the variable restriction maps $\tS^\aff_{n+m}(T)\to \tS^\aff_n(T)$ between full affine type spaces are always face-preserving was already observed (in an equivalent formulation) and thoroughly exploited in \cite{BITaffine}.

As it turns out, for the theories $\PMP_G$ and $\PMP_{G/H}$, even though we have no reason to believe the variable restriction maps $\tS^\qf_{n+m}(T)\to \tS^\qf_n(T)$ should be always open, they happen to be face-preserving for free.
Indeed, the latter is essentially equivalent to the fact that a factor of an ergodic system is also ergodic.\footnote{By contrast, the theory $\mathrm{TvN}$ of tracial von Neumann algebras, as presented in \cite[\textsection29]{BITaffine}, is an irreducible, affine Robinson theory, but is not face-preserving, since a subalgebra of a von Neumann factor need not be a factor.}

It is therefore natural to wonder:
\begin{question*}
What special properties does the class of \emph{affinely existentially closed} models of an affine Robinson theory have, under the mere assumption that the variable restriction maps are face-preserving?
\end{question*}

As we show in Sections \ref{sec:face-preserving-theories}--\ref{sec:dichotomy}, towards our generalization of \cite[Thm.~20.8]{BITaffine}, this hypothesis is in fact enough to reproduce much of the theory developed in \cite{BITaffine}.
In particular, by an almost verbatim adaptation of \cite[Thm.~12.3]{BITaffine}, we show the following.

\begin{thm}[See \autoref{thm:decomposition-for-aec-models}]
\label{thm:intro:decomposition}
Let $T$ be a qf-simplicial, face-preserving Robinson theory.
Then every affinely existentially closed model of $T$ can be decomposed as a direct integral of qf-extremal models.
\end{thm}

Now, given an affine Robinson theory $T$, let $T^{\aec,\ext}_{\forall\exists\text{-}\cont}$ denote the common $\forall\exists$-continuous logic theory of its affinely existentially closed, qf-extremal models.

The statement of our general dichotomy theorem can be phrased as follows.

\begin{thm}[See \autoref{thm:abstract-dichotomy}]
\label{thm:intro:general-dichotomy}
Let $T$ be a qf-simplicial, face-preserving, irreducible Robinson theory.
Assume moreover that every model of $T^{\aec,\ext}_{\forall\exists\text{-}\cont}$ can be decomposed as a direct integral of qf-extremal models of $T$.
Then the simplices $\tS^\qf_n(T)$ are either Bauer or the Poulsen simplex.
\end{thm}

The main assumption in \autoref{thm:intro:general-dichotomy}, which makes the dichotomy work in the face-preserving Robinson setting, requires that the decomposability property enjoyed by the a.e.c.\ models (as per \autoref{thm:intro:decomposition}) be somehow encoded in their $\forall\exists$-continuous theory. However, in fact, it is enough to consider the theory of the qf-extremal a.e.c.\ models.

By the Ergodic Decomposition Theorem (and, more precisely, by \cite[Cor.~28.4]{BITaffine}), every model of $\PMP_G$ can be decomposed as a direct integral of qf-extremal (i.e., ergodic) models.
Therefore, \autoref{thm:intro:general-dichotomy} applies immediately to that case; see \autoref{sec:Glasner-Weiss}.

As we show subsequently in \autoref{sec:pmp}, under the hypothesis that the pair $(\cl{G},\cl{H})$ has Property (T), the main assumption of \autoref{thm:intro:general-dichotomy} is also satisfied by the theory $\PMP_{G/H}$
This yields \autoref{thm:intro:main}.

\subsection*{Acknowledgments}
AI assistants (ChatGPT and Gemini) were used primarily for proofreading and, in that context, contributed minor textual and mathematical corrections.
They were also used for literature searches and stylistic editing.
The mathematical ideas and arguments of the paper were developed by the author.

\section{Preliminaries: convex sets and affine maps}

\subsection{Compact convex sets}
\label{sec:background-compact-convex}
For more details on the following material, see for instance \cite[\textsection1]{BITaffine} and the references therein.

A \emph{compact convex set} is a compact convex subset of a locally convex, Hausdorff topological vector space.
A map $\pi\colon X\to Y$ between convex subsets of two vector spaces is \emph{affine} if $\pi(\lambda x + (1-\lambda)y) = \lambda \pi(x) + (1-\lambda)\pi(y)$ for every $x,y\in X$ and $\lambda\in [0,1]$.
Given a compact convex set $X$, we denote by $\cA(X)$ the space of affine, continuous real-valued functions on $X$.
With its natural order and the constant function $1$, $\cA(X)$ forms an \emph{order unit space}, i.e., an Archimedean ordered vector space with an order unit.
Given an order unit space $A$, we denote by $S(A)\subseteq A^*$ the state space of $A$, i.e., the compact convex set of positive, unital linear functionals on $A$.

The functorial correspondences
\begin{equation*}
X\mapsto \cA(X),\quad  S(A)\mapsfrom A
\end{equation*}
form a duality, due to Kadison, between the category of compact convex sets with continuous affine functions, and the category of complete order unit spaces with positive, unital linear maps.
If $\pi\colon X\to Y$ is an affine continuous function between compact convex sets, we denote by $\pi^{*\aff}\colon \cA(Y)\to \cA(X)$, $f\mapsto f\circ\pi$ the corresponding dual map.

For a compact Hausdorff space $X$, we denote by $C(X)$ the space of continuous real-valued functions on $X$.
We may see it as an \emph{order unit vector lattice}, i.e., an order unit space where the order forms a lattice.
If $A$ is an order unit vector lattice, the space $\Spec(A)\subseteq S(A)$ of lattice-preserving, unital linear maps is a compact Hausdorff space.
The functorial correspondences
\begin{equation*}
X\mapsto C(X),\quad  \Spec(A)\mapsfrom A
\end{equation*}
form a duality between the category of compact Hausdorff spaces and the category of complete order unit vector lattices (a variant of Gelfand duality).
If $\pi\colon X\to Y$ is a continuous map between compact Hausdorff spaces, we denote by $\pi^*\colon C(Y)\to C(X)$, $f\mapsto f\circ\pi$ the dual map.

If $X$ is a compact Hausdorff space, we may identify the state space $S(C(X))$ with $\cM(X)$, the compact convex set of Borel regular probability measures on $X$.
Dually:
\begin{equation*}
C(X) \cong \cA(\cM(X)),\quad f\mapsto f',
\end{equation*}
where $f'(\mu) = \mu(f)$ for $f\in C(X)$ and $\mu\in\cM(X)$.
We also identify $X$ with the subset of $\cM(X)$ consisting of the Dirac measures.

If $\pi\colon X\to Y$ is a continuous function between compact Hausdorff spaces, the Kadison dual of its Gelfand dual is the pushforward map
\begin{equation*}
\pi_*\colon \cM(X)\to \cM(Y),\ \mu \mapsto \mu\circ \pi^*.
\end{equation*}

If $X$ is a compact convex set, the Kadison dual of the embedding $\cA(X)\subseteq C(X)$ is the barycenter map
\begin{equation*}
R\colon \cM(X) \to X.
\end{equation*}
It is characterized by the condition: $f(R(\mu)) = \mu(f)$ for every $f\in\cA(X)$.

A \emph{face} of a compact convex set $X$ is a convex subset $F\subseteq X$ such that whenever $\lambda x + (1-\lambda)y \in F$ for some $x,y\in X$ and $0<\lambda<1$, we have $x,y\in F$.
A face $F$ is \emph{exposed} if $F = \{x\in X : \varphi(x) = \min\varphi\}$ for some $\varphi\in\cA(X)$.
A point $x\in X$ is \emph{extreme} if $\{x\}$ is a face.
The set of extreme points of $X$ will be denoted by $\cE(X)$.

\subsection{Choquet simplices}
Let $X$ be a compact convex set.
We denote by $\widetilde{X} \subseteq \cA(X)^*$ the cone of positive linear functionals on $\cA(X)$.
By Kadison duality, the elements of $\widetilde{X}$ are precisely the linear functionals of the form
\begin{equation}\label{eq:widetilde-X}
\lambda\hat{x}\colon \cA(X)\to\bR,\quad \varphi \mapsto \lambda\varphi(x)
\end{equation}
for $\lambda\in\bR_{\geq 0}$ and $x\in X$.
On the other hand, using the Hahn Decomposition Theorem for signed measures $\mu\in \cM^{\textrm{sg}}(X)\cong C(X)^*$, and the barycenter map, one sees that $\cA(X)^* = \widetilde{X} - \widetilde{X}$.

We consider the order induced on the dual space $\cA(X)^*$ by the cone $\widetilde{X}$:
\begin{equation*}
s\leq_{\widetilde{X}} t\iff t-s\in \widetilde{X}.
\end{equation*}
A compact convex set $X$ is a \emph{Choquet simplex} if $\cA(X)^*$ is a vector lattice in this order.
Equivalently, if every pair of elements of $\widetilde{X}$ has a least upper bound (see \cite[p.~52]{Phelps2001}).

A measure $\mu\in \cM(X)$ is greater than another measure $\nu$ for the \emph{Choquet order} on $\cM(X)$ if $\nu(f)\leq \mu(f)$ for every continuous, convex function $f\colon X\to\bR$.
The maximal elements for this order are called \emph{boundary measures}, and are in a sense concentrated on the extreme points of $X$.
We will denote by $\partial\cM(X)\subseteq \cM(X)$ the set of all boundary measures on $X$.
Given a point $x\in X$, we let $\cM_x(X)$ denote the set of $\mu\in\cM(X)$ with $R(\mu) = x$, and $\partial\cM_x(X) = \partial\cM(X) \cap \cM_x(X)$.
The set $\partial\cM_x(X)$ is always non-empty, and $X$ is a Choquet simplex precisely when $\partial\cM_x(X)$ is a singleton for every $x\in X$.

A compact convex set $X$ is a \emph{Bauer simplex} if the order unit space $\cA(X)$ is a vector lattice.
Equivalently, if $X$ is a Choquet simplex and the set $\cE(X)$ of its extreme points is closed.
If $Z$ is a compact Hausdorff space, then the compact convex set $\cM(Z)$ is a Bauer simplex, and conversely, every Bauer simplex $X$ is affinely homeomorphic to $\cM(\cE(X))$.
The \emph{Poulsen simplex} is, up to affine homeomorphism, the unique metrizable Choquet simplex $X$ with at least two elements such that $\cl{\cE(X)} = X$.

The Choquet simplices we will be mostly interested in are sets of invariant measures for group actions on compact Hausdorff spaces.
The fact that these compact convex sets are indeed Choquet simplices is a classical fact (see \cite[Prop.~12.3]{Phelps2001}).
However, we provide a self-contained argument which shows that this is an instance of a more general fact.

Let $X$ be a compact convex set and let $Y\subseteq X$ be a closed convex subset.
Using the description in \autoref{eq:widetilde-X}, we can see $\widetilde{Y}$ as a subset of $\widetilde{X}$.
Following standard notation, we will denote by $\widetilde{Y}-\widetilde{Y}$ the linear span of $\widetilde{Y}$ inside $\cA(X)^*$.
This is the (injective) image of the canonical map $\cA(Y)^* \to \cA(X)^*$, i.e., the linear dual of the restriction map $\cA(X)\to \cA(Y)$.

Given a Choquet simplex $X$ and an element $s\in \cA(X)^*$, we write $s^+ = s\vee 0$ for the \emph{positive part} of $s$.

\begin{lem}\label{lem:criterion-Choquet-simplex}
Let $X$ be a Choquet simplex, and let $Y\subseteq X$ be a non-empty, closed convex subset.
Assume that for every $s\in\widetilde{Y}-\widetilde{Y}\subseteq \cA(X)^*$, the positive part $s^+$ belongs to $\widetilde{Y}$.
Then $Y$ is a Choquet simplex.
\end{lem}
\begin{proof}
The assumption implies that for every pair of elements $s,t\in\widetilde{Y}$, the least upper bound $s\vee t$ belongs to $\widetilde{Y}$, since $s\vee t = (s-t)^+ +t$.

On the other hand, we claim that the order relation $\leq_{\widetilde{Y}}$ on $\widetilde{Y}$ (induced from $\cA(Y)^*$) is the trace of the order relation $\leq_{\widetilde{X}}$ on $\widetilde{X}$.
Indeed, if $s,t\in\widetilde{Y}$ satisfy $s\leq_{\widetilde{Y}} t$, then $t-s\in \widetilde{Y}\subseteq \widetilde{X}$, and therefore $s\leq_{\widetilde{X}} t$.
Conversely, if $s\leq_{\widetilde{X}} t$, then $t-s = (t-s)^+\in \widetilde{Y}$ by hypothesis, and therefore $s\leq_{\widetilde{Y}} t$.

It follows that if $s,t\in\widetilde{Y}$, then $s\vee t$ as computed in $\widetilde{X}$ is the least upper bound of $s$ and $t$ in $\widetilde{Y}$ with respect to $\leq_{\widetilde{Y}}$.
Therefore, every pair of elements of $\widetilde{Y}$ has a least upper bound in $\widetilde{Y}$, which is enough to conclude.
\end{proof}

Now let $G\actson X$ be a left action of a group $G$ on a compact convex set $X$ by affine homeomorphisms.
We denote by
\begin{equation*}
X_\inv = \{x\in X : gx = x\ \forall g\in G\}
\end{equation*}
the closed convex set of invariant points of the action.

\begin{prop}
\label{prop:invariant-simplex}
Let $X$ be a Choquet simplex, and let $G\actson X$ be a group action by affine homeomorphisms.
If $X_\inv$ is non-empty, then it is a Choquet simplex.
\end{prop}
\begin{proof}
The action $G\actson X$ extends canonically to an action on the positive cone $\widetilde{X}$, and on the whole dual space $\cA(X)^*$.
Let us denote $Y = X_{\inv}$.
From the descriptions of $\widetilde{X}$ and of $\widetilde{Y}$ given in \autoref{eq:widetilde-X}, we see that $\widetilde{Y}$ consists precisely of the $G$-invariant elements of $\widetilde{X}$.

We use the criterion of \autoref{lem:criterion-Choquet-simplex}.
Let $s\in \widetilde{Y}-\widetilde{Y}\subseteq \cA(X)^*$.
Then for every $g\in G$ we have $g(s^+) \geq gs = s$ and $g(s^+) \geq 0$, so $g(s^+) \geq s^+$.
By considering the corresponding inequality for $g^{-1}$ we have also $g(s^+) \leq s^+$.
That is, $s^+$ is $G$-invariant and therefore $s^+\in \widetilde{Y}$, as desired.
\end{proof}

Every action $G\actson X$ by homeomorphisms on a compact Hausdorff space induces an action $G\actson \cM(X)$ by affine homeomorphisms on the compact convex set of Borel regular probability measures on $X$, given by the corresponding pushforward maps.
Following the notation used in the introduction, we write
\begin{equation*}
\cM_\inv(X) = \big(\cM(X)\big)_\inv = \{\mu\in \cM(X) : g_*\mu = \mu\ \forall g\in G\}
\end{equation*}
for the compact convex set of $G$-invariant Borel regular probability measures on $X$.
We recall that an invariant measure $\mu\in \cM_\inv(X)$ is \emph{ergodic} if the only measurable sets $A\subseteq X$ such that $\mu(A\triangle gA) = 0$ for all $g\in G$ are those with $\mu(A)=0$ or $\mu(A)=1$.

\begin{prop}
\label{prop:cMinv(X)-properties}
Let $G\actson X$ be a group action by homeomorphisms on a compact Hausdorff space, admitting invariant probability measures.
The following hold:
\begin{enumerate}
\item\label{i:cMinv(X)-properties:simplex} The compact convex set $\cM_\inv(X)$ is a Choquet simplex.
\item\label{i:cMinv(X)-properties:ergodic} The extreme points of $\cM_\inv(X)$ are precisely the ergodic measures.
\item\label{i:cMinv(X)-properties:face-pres} Let $G\actson Y$ be another $G$-action by homeomorphisms on a compact Hausdorff space and let $\pi\colon X\to Y$ be a continuous $G$-equivariant map.
Then the pushforward map restricted to the invariant measures,
\begin{equation*}
\pi_{*\inv}\colon \cM_\inv(X)\to \cM_\inv(Y),
\end{equation*}
preserves the extreme points, i.e., $\pi_{*\inv}\big(\cE\big(\cM_\inv(X)\big)\big) \subseteq \cE\big(\cM_\inv(Y)\big)$.

\noindent Moreover, $\pi_{*\inv}$ is face-preserving, i.e., it sends faces to faces.

\noindent (Face-preserving maps will be discussed in detail in \autoref{sec:appendix:open-and-face-preserving}.)
\end{enumerate}
\end{prop}
\begin{proof}
Item \autoref{i:cMinv(X)-properties:simplex} is a particular case of \autoref{prop:invariant-simplex}.
For \autoref{i:cMinv(X)-properties:ergodic}, see \cite[Prop.~12.4]{Phelps2001} or \cite[Thm.~4.2]{Glasner2003}.

The first assertion of \autoref{i:cMinv(X)-properties:face-pres} is a consequence of \autoref{i:cMinv(X)-properties:ergodic}, since the pushforward map sends ergodic measures to ergodic measures.
In the case where $X$ is metrizable, face-preservation is then automatic: see \autoref{prop:ext-preserving-simplex}.
In the general case, face-preservation can be shown exactly as in the proof of \autoref{prop:M(X)-convex-lifting}, by observing that the Radon--Nikodym derivative between two $G$-invariant measures must be $G$-invariant.
\end{proof}

\section{Preliminaries: Property (T) and Bauer simplices}
\label{sec:property-T}

We recall that a topological group $G$ has \emph{Property (T)} if there exist a compact subset $Q\subseteq G$ and $\epsilon>0$ such that for every continuous unitary representation $G\actson \cH$ without $G$-invariant unit vectors, we have in fact
\begin{equation*}
\max_{g\in Q}\|v - gv\| \geq \epsilon
\end{equation*}
for every unit vector $v\in \cH$.
In that case, the pair $(Q,\epsilon)$ is a \emph{Kazhdan pair} for $G$.

Given a group action $G\actson X$ on a compact Hausdorff space, an element $g\in G$ and a function $f\in C(X)$, we denote by $gf$ the function $x\mapsto f\big(g^{-1}(x)\big)$.

\begin{lem}
\label{lem:prop-T-mu-ergodic-continuous-f}
Let $G$ be a topological group with Property (T), with Kazhdan pair $(Q,\epsilon)$.
Let $G\actson X$ be a continuous action on a compact Hausdorff space.
Then for every invariant probability measure $\mu\in\cM_\inv(X)$, the following are equivalent:
\begin{enumerate}
\item The measure $\mu$ is ergodic.
\item For every $f\in C(X)$ with $\mu(f)=0$ and $\|f\|_{L^2(X,\mu)} = 1$, we have:
\begin{equation*}
\max_{g\in Q}\|f-gf\|_{L^2(X,\mu)} \geq\epsilon.
\end{equation*}
\end{enumerate}
\end{lem}
\begin{proof}
The measure $\mu$ is ergodic if and only if the induced unitary representation on $L^2_0(X,\mu) = \{f\in L^2(X,\mu) : \mu(f)=0 \}$ has no $G$-invariant unit vectors.
One of the implications of the statement then follows immediately from Property (T), and the converse follows from the density of $\{f\in C(X) : \mu(f)=0\}$ in $L_0^2(X,\mu)$.
\end{proof}

\begin{lem}
\label{lem:max_Q-continuity}
Let $G$ be a topological group and let $Q\subseteq G$ be a compact subset.
Then for every continuous action $G\actson X$ on a compact Hausdorff space and every continuous function $f\in C(X)$, the map
\begin{equation*}
\Psi_{Q,f}\colon \cM(X)\to \bR,\quad \mu\mapsto \max_{g\in Q}\|f-gf\|_{L^2(X,\mu)}
\end{equation*}
is continuous.
\end{lem}
\begin{proof}
Since $Q$ is compact, for every $f\in C(X)$ and $\epsilon>0$ we can find a finite subset $R\subseteq Q$ such that for every $g\in Q$ there is $h\in R$ with $\|gf-hf\|_\infty <\epsilon$.
Therefore $\big|\Psi_{Q,f}(\mu) - \Psi_{R,f}(\mu)\big| \leq \epsilon$, and thus
\begin{equation*}
\big|\Psi_{Q,f}(\mu) - \Psi_{Q,f}(\mu')\big| \leq \big|\Psi_{R,f}(\mu) - \Psi_{R,f}(\mu')\big| + 2\epsilon
\end{equation*}
for every $\mu,\mu'\in\cM(X)$.
The continuity of $\Psi_{Q,f}$ then follows from that of $\Psi_{R,f}$, which is clear.
\end{proof}

The following proposition was proved by Glasner and Weiss \cite{Glasner1997} in the setting of locally compact, second countable groups, but holds in all generality, as observed by Evans and Tsankov \cite{EvaTsa} (see the discussion after their Corollary 1.2).
The argument we give here avoids the use of the ergodic decomposition.

\begin{prop}
\label{prop:prop-T-implies-Bauer}
Let $G$ be a topological group with Property (T).
Then for every continuous action $G\actson X$ on a compact Hausdorff space, if $\cM_\inv(X)$ is non-empty, then it is a Bauer simplex.
\end{prop}
\begin{proof}
By \autoref{prop:cMinv(X)-properties}, it is enough to show that the set of ergodic measures is closed.
In turn, this is a direct consequence of \autoref{lem:prop-T-mu-ergodic-continuous-f} and \autoref{lem:max_Q-continuity}.
\end{proof}

\begin{rmk}
\label{rmk:oligomorphic-groups-exchangeability-theory}
As we mentioned in the introduction, the simplices $\cM_\inv(K^\cS)$ arising in concrete examples from exchangeability theory tend to be Bauer simplices.
This can be explained using \autoref{prop:prop-T-implies-Bauer} (and, more precisely, \cite[Cor.~1.2]{EvaTsa}).
Indeed, the permutation groups $G\actson\cS$ of interest in that setting are often \emph{oligomorphic}, meaning that the orbit spaces $\cS^n/G$ are finite for all $n\in\bN$.
In that case, the closure $\cl{G}\leq \Sym(\cS)$ within the full symmetric group of $\cS$ (with the topology of pointwise convergence) is a \emph{Polish oligomorphic group}, and these groups have Kazhdan's Property (T) by the main result of Evans and Tsankov \cite{EvaTsa} (see also \cite{Iba2021}).
Since the $G$-invariant and the $\cl{G}$-invariant measures on $K^\cS$ coincide, one has that in the oligomorphic exchangeability contexts, $\cM_\inv(K^\cS)$ is always a Bauer simplex.

This argument applies to the classical setting of \emph{hypergraph-exchangeability} (where $\cS = \binom{\bN}{k}$ and $G = \Sym_0(\bN)$), but also, for instance, to the case of \emph{affine-exchangeable} measures established recently in \cite[Thm.~1.6]{Candela2023} (where $\cS = \bF_2^{\oplus\bN}$ and $G = \operatorname{AGL}_\infty(\bF_2) = \bF_2^{\oplus\bN}\rtimes\operatorname{GL}(\bF_2^{\oplus\bN})$).
There, the Bauer property is obtained as a corollary of a finer representation theorem.
However, if one is interested in this property alone, it suffices to observe that the affine space $\bF_2^{\oplus\bN}$ is an \emph{$\aleph_0$-categorical} countable structure, and hence the action of its automorphism group, $\operatorname{AGL}_\infty(\bF_2)$, is oligomorphic.

The same argument also yields the following broad generalization, as well as other similar statements: for every finite field $\bF$ and every $k\in\bN$, letting $\operatorname{Graff}(k,A)$ denote the space of $k$-dimensional affine subspaces of $A=\bF^{\oplus\bN}$, the $\operatorname{AGL}_\infty(\bF)$-invariant measures on $K^{\operatorname{Graff}(k,A)}$ form a Bauer simplex.
\end{rmk}

\section{Relative Property (T) and Bauer simplices}
\label{sec:permutation-groups-relative-T}

\begin{dfn}
\label{dfn:relative-(T)}
Let $G$ be a topological group and $H\leq G$ be a closed subgroup.
We say the pair $(G,H)$ has \emph{Property (T)}, or that $H$ has \emph{relative Property (T)} in $G$, if for every $\delta>0$ there exist a compact subset $Q\subseteq G$ and $\epsilon>0$ such that for every continuous unitary representation $G\actson \cH$, if
\begin{equation*}
\max_{g\in Q}\|v - gv\| < \epsilon
\end{equation*}
for some unit vector $v\in \cH$, then there is an $H$-invariant unit vector $w\in\cH$ with 
\begin{equation*}
\|v-w\|\leq\delta.
\end{equation*}
\end{dfn}

\begin{rmk}
Let $G$ be a topological group, $H\leq G$ a closed subgroup. Then:
\begin{enumerate}
\item If $H$ has Property (T), the pair $(G,H)$ has Property (T).
\item $G$ has Property (T) if and only if the pair $(G,G)$ has Property (T).
\end{enumerate}
\end{rmk}

\begin{rmk}
\label{rmk:relative-(T)-definitions}
In the context of countable groups, Property (T) for a pair $(G,H)$ is usually defined by the weaker condition that whenever there is a non-zero vector $v\in\cH$ with $\max_{g\in Q}\|v - gv\| < \epsilon$, then there is a non-zero $H$-invariant vector (not necessarily close to $v$).
The two definitions are known to be equivalent for locally compact, $\sigma$-compact groups: see \cite{Jolissaint2005}.
I do not know whether the equivalence holds for more general classes of groups, and in particular for closed permutation groups.
In fact, the tools used in \cite{Jolissaint2005} to prove this equivalence are closely related to the representation theoretic facts used in the proof of the main theorem of Glasner--Weiss \cite{Glasner1997}, and which are lacking in the context of non-Archimedean Polish groups, as I mentioned in the introduction.

I believe the \emph{correct} definition in the general setting is the one given above, and this is indeed the definition that we will need.
\end{rmk}

In what follows, given an action $G\actson M$ of a group by automorphisms of a measure algebra, and a subgroup $H\leq G$, we denote by
\begin{equation*}
\Fix_H(M) = \{b\in M: gb=b\ \forall g\in H\}
\end{equation*}
the subalgebra of $H$-invariant elements of $M$.

\begin{lem}
\label{lem:Connes-Weiss-variant}
Let $G$ be a Polish group and $H\leq G$ be a closed subgroup.
Suppose the pair $(G,H)$ has Property (T) but $G$ does not.
Then for every finite subset $F\subseteq G$ and every $\eta>0$ one may find a continuous, ergodic action $G\actson M$ by automorphisms on the measure algebra of a probability space, and an element $a\in \Fix_H(M)$, such that
\begin{equation*}
\max_{g\in F}\mu(a\triangle ga)\leq \eta\quad \text{and}\quad 1/8\leq \mu(a)\leq 7/8.
\end{equation*}
\end{lem}

The proof is an adaptation of the argument in Connes--Weiss \cite{Connes1980}.
In their argument they use that for countable groups, Property (T) is equivalent to the trivial representation being isolated in the space of irreducible representations.
I thank Todor Tsankov for pointing out to me that this is not really necessary.

\begin{proof}
We fix $0<\eta<1$ and a non-empty finite subset $F\subseteq G$, and we let
\begin{equation*}
\delta = \frac{1}{3}\left(1-\cos\frac{\pi\eta}{3|F|}\right)^{1/2} > 0.
\end{equation*}
Since the pair $(G,H)$ has Property (T), we may find a compact set $Q\subseteq G$ and $\epsilon>0$ satisfying the conditions of \autoref{dfn:relative-(T)} with respect to $\delta$.
Since $G$ does not have Property (T), there is a continuous unitary representation $G\actson\cK$ on a separable Hilbert space $\cK$ without non-zero invariant vectors such that for some unit vector $u\in\cK$, $\|u - gu\|<\min\{\delta,\epsilon\}$ for every $g\in Q\cup F$.
By choice of $Q$ and $\epsilon$, there is an $H$-invariant unit vector $w\in \cK$ with $\|u-w\| \leq \delta$.
In particular, $\|w - gw\|<3\delta$ for every $g\in F$.

Let $\cK_0$ be the closed subspace of $\cK$ generated by its finite-dimensional subrepresentations, and let $\cK_1$ be its orthogonal complement. Let $w = w_0 + w_1$ be the corresponding orthogonal decomposition of $w$. For the rest of the proof we fix $k\in \{0,1\}$ such that
\begin{equation*}
\|w_k\| = \max\{\|w_0\|,\|w_1\|\}.
\end{equation*}
We will consider the unitary representation $G\actson \cH$ where $\cH\subseteq\cK_k$ is the cyclic subrepresentation generated by the unit vector $v = w_k/\|w_k\|\in \cH$.
We observe that $v$ is fixed by the subgroup $H$ and satisfies $\|v - gv\|<3\sqrt{2}\delta$ for every $g\in F$.
In other words, for every $g\in F$,
\begin{equation}
\label{eq:CW-cos}
\arccos\big(\Re\langle v,gv\rangle\big) < \frac{\pi\eta}{3|F|}.
\end{equation}

Let $\Gamma\leq G$ be a countable dense subgroup of $G$.
Consider the positive definite function $\varphi\colon \Gamma\to\bR$ given by
\begin{equation*}
\varphi(g) = \Re\langle v,gv\rangle,
\end{equation*}
and the corresponding Gaussian space $(\Omega,\mu) = (\bR^\Gamma,\mu_\varphi)$ (see, for instance, \cite[Appendix~C]{KechrisGlobal}).
In particular, the projection maps $X_g\colon \bR^\Gamma\to \bR$, $X_g(\omega) = \omega_g$ are centered Gaussian random variables, with covariance given by
\begin{equation*}
\bE(X_gX_h) = \varphi(g^{-1}h) = \Re\langle gv,hv\rangle.
\end{equation*}
The group $\Gamma$ acts on $\bR^\Gamma$ by permutation of the coordinates, and this action is measure-preserving.
We therefore have an action of $\Gamma$ by automorphisms of the measure algebra $M_\varphi = \MALG(\Omega,\mu)$, i.e., a group homomorphism $\alpha_\Gamma\colon \Gamma\to \Aut(M_\varphi)$.
On the other hand, there is a topological group embedding
\begin{equation*}
\rho\colon \cU(\cH) \to \Aut(M_\varphi)
\end{equation*}
that takes the representation $\Gamma\to \cU(\cH)$ into $\alpha_\Gamma$ (see \cite[Appendix~E]{KechrisGlobal}).
Composing $\rho$ with the original representation $\pi\colon G\to\cU(\cH)$, we also get a continuous action $G\actson M_\varphi$ by automorphisms, which restricts to $\alpha_\Gamma$.

Let $b = \llbracket X_{1_G}\geq 0\rrbracket = \{\omega\in \Omega : X_{1_G}(\omega) \geq 0\}$, which we may see as an element of $M_\varphi$.
Since $X_{1_G}$ is centered, $\mu(b) = 1/2$.
For every $g\in \Gamma$, a standard computation (see \cite[p.~234]{Glasner2003}) gives
\begin{equation*}
\mu(b\triangle gb) = \mu\big(\llbracket X_{1_G}\geq 0\ \&\ X_g < 0\rrbracket\cup \llbracket X_{1_G}< 0\ \&\ X_g\geq 0\rrbracket\big) = \frac{1}{\pi}\arccos(\varphi(g)).
\end{equation*}
By density and continuity, the same holds for every $g\in G$.
Therefore, by \autoref{eq:CW-cos},
\begin{equation}
\label{eq:max_Fmu(b-triangle-gb)}
b\in\Fix_H(M_\varphi)\quad \text{and}\quad \max_{g\in F}\mu(b\triangle gb) < \frac{\eta}{3|F|}.
\end{equation}

We now distinguish two cases, according to the value of $k\in\{0,1\}$ set at the beginning of the proof.

If $k=1$, the unitary representation $\Gamma\actson\cH$ has no finite-dimensional subrepresentations.
This implies that the Gaussian system $\Gamma\actson M_\varphi$, as well as the induced action $G\actson M_\varphi$, are ergodic (see \cite[Thm.~3.59]{Glasner2003}).
We may then take $M = M_\varphi$ and $a = b$, and we are done.

If $k=0$, then $\cH$ is a direct sum of finite-dimensional representations.
In that case, the closure $K = \cl{\pi(G)}$ of the image of $G$ inside the unitary group $\cU(\cH)$ is a compact metrizable group.
Let $\nu$ be its normalized Haar measure.
Since the representation $G\actson \cH$ does not have non-zero invariant vectors, the vector $\int_K kv\,d\nu(k)\in\cH$ must be zero.
Therefore, $\int_K\langle v,kv\rangle\,d\nu(k) = 0$, which implies that there is $g_0\in G$ with $\varphi(g_0)\leq 0$.
In turn, this ensures that
\begin{equation}
\label{eq:mu(b-triangle-g_0b)}
\mu(b\triangle g_0b)\geq 1/2.
\end{equation}
On the other hand, by the classical theory of probability measure-preserving (pmp) systems of locally compact second countable groups (see \cite{Mackey1962,Ramsay1971}; see also \cite{Glasner2005}), the action $K\actson M_\varphi$ induced by $\rho$ can be realized as a Borel pmp system $K\actson (\Omega',\mu')$ on a standard probability space.
Let $(\mu_z : z\in\Omega')$ denote the ergodic decomposition of this system.
Let $B\subseteq \Omega'$ be a measurable set representing $b\in M_\varphi$.
From \autoref{eq:mu(b-triangle-g_0b)} we have
\begin{equation*}
\mu'\big(\{z\in \Omega' : \mu_z(B\triangle \pi(g_0)B)\leq 1/4\}\big) \leq 2/3,
\end{equation*}
and from \autoref{eq:max_Fmu(b-triangle-gb)} and Markov's inequality,
\begin{equation*}
\mu'\Big(\big\{z\in \Omega' : \sum_{g\in F}\mu_z(B\triangle \pi(g)B)\geq \eta\big\}\Big) < 1/3.
\end{equation*}
In addition, for every $g\in H$, $\mu_z(B\triangle \pi(g)B) = 0$ for almost every $z\in \Omega'$.
Let $H_0\subseteq H$ be a countable dense subset.
It follows that for some $z_0\in \Omega'$, the ergodic component $\mu_{z_0}$ satisfies $\mu_{z_0}(B\triangle \pi(g_0)B) > 1/4$ as well as $\mu_{z_0}(B\triangle \pi(g)B) < \eta$ for every $g\in F$ and $\mu_{z_0}(B\triangle \pi(h)B) = 0$ for every $h\in H_0$.
Since the Borel pmp action $K\actson (\Omega',\mu_{z_0})$ induces a continuous action on the corresponding measure algebra, we have in fact $\mu_{z_0}(B\triangle \pi(h)B) = 0$ for every $h\in H$.
Finally, we observe that the condition $\mu_{z_0}(B\triangle \pi(g_0)B) > 1/4$ implies $1/8 \leq \mu_{z_0}(B) \leq 7/8$.
We may therefore take
\begin{equation*}
M = \MALG(\Omega',\mu_{z_0}),
\end{equation*}
and let $a\in M$ be the class of $B$.
To conclude, it is enough to note that since the action $K\actson M$ is continuous and ergodic, then so is the induced action $G\actson M$.
\end{proof}

The following result generalizes \cite[Prop.~2]{Glasner1997}.
The proof strategy is essentially the same.

\begin{prop}
\label{prop:prop-T-iff-Bauer}
Let $G$ be a Polish group with an open subgroup $H\leq G$.
In particular, $H$ is closed and $G$ acts by permutations of the countable set $G/H$.
Assume the pair $(G,H)$ has Property (T).
Then $G$ has Property (T) if and only if $\cM_\inv\big(2^{G/H}\big)$ is a Bauer simplex.
\end{prop}
\begin{proof}
One implication is just \autoref{prop:prop-T-implies-Bauer}.
Note that $\cM_\inv\big(2^{G/H})$ is never empty, since it contains the product measures.

For the converse, suppose $G$ does not have Property (T).
Let $(g_kH : k\in\N)$ be an enumeration of $G/H$.
By \autoref{lem:Connes-Weiss-variant}, for every $k\in\N$ we may find a continuous ergodic action $G\actson (M_k,\mu_k)$ by automorphisms on a probability algebra and an element $a_k\in \Fix_H(M_k)$ such that $\mu_k(a_k\triangle g_i a_k)\leq 1/k$ for every $i<k$, and $\mu_k(a_k)\in [1/8,7/8]$.
Up to passing to a subsequence, we may assume that the measures $\mu_k(a_k)$ converge to some $\lambda\in [1/8,7/8]$.

Given $n\in \N$ and a finite sequence $z\in 2^n$ of zeros and ones, let $U_z\subseteq 2^{G/H}$ denote the basic clopen set given by
\begin{equation*}
U_z = \{x\in 2^{G/H} : \forall i<n,\, x_{g_iH} = z_i\}.
\end{equation*}
Then for each $k\in \N$, let $\nu_k\in \cM_\inv\big(2^{G/H}\big)$ be the probability measure determined by the conditions
\begin{equation*}
\nu_k(U_z) = \mu_k\Big(\bigcap_{i<n}g_i a_k^{z_i}\Big)
\end{equation*}
for every $z\in 2^n$, $n\in \N$, where $a_k^1$ denotes $a_k$ and $a_k^0$ denotes its complement in $M_k$.
The measure $\nu_k$ is indeed $G$-invariant because $a_k$ is $H$-invariant.
(If the action $G\actson M_k$ were represented by a Borel pmp system $G\actson (\Omega,\mu_k)$, the measure $\nu_k$ would be the pushforward of $\mu_k$ under the equivariant map $\theta_k\colon \Omega\to 2^{G/H}$, $(\theta_k(\omega))_{g_iH} = \chi_{g_ia_k}(\omega)$.)
Since the action $G\actson M_k$ is ergodic, so are the measures $\nu_k$.

Finally, let $\bf0$ and $\bf1$ denote the two constant sequences of zeros and ones in $2^{G/H}$, and let $\delta_{\bf0}$ and $\delta_{\bf1}$ be the corresponding Dirac measures.
We claim that the ergodic measures $\nu_k$ converge to the non-ergodic measure $\nu = \lambda\delta_{\bf1} + (1-\lambda)\delta_{\bf0}$.
The latter is characterized by the property that $\nu(U_{{\bf1}_n}) = \lambda$ and $\nu(U_{{\bf0}_n}) = 1-\lambda$ for every $n\in\N$, where ${\bf1}_n\in 2^n$ denotes the constant sequence of ones, and ${\bf0}_n\in 2^n$ denotes the constant sequence of zeros.
Now, since $\mu_k(a_k\triangle g_i a_k)\to 0$ for every fixed $i\in \N$ and since $\mu_k(a_k)\to \lambda$, we have, for every $n\in\N$,
\begin{equation*}
\lim_k \nu_k(U_{{\bf1}_n}) = \lim_k\mu_k\Big(\bigcap_{i<n}g_ia_k\Big) = \lim_k\mu_k(a_k) = \lambda.
\end{equation*}
Similarly, $\lim_k\nu_k(U_{{\bf0}_n}) = 1-\lambda$.
This proves our claim, and lets us conclude that the simplex $\cM_\inv\big(2^{G/H}\big)$ is not Bauer.
\end{proof}

We end this section with a lemma that we will need later.

\begin{lem}
\label{lem:relative-(T)}
Let $G$ be a topological group, let $H$ be a closed subgroup with relative Property (T) in $G$, and let $\delta>0$ be arbitrary.
Let $Q\subseteq G$ and $\epsilon>0$ satisfy the conditions of \autoref{dfn:relative-(T)} with respect to $\delta$.
Then the following hold:
\begin{enumerate}
\item\label{i:relative-(T)-pmp}
For every continuous action $G\actson M$ by automorphisms on the measure algebra of a probability space, whenever 
\begin{equation}
\label{eq:relative-(T)-pmp}
\max_{g\in Q}\mu(a\triangle ga) < \epsilon^2/4
\end{equation}
for some $a\in M$, there is $b\in \Fix_H(M)$ such that
\begin{equation*}
\mu(a\triangle b) \leq \delta^2.
\end{equation*}

\item\label{i:relative-(T)-measures}
Assume moreover that $H$ is an open subgroup of $G$.
Then for any set $x$ and every clopen set $A\subseteq (2^x)^{G/H}$ there is a finite set $F\subseteq Q$ such that for every $\mu\in \cM_\inv\big((2^x)^{G/H}\big)$ we have:
\begin{equation}
\label{eq:relative-(T)-measures}
\inf_{[B]_\mu\in\Fix_H(M_\mu)} \mu(A\triangle B) \leq \frac{4}{\epsilon^2}\max_{g\in F}\mu(A\triangle gA) + \delta^2,
\end{equation}
where $M_\mu = \MALG\big((2^x)^{G/H},\mu\big)$.
\end{enumerate}
\end{lem}
\begin{proof}
We first prove \autoref{i:relative-(T)-pmp}.
Let $M = \MALG(\Omega,\mu)$ be the measure algebra of a probability space, and let $G\actson M$ be a continuous action by automorphisms.
We consider the Koopman representation $G\actson \cH = L^2(\Omega,\mu)$.
Given $a\in M$, we let $v_a = 2\chi_a -1 = \chi_a - \chi_{\Omega\setminus a}$.
Then $\|v_a\| = 1$, and for every $g\in G$ we have
\begin{equation*}
\| v_a - gv_a\|^2 = 4\| \chi_a - g\chi_a\|^2 = 4\mu(a\triangle ga).
\end{equation*}
If \autoref{eq:relative-(T)-pmp} holds, then $\max_{g\in Q} \| v_a - gv_a\| < \epsilon$. Therefore, by choice of $Q$ and $\epsilon$, there is an $H$-invariant unit vector $w\in \cH$ such that
\begin{equation*}
\|v_a - w\| \leq \delta.
\end{equation*}
Now let $b=\{\omega\in\Omega : \Re(w)\geq 0\}$.
Clearly, $b\in \Fix_H(M)$.
On the other hand, on $a\triangle b$ we have $|v_a-w|\geq 1$.
Therefore,
\begin{equation*}
\mu(a\triangle b) \leq \int_{a\triangle b}|v_a-w|^2 d\mu \leq \|v_a-w\|^2 \leq \delta^2,
\end{equation*}
as desired.

To prove \autoref{i:relative-(T)-measures}, we assume now that $H$ is open and we let $A\subseteq (2^x)^{G/H}$ be a clopen subset.
In particular, $A$ depends only on finitely many coordinates of the index set $G/H$, say $g_iH$ for $i<n$.
Since $Q$ is compact and $\bigcap_{i<n}g_iHg_i^{-1}$ is open, there is a finite subset $F\subseteq Q$ such that for every $g\in Q$ there is $h\in F$ with $gg_iHg_i^{-1} = hg_iHg_i^{-1}$ for all $i<n$.
Therefore, for every $g\in Q$ there is $h\in F$ with $gA = hA$, and we have:
\begin{equation*}
\max_{g\in Q}\mu(A\triangle gA) = \max_{g\in F}\mu(A\triangle gA).
\end{equation*}
The inequality \autoref{eq:relative-(T)-measures} then follows directly from \autoref{i:relative-(T)-pmp}, applied to the action of $G$ on $M_\mu$.
\end{proof}

\section{Preliminaries: affine logic}

From \autoref{sec:face-preserving-theories} onward we will assume familiarity with continuous logic, and revisit many aspects of affine logic developed in \cite{BITaffine}.
In this section, we recall several basic definitions, results, and notation.

Throughout this section, $\cL$ denotes a metric language with concave continuity moduli (see \cite[\textsection2]{BITaffine}).

\subsection{Type spaces in continuous and affine logic}
For every tuple of variables $x$, we denote:
\begin{itemize}
\item $\cL^\cont_x$: the set of continuous logic $\cL$-formulas with free variables in $x$.
\item $\cL^\qf_x$: the set of quantifier-free continuous logic $\cL$-formulas in those variables (a more precise but cumbersome notation would be $\cL^{\qf,\cont}_x$).
\item $\cL^\aff_x$: the set of \emph{affine} $\cL$-formulas with variables in $x$, i.e., continuous logic formulas whose only connectives are linear combinations and constants.
\item $\cL^{\qf,\aff}_x$: the set of quantifier-free, affine $\cL$-formulas in the variables $x$.
\end{itemize}

Now let $T$ be a continuous logic $\cL$-theory, i.e., a set of conditions of the form $\varphi\leq \psi$ with $\varphi,\psi\in\cL^\cont_\emptyset$.
We may identify two continuous $\cL$-formulas in the variables $x$ if they take the same values on every $x$-tuple of models of $T$.
The corresponding quotients $\cL_x^\cont(T) = {\cL_x^\cont/\equiv_T}$ and ${\cL_x^\qf(T) = \cL_x^\qf/\equiv_T}$ are naturally endowed with a structure of order unit vector lattices (see \autoref{sec:background-compact-convex}). The dual spaces
\begin{equation*}
\tS_x^\cont(T) = \Spec\big(\cL_x^\cont(T)\big),\quad \tS_x^\qf(T) = \Spec\big(\cL_x^\qf(T)\big)
\end{equation*}
are the spaces of \emph{types} and \emph{quantifier-free types} of $T$ in the variables $x$, respectively.

Similarly, if $T$ is an affine $\cL$-theory, the quotient $\cL_x^\aff(T) = {\cL_x^\aff/\equiv_T}$, as well as ${\cL_x^{\qf,\aff}(T) = \cL_x^{\qf,\aff}/\equiv_T}$, form order unit spaces. The state spaces
\begin{equation*}
\tS_x^\aff(T) = S\big(\cL_x^\aff(T)\big),\quad \tS_x^\qf(T) = S\big(\cL_x^{\qf,\aff}(T)\big)
\end{equation*}
are the corresponding spaces of \emph{affine types} and \emph{quantifier-free (affine) types} of $T$ in the variables $x$.
The \emph{type} (affine or quantifier-free, respectively) of a tuple $a\in M^x$ in a model $M\models T$ is denoted by $\tp^\aff(a)$ or $\tp^\qf(a)$.
It is the state sending a formula $\varphi(x)$ of the appropriate kind to the value $\varphi(a)\in \bR$.

\begin{rmk}
For affine theories (which are, in particular, continuous theories), the canonical map $\Spec\big(\cL^\qf_x(T)\big) \to S\big(\cL^{\qf,\aff}_x(T)\big)$ is a homeomorphism. Indeed, a quantifier-free continuous type is determined by its atomic part, and in particular by its affine part.
This justifies the use of the same notation, $S^\qf_x(T)$, for both spaces in either context.
\end{rmk}

The order unit spaces $\cL^\aff_x(T)$, $\cL^\qf_x(T)$ are in general not complete.
Their completions, which may be identified by Kadison duality with the spaces $\cA\big(\tS^\aff_x(T)\big)$, $\cA\big(\tS^\qf_x(T)\big)$, consist precisely of the \emph{affine definable predicates} of~$T$, i.e., the uniform limits of affine formulas modulo $T$. Similarly, mutatis mutandis, in continuous logic.

The sets of \emph{extreme types} of an affine theory $T$ are denoted as follows:
\begin{equation*}
\cE_x(T) = \cE\big(\tS^\aff_x(T)\big),\quad \cE^\qf_x(T) = \cE\big(\tS^\qf_x(T)\big).
\end{equation*}
An \emph{extreme completion} of $T$ is an extreme point of $\tS^\aff_0(T)$.
More generally, a \emph{face} of $T$ is an extension of $T$ whose completions form a face of $\tS^\aff_0(T)$.

An \emph{extremal model} of $T$ is a model $M\models T$ that only realizes extreme (full) types, i.e., such that $\tp^\aff(a) \in \cE_x(T)$ for every $a\in M^x$ and every tuple of variables $x$.
These models were studied in detail in \cite{BITaffine} (see also \cite{Bagheri2024p}).
In \autoref{sec:face-preserving-theories}, we will introduce and study \emph{qf-extremal} models, which omit the non-extreme types of $\tS^\qf_x(T)$.

The \emph{affine part} of a continuous logic theory $T$, denoted by $T_\aff$, is the set of all affine conditions implied by $T$.
The following lemma relates the quantifier-free type spaces of a continuous theory to those of its affine part.
Compare with \cite[Lemma~13.3]{BITaffine}.

\begin{lem}
\label{lem:E(Taff)-S(T)-S(Taff)}
For every continuous theory $T$ and every tuple of variables $x$, we have canonical inclusions:
\begin{equation*}
\cl{\cE^\qf_x(T_\aff)} \subseteq \tS^\qf_x(T) \subseteq \tS^\qf_x(T_\aff).
\end{equation*}
\end{lem}
\begin{proof}
We have a canonical embedding of order unit spaces, $\iota^\qf_x\colon \cL^{\qf,\aff}_x(T_\aff)\to \cL^\qf_x(T)$.
As the latter is a vector lattice, the Kadison dual map has the form:
\begin{equation*}
\iota^{\qf*}_x\colon \cM\big(\tS^\qf_x(T)\big) \to \tS^\qf_x(T_\aff).
\end{equation*}
The restriction of $\iota^{\qf*}_x$ to $\tS^\qf_x(T)$ is the map taking a continuous quantifier-free type to its affine part, and it is injective because quantifier-free types are determined by their atomic parts.
We thus see it as an inclusion $\tS^\qf_x(T) \subseteq \tS^\qf_x(T_\aff)$.
Moreover, the image of this inclusion (being the restriction of $\iota^{\qf*}_x$ to the set of extreme points of its domain) must contain the extreme points of $\tS^\qf_x(T_\aff)$.
Since moreover $\tS^\qf_x(T)$ is compact, we have $\cl{\cE^\qf_x(T_\aff)} \subseteq \tS^\qf_x(T)$.
\end{proof}

Given a subset $A$ of an $\cL$-structure $M$, we denote by $\cL(A)$ the language $\cL$ expanded with constants for the elements of $A$, and we see $M$ as an $\cL(A)$-structure in the obvious way.
The \emph{elementary diagram} of $A$ in $M$ (i.e., the continuous logic theory of the $\cL(A)$-structure $M$), will be denoted by $D^\cont(A)$.
Similarly, the \emph{affine diagram} of $A$ in $M$ will be denoted by $D^\aff(A)$.
These depend of course on $M$, although the notation does not reflect it.

The type spaces associated to the diagrams $D^\cont(A)$ and $D^\aff(A)$ are the spaces of \emph{types over $A$} (or \emph{types with parameters from $A$}).
The only type spaces with parameters that we shall consider in this paper are the quantifier-free type spaces of the affine diagrams, which we may denote by
\begin{equation*}
\tS^{\qf,\aff}_x(A) = \tS^\qf_x\big(D^\aff(A)\big),\quad \cE^\qf_x(A) = \cE^\qf_x\big(D^\aff(A)\big).
\end{equation*}
The ambient structure $M$ will always be clear from the context.

An affine theory $T$ is \emph{simplicial} if all its (full) type spaces $\tS^\aff_x(T)$ are Choquet simplices.
It is a \emph{Bauer theory} if $\tS^\aff_x(T)$ is a Bauer simplex for every tuple of variables $x$, and it is a \emph{Poulsen theory} if it is simplicial and $\cE_x(T)$ is dense in $\tS^\aff_x(T)$ for every tuple $x$.
In each case, it is enough to verify the condition for finite tuples~$x$.
In \autoref{sec:qf-simplicial} and \autoref{sec:dichotomy} we will introduce quantifier-free counterparts of these definitions.

As mentioned in the introduction, an important result of \cite{BITaffine} is that every affinely complete, simplicial theory is Bauer or Poulsen.
We aim to generalize this and other results from full type spaces to quantifier-free type spaces.
Accordingly, we shift our focus from general affine theories to universal affine theories.

\subsection{Universal theories}
Let $T$ be a continuous logic $\cL$-theory.
We will use the following definitions and notation:
\begin{itemize}
\item The \emph{universal part} of $T$, denoted by $T_\forall$, is the set of conditions of the form 
\begin{equation*}
\sup_x \varphi(x) \leq 0
\end{equation*}
implied by $T$, 
where $\varphi\in \cL^\qf_x$.
\item The theory $T$ is \emph{universal} if it is equivalent to $T_\forall$.
\item The \emph{universal affine part} of $T$, denoted by $T_{\forall^\aff}$, is the set of consequences of $T$ of the form $\sup_x \varphi(x) \leq 0$, where $\varphi\in\cL^{\qf,\aff}_x$.
\end{itemize}

We may identify two theories if they are equivalent; for instance, we write $T=T_\forall$ if $T$ is universal.

The models of $T_\forall$ are precisely the substructures of models of $T$. The following affine analogue of this fact was observed early on by Bagheri and Safari (see \cite[Lemma~4.2]{Bagheri2014a}).

\begin{lem}
Let $T$ be an affine theory.
Then the models of $T_{\forall^\aff}$ are precisely the substructures of models of $T$.
\end{lem}
\begin{proof}
One implication is clear, and the converse follows exactly as in the classical or continuous setting, using Bagheri's Compactness Theorem for Affine Logic.
\end{proof}

The previous lemma may be rephrased as follows: if $T$ is an affine theory, then $T_{\forall^\aff} = T_\forall$.
Slightly more generally:

\begin{lem}
\label{lem:forall-aff}
For any continuous logic theory $T$, we have $T_{\forall^\aff} = (T_\aff)_\forall$.
\end{lem}
\begin{proof}
$T_{\forall^\aff} = (T_\aff)_{\forall^\aff} = (T_\aff)_\forall$.
\end{proof}

In particular, an affine theory $T$ is universal (as a continuous logic theory) if and only if it is equivalent to its universal affine part, $T_{\forall^\aff}$.

\begin{rmk}
In general $T_{\forall^\aff}\subseteq (T_\forall)_\aff$, but the inclusion can be strict, as is the case when $T$ is the continuous theory of two named points at distance one.
\end{rmk}

We will also employ the following notation:
\begin{itemize}
\item The \emph{universal theory} of a structure $M$ is $\Th^\cont_\forall(M) = \big(\Th^\cont(M)\big)_\forall$, that is, the set of universal conditions that hold in $M$.
\item The \emph{universal diagram} of $A\subseteq M$ is $D^\cont_\forall(A) = \big(D^\cont(A)\big)_\forall$.
\item Similarly, the \emph{universal affine theory} of $M$ is $\Th^\aff_\forall(M) = \big(\Th^\aff(M)\big)_\forall$, the set of universal affine conditions true in $M$.
\item The \emph{universal affine diagram} of $A\subseteq M$ is $D^\aff_\forall(A) = \big(D^\aff(A)\big)_\forall$.
\end{itemize}
Clearly, $\Th^\aff_\forall(M) = \big(\Th^\cont(M)\big)_{\forall^\aff}$, but possibly $\Th^\aff_\forall(M) \subsetneq \big(\Th^\cont_\forall(M)\big)_\aff$.

Recall that a class of $\cL$-structures $\cK$ has the \emph{joint embedding property} if any two structures in $\cK$ can be $\cL$-embedded jointly into another member of $\cK$.
The class $\cK$ has the \emph{amalgamation property} if this property holds \emph{over members of $\cK$}.
More precisely: if for any three structures $A,B_0,B_1\in\cK$ and $\cL$-embeddings $f_0\colon A\to B_0$, $f_1\colon A\to B_1$, there are $C\in \cK$ and $\cL$-embeddings $g_0\colon B_0\to C$, $g_1\colon B_1\to C$ with $g_0 f_0 = g_1 f_1$.

Following Hrushovski \cite{HruPatterns}, we will employ the following terminology in the setting of universal theories.

\begin{dfn}
Let $T$ be a universal continuous logic theory.
\begin{itemize}
\item $T$ is \emph{irreducible} if its models have the joint embedding property.
\item $T$ is a \emph{Robinson theory} if its models have the amalgamation property.
\end{itemize}
\end{dfn}

Irreducibility is the appropriate notion of completeness for universal theories, as explained by the following basic fact.
Note however that an irreducible universal theory need not be maximal among consistent universal theories.

\begin{lem}
\label{lem:irreducible-theories}
Let $T$ be a consistent, universal, continuous logic $\cL$-theory.
Then the following are equivalent:
\begin{enumerate}
\item\label{i:lem:irred:irred} $T$ is irreducible.
\item\label{i:lem:irred:Th-cont-forall(M)} There is an $\cL$-structure $M$ such that $T = \Th^\cont_\forall(M)$.
\item\label{i:lem:irred:disjunction} Whenever $\varphi,\psi\in\cL^\qf_x$ are such that $T\models \sup_x\varphi(x) \wedge \sup_x\psi(x) \leq 0$, we have $T\models \sup_x\varphi(x) \leq 0$ or $T\models \sup_x\psi(x) \leq 0$.
\end{enumerate}
If moreover $T$ is an affine theory, then the preceding conditions are also equivalent to the following:
\begin{enumerate}[resume]
\item\label{i:lem:irred:Th-aff-forall(M)} There is an $\cL$-structure $M$ such that $T = \Th^\aff_\forall(M)$.
\item\label{i:lem:irred:sum} Whenever $\varphi,\psi\in\cL^\qf_x$ are such that $T\models \sup_x\varphi(x) + \sup_x\psi(x) \leq 0$, we have $T\models \sup_x\varphi(x) \leq 0$ or $T\models \sup_x\psi(x) \leq 0$.
\end{enumerate}
\end{lem}
\begin{proof}
The implications \autoref{i:lem:irred:irred}$\Rightarrow$\autoref{i:lem:irred:disjunction} and \autoref{i:lem:irred:irred}$\Rightarrow$\autoref{i:lem:irred:sum} are easy.

Let us prove \autoref{i:lem:irred:disjunction}$\Rightarrow$\autoref{i:lem:irred:Th-cont-forall(M)} and \autoref{i:lem:irred:sum}$\Rightarrow$\autoref{i:lem:irred:Th-aff-forall(M)}.
For each tuple of variables $x$ and formula $\varphi\in\cL^\qf_x$, let
\begin{equation*}
r_\varphi = \sup\{\varphi(a) : M\models T,\ a\in M^x\}.
\end{equation*}
The Compactness Theorem of Continuous Logic together with \autoref{i:lem:irred:disjunction} imply the theory
\begin{equation*}
T' = T \cup \{\textstyle\sup_x\varphi(x)\geq r_\varphi : \varphi\in\cL^\qf_x\}
\end{equation*}
is consistent.
Then any model $M\models T'$ satisfies $T = \Th^\cont_\forall(M)$.
Similarly, if $T$ is affine, the theory
$T'' = T \cup \{\textstyle\sup_x\varphi(x)\geq r_\varphi : \varphi\in\cL^{\qf,\aff}_x\}$
is consistent as a consequence of the Compactness Theorem for Affine Logic and \autoref{i:lem:irred:sum}, and any model $M\models T''$ satisfies $T = \Th^\aff_\forall(M)$.

Finally, for \autoref{i:lem:irred:Th-cont-forall(M)}$\Rightarrow$\autoref{i:lem:irred:irred} and \autoref{i:lem:irred:Th-aff-forall(M)}$\Rightarrow$\autoref{i:lem:irred:irred}, assume $T = \Th^\cont_\forall(M)$ and $N_0,N_1\models T$.
Then $N_0$ and $N_1$ are substructures of models of $\Th^\cont(M)$.
Since models of complete continuous theories satisfy joint embedding, the first implication follows.
The same argument applies for the second: the models of $\Th^\aff_\forall(M)$ embed into models of $\Th^\aff(M)$, and the models of complete affine theories also have joint embedding (\cite[Prop.~3.16]{BITaffine}).
\end{proof}

\begin{rmk}
For any $A\subseteq M$, the universal diagram $D^\aff_\forall(A)$ is irreducible.\end{rmk}

The following fact is a counterpart to \cite[Lemma~3.21, Lemma~5.2]{BITaffine}.

\begin{lem}
\label{lem:S(T)-cco-realized}
Let $T$ be an irreducible universal affine theory, and let $M$ be a structure with $T = \Th^\aff_\forall(M)$.
Then for every tuple of variables $x$, the following hold:
\begin{equation*}
\tS^\qf_x(T) = \cl{\co\{\tp^\qf(a) : a\in M^x\}},
\end{equation*}
\begin{equation*}
\cE^\qf_x(T) \subseteq \cl{\{\tp^\qf(a) : a\in M^x\}}.
\end{equation*}
\end{lem}
\begin{proof}
If the convex hull of the types realized in $M$ is not dense in $\tS^\qf_x(T)$, then by the Hahn-Banach Separation Theorem there exist $p\in \tS^\qf_x(T)$ and a formula $\varphi\in\cL^{\qf,\aff}_x$ such that $\varphi(p) = 1$ and $M\models \sup_x\varphi(x)\leq 0$.
Since $p$ can be realized in a model of $T$, we have that $T\not\models \sup_x\varphi(x)\leq 0$, contradicting that $\Th^\aff_\forall(M)\subseteq T$.

The second assertion follows from the first one and Milman's Theorem.
\end{proof}

\begin{rmk}
\label{rmk:Sqf(Tforall)}
For every continuous theory $T$ (in particular, for every affine theory) and every tuple of variables $x$, we have
\begin{equation*}
\tS^\qf_x(T) = \tS^\qf_x(T_\forall),
\end{equation*}
i.e., the canonical embedding $\tS^\qf_x(T) \to \tS^\qf_x(T_\forall)$ is surjective.
In particular, for any structure $M$ and $A\subseteq M$ we have $\tS^{\qf,\aff}_x(A) = \tS^\qf_x\big(D^\aff_\forall(A)\big)$.
\end{rmk}

Before stating some basic facts about Robinson theories, let us recall and introduce the notion of existential closedness in continuous and in affine logic.

\subsection{Existentially closed models}
Let $M\subseteq N$ be an extension of $\cL$-structures.
We recall that $M$ is \emph{existentially closed} in $N$, which we will write $M\preceq^\ec N$, if for every $\varphi\in \cL^\qf_{xy}$ and $a\in M^x$ we have
\begin{equation}
\label{eq:exist-closed}
\big(\inf_y\varphi(a,y)\big)^M = \big(\inf_y\varphi(a,y)\big)^N.
\end{equation}
We will say $M$ is \emph{affinely existentially closed} in $N$, and write $M\preceq^\aec N$, if the equality \autoref{eq:exist-closed} holds for every $\varphi\in\cL^{\qf,\aff}_{xy}$ and $a\in M^x$.

\begin{rmk}
\label{rmk:ec-vs-universal-diagram}
Suppose $M\subseteq N$.
Then $M\preceq^\ec N$ if and only if, seen as an $\cL(M)$-structure, $N$ is a model of the universal diagram $D^\cont_\forall(M)$.
Equivalently: $M\preceq^\ec N$ if and only if there is an extension $N\subseteq M'$ such that $M\preceq^\cont M'$ (i.e., such that $M$ is an elementary substructure of $M'$).

Similarly, $M\preceq^\aec N$ if and only if $N\models D^\aff_\forall(M)$, if and only if there is an extension $N\subseteq M'$ with $M\preceq^\aff M'$ (i.e., $M$ is an affine substructure of $M'$).
\end{rmk}

The study of existentially closed structures is also tightly related to $\forall\exists$-theories.
Let us define:
\begin{itemize}
\item The \emph{$\forall\exists$-part} of $T$, denoted by $T_{\forall\exists}$, is the set of conditions of the form
\begin{equation*}
\sup_x\inf_y \varphi(x,y) \leq 0
\end{equation*}
implied by $T$, 
where $\varphi\in \cL^\qf_{xy}$.
\item The \emph{$\forall\exists$-theory} of a structure $M$ is $\Th^\cont_{\forall\exists}(M) = \big(\Th^\cont(M)\big)_{\forall\exists}$.
\item The \emph{$\forall\exists$-affine part} of $T$, denoted by $T_{\forall\exists^\aff}$, is the set of consequences of $T$ of the form $\sup_x\inf_y \varphi(x,y) \leq 0$, where $\varphi\in\cL^{\qf,\aff}_{xy}$.
\item The \emph{$\forall\exists$-affine theory} of $M$ is $\Th^\aff_{\forall\exists}(M) = \big(\Th^\cont(M)\big)_{\forall\exists^\aff}$.
\end{itemize}

Some of the basic results below are stated in the continuous logic setting, although being of course classical and well-known, because we shall use them later.
Other basic facts we only state for affine theories.

\begin{lem}
\label{lem:ec-vs-models-AE}
Let $T$ be a continuous logic theory.
Then the models of $T_{\forall\exists}$ are precisely the existentially closed substructures of models of $T$.

In particular, for any pair of structures $M,N$, if $M\preceq^\ec N$ then $M\models\Th^\cont_{\forall\exists}(N)$.

Similarly, mutatis mutandis, in the affine setting.
\end{lem}
\begin{proof}
The fact that $M\preceq^\ec N\models T$ implies $M\models T_{\forall\exists}$ is immediate from the definitions.
Conversely, if $M\models T_{\forall\exists}$, then an easy application of the Compactness Theorem shows that $T\cup D^\cont_\forall(M)$ is satisfiable.
The same argument works in affine logic.
\end{proof}

Recall that a \emph{chain of structures} is a family $(M_i:i\in I)$ of $\cL$-structures indexed by a directed set $(I,<)$, such that $M_i\subseteq M_j$ whenever $i<j$.
The direct limit of such a chain, whose domain in each sort is the completion of the corresponding union, will be denoted by $\cl{\bigcup_{i\in I}M_i}$.

A theory $T$ is \emph{inductive} if the class of its models is closed under direct limits of chains.
It is a classical fact that a theory is inductive if and only if it admits an $\forall\exists$-axiomatization.
Similarly, an affine theory $T$ is inductive if and only if $T = T_{\forall\exists^\aff}$ (see, for instance, \cite[Prop.~4.3]{Bagheri2014a}).

The definitions and facts of the rest of this section can be stated (up to appropriate modifications, especially for those concerning model completions) for general inductive theories.
However, we restrict our attention to the case of universal theories, and even to Robinson theories when this makes the presentation simpler.

\begin{dfn}
Let $T$ be a universal continuous logic theory.
A model $M$ of $T$ is \emph{existentially closed} if for every extension $M\subseteq N$ with $N\models T$, $M$ is existentially closed in $N$.
In that case, we write $M\models_\ec T$.

Similarly, if $T$ is a universal affine theory, we will say that a model $M$ of $T$ is \emph{affinely existentially closed (a.e.c.)}\ if $M$ is affinely existentially closed in every model of $T$ extending it.
In that case, we shall write $M\models_\aec T$.
\end{dfn}

\begin{lem}
\label{lem:embedding-into-ec-model}
Let $T$ be a universal continuous logic theory.
Then for every $M\models T$ there is an extension $M\subseteq N$ with $N\models_\ec T$.

If $T$ is affine, then, in particular, $N\models_\aec T$.
\end{lem}
\begin{proof}
By a standard chain construction, successively adding witnesses for all the relevant existential conditions.
We shall revisit this argument below, in the proof of \autoref{prop:Existence-aec-qf-ext}.
\end{proof}

\begin{lem}
\label{lem:irreducible-T-ec-model-Th-forall}
Let $T$ be an irreducible universal affine theory, and let $M\models_\aec T$.
Then $T = \Th^\aff_\forall(M)$.

In particular, $\tS^\qf_x(T) = \cl{\co\{\tp^\qf(a) : a\in M^x\}}$ and $
\cE^\qf_x(T) \subseteq \cl{\{\tp^\qf(a) : a\in M^x\}}$.
\end{lem}
\begin{proof}
Let $N\models T$.
Since $T$ is irreducible, there is an extension $M\subseteq K$ with $K\models T$ into which $N$ embeds.
By the hypothesis on $M$, we have $M\preceq^\aec K$, and in particular $K\models\Th^\aff_\forall(M)$.
Since $N$ embeds into $K$, we also have $N\models \Th^\aff_\forall(M)$.
This proves the first assertion.

The second assertion follows from \autoref{lem:S(T)-cco-realized}.
\end{proof}

\begin{lem}
Let $T$ be an irreducible universal affine theory.
Then for any two models $M,N\models_\aec T$, we have $\Th^\aff_{\forall\exists}(M) = \Th^\aff_{\forall\exists}(N)$.
\end{lem}
\begin{proof}
By irreducibility and \autoref{lem:embedding-into-ec-model}, we may assume that $M\subseteq N$.
By \autoref{lem:ec-vs-models-AE}, we already know that $M\models\Th^\aff_{\forall\exists}(N)$, and we must show the converse.
By \autoref{rmk:ec-vs-universal-diagram} and since $N\models_\aec T$, there is an extension $N\preceq^\aec M'$ such that $M\preceq^\aff M'$.
Therefore, again by \autoref{lem:ec-vs-models-AE}, $N$ is a model of $\Th^\aff_{\forall\exists}(M') = \Th^\aff_{\forall\exists}(M)$, as desired.
\end{proof}

\begin{lem}
\label{lem:ec-chains}
Let $(M_i:i\in I)$ be a chain of $\cL$-structures with $M_i\preceq^\ec M_j$ whenever $i<j$.
Then $M_i\preceq^\ec \cl{\bigcup_{i\in I}M_i}$ for every $i\in I$.

In addition, if $T$ is a universal $\cL$-theory such that $M_i\models_\ec T$ for every $i\in I$, then $\cl{\bigcup_{i\in I}M_i}\models_\ec T$.

Similarly, mutatis mutandis, in affine logic.
\end{lem}
\begin{proof}
The easy proof is left to the reader.
\end{proof}

The following basic fact shows that when considering types over parameters from a model $M$ of an ambient Robinson theory $T$, the model $M$ should be taken a.e.c.\ in order to ensure compatibility with parameter-free types.

\begin{lem}
\label{lem:types-with-parameters-vs-fibers}
Let $T$ be an affine universal theory, and let $x,y$ be two tuples of variables.
Let $\pi^\qf_{x,y}\colon \tS^\qf_{xy}(T)\to \tS^\qf_x(T)$ be the variable restriction map.
For every model $M\models T$ and tuple $a\in M^x$, the map
\begin{equation*}
\tS^\qf_y(a)\to \big(\pi^\qf_{x,y}\big)^{-1}\big(\tp^\qf(a)\big)\subseteq \tS^\qf_{xy}(T),\quad \tp^\qf(b/a)\mapsto \tp^\qf(ab),
\end{equation*}
for $b\in N^y$, $M\preceq^\aec N$,
is an affine embedding.

If $T$ is a Robinson theory and $M\models_\aec T$, then it is an affine homeomorphism.
\end{lem}
\begin{proof}
Given $M\preceq^\aec N$, by \autoref{rmk:ec-vs-universal-diagram} we have $N\models D^\aff_\forall(a)$, which ensures the type $\tp^\qf(b/a)$ is well-defined.
Continuity, affineness and injectivity of the map $\tp^\qf(b/a)\mapsto \tp^\qf(ab)$ are clear.

If $T$ Robinson and $M\models_\aec T$, let us argue that it is also surjective onto the fiber of $\tp^\qf(a)$.
Every $p\in \big(\pi^\qf_{x,y}\big)^{-1}\big(\tp^\qf(a)\big)$ is realized by a tuple $a'b'$ in a model $N'\models T$.
Since $\tp^\qf(a') = \tp^\qf(a)$ and $T$ is a Robinson theory, we may amalgamate $M$ and $N'$ over the substructure generated by $a$.
Since $M\models_\aec T$, this means there is an extension $M\preceq^\aec N$ into which $N'$ embeds.
The image $b$ of $b'$ in $N$ satisfies $p = \tp^\qf(ab)$.
\end{proof}

\begin{dfn}
A \emph{model completion} for a continuous Robinson $\cL$-theory $T$ is a continuous $\cL$-theory $T^*$ such that $(T^*)_\forall = T$ and $T^*$ has quantifier elimination.

Similarly, an \emph{affine model completion} for an affine Robinson theory $T$ is an affine theory $T^{*\aff}$ such that $(T^{*\aff})_\forall = T$ and $T^{*\aff}$ has affine quantifier elimination.
\end{dfn}

\begin{rmk}
For Robinson theories, model completions and \emph{model companions} are the same thing.
\end{rmk}

\begin{rmk}
\label{rmk:irreducible-model-completion}
Let $T$ be an irreducible Robinson theory.
If $T$ admits a model completion $T^*$, then $T^*$ is a complete continuous logic theory.

Similarly, the affine model completion of an irreducible, affine Robinson theory is affinely complete.
\end{rmk}

It is a classical fact that a Robinson theory $T$ has a model completion if and only if the class of its existentially closed models is axiomatizable.
In that case, since the existentially closed models form an inductive class, they are axiomatized by their common $\forall\exists$-theory.
We record the analogous statement in affine logic.

\begin{lem}
\label{lem:model-completion-vs-ec}
Let $T$ be an affine Robinson $\cL$-theory.
Then the following are equivalent:
\begin{enumerate}
\item\label{i:model-completion-vs-ec:mc} 
$T$ admits an affine model completion, $T^{*\aff}$.
\item\label{i:model-completion-vs-ec:ec} The class of the a.e.c.\ models of $T$ is affinely axiomatizable.
\end{enumerate}
If these conditions hold, then $T^{*\aff}$ is precisely the common affine theory of the a.e.c.\ models of $T$.
In particular, the affine model completion is unique and inductive.
\end{lem}
\begin{proof}
The standard arguments work.
See \cite[\textsection4]{Bagheri2014a} for affine analogues of the classical criteria for quantifier elimination.
\end{proof}

\subsection{The convex realization property}
\label{subsec:Tcr}
An important basic construction of affine logic is the \emph{direct convex combination} of structures.
More precisely, if $(M_i:i<n)$ are $\cL$-structures and $(\lambda_i:i<n)$ are positive scalars adding up to $1$, the direct convex combination
\begin{equation*}
K = \bigoplus_{i<n}\lambda_i M_i
\end{equation*}
is an $\cL$-structure with domain $\prod_{i<n}M_i$ that satisfies $\varphi^K(a) = \sum_{i<n}\lambda_i\varphi^{M_i}(a_i)$ for every affine formula $\varphi(x)$ and every tuple $a\in\prod_{i<n} M_i^x$.
In particular, if the $M_i$ are models of an affine theory $T$, then so is $K$.
See \cite[Lem.~3.4]{BITaffine}.

A far-reaching generalization of this construction, introduced in \cite{BITaffine}, is given by the \emph{direct integral} 
\begin{equation*}
K = \int^\oplus_\Omega M_\omega d\mu(\omega)
\end{equation*}
of a \emph{measurable field} $M_\Omega = (M_\omega:\omega\in\Omega)$ of $\cL$-structures, indexed by a probability space $(\Omega,\mu)$.
This construction will be revisited and adapted to our quantifier-free setting in \autoref{sec:qf-measurable-fields}.

Finally, we recall an important notion from \cite[\textsection17]{BITaffine}.
The \emph{convex realization property} is an axiom scheme saying, roughly, that every affine type realized in a finite direct convex self-combination $\bigoplus\lambda_i M$ of a model $M$ can be approximately finitely realized in $M$.
In \autoref{sec:convex-realization}, we will revisit this property and show, among other things, that it is $\forall\exists$-axiomatizable.

\section{Face-preserving Robinson theories}
\label{sec:face-preserving-theories}

As mentioned earlier, a function $\pi\colon X\to Y$ between compact convex sets is \emph{face-preserving} if for every face $F\subseteq X$, the image $\pi(F)$ is a face of $Y$.
For basic general properties and characterizations of these maps, see \autoref{sec:appendix:open-and-face-preserving}.

\begin{dfn}
A theory $T$ is \emph{face-preserving} if it is affine and the variable restriction maps
\begin{equation*}
\pi^\qf_{x,y}\colon \tS^\qf_{xy}(T)\to \tS^\qf_x(T)
\end{equation*}
are face-preserving, for every pair of finite tuples $x,y$.
\end{dfn}

We note that by \autoref{rmk:Sqf(Tforall)}, the property of the definition depends only on the universal part $T_\forall$.

\begin{lem}
\label{lem:pi-qf-xy-face-pres-arbitrary}
Let $T$ be a face-preserving theory, and let $x,y$ be tuples of variables of arbitrary length.
Then the variable projection map $\pi^\qf_{x,y}$ is face-preserving.
\end{lem}
\begin{proof}
It is easy to see that the \emph{convex lifting property} (\autoref{dfn:convex-lifting}) is preserved by inverse limits.
The result thus follows from the equivalence recalled in \autoref{lem:Villadsen}.
\end{proof}

\begin{lem}
\label{lem:inverse-limit-ext}
Let $T$ be a face-preserving theory, and let $x=(x_i:i\in I)$ be a tuple of variables.
Then for every $p \in \tS^\qf_x(T)$ we have  $p \in \cE^\qf_x(T)$ if and only if $\pi_F(p)\in \cE^\qf_{x|_F}(T)$ for every finite subset $F\subseteq I$, where $\pi_F$ denotes the variable restriction map $\tS^\qf_x(T)\to \tS^\qf_{x|_F}(T)$.
In other words,
\begin{equation*}
\cE^\qf_x(T) = \varprojlim_{x' \subseteq x \ \text{finite}} \cE^\qf_{x'}(T).
\end{equation*}
\end{lem}
\begin{proof}
The right-to-left implication is easy (and does not use that $T$ is face-preserving); see \cite[Lemma~5.7]{BITaffine}.
The converse follows from \autoref{lem:pi-qf-xy-face-pres-arbitrary}.
\end{proof}

\begin{prop}
\label{prop:ExtremeTypeTwoSteps}
Let $T$ be a face-preserving Robinson theory and let $M\models_\aec T$.
For any tuples $a\in M^x$ and $b\in M^y$, we have $\tp^\qf(ab) \in \cE^\qf_{xy}(T)$ if and only if $\tp^\qf(a) \in \cE^\qf_x(T)$ and $\tp^\qf(b/a) \in \cE^\qf_y(a)$.
\end{prop}
\begin{proof}
Just as in \cite[Prop.~5.4]{BITaffine}, with the non-trivial step now granted for free, by the hypothesis that $T$ is face-preserving (and \autoref{lem:pi-qf-xy-face-pres-arbitrary}, in the case of infinite tuples).
The easier part uses \autoref{lem:types-with-parameters-vs-fibers}.
\end{proof}

\begin{dfn}
A model $M$ of an affine theory $T$ is \emph{qf-extremal} if for every finite (equivalently, arbitrary) tuple $a\in M^x$, we have $\tp^\qf(a)\in \cE^\qf_x(T)$.
\end{dfn}

Let $T$ an affine Robinson theory and $x$ be a finite tuple of variables.
On the type spaces $\tS^\qf_x(T)$ we define a generalized metric (i.e., a metric allowing to take the value $+\infty$) in the usual way:
\begin{equation*}
d(p, q) = \inf \, \bigl\{ d(a,b) : M \models T, \ a,b \in M^x, \ a \models p, \ b \models q \bigr\},
\end{equation*}
where the infimum of $\emptyset$ is $+\infty$.
The fact that $T$ is Robinson ensures that the triangle inequality is satisfied.
Compare with \cite[Prop.~3.19]{BITaffine}.

\begin{lem}
\label{lem:Ex(T)-partial-closed}
Let $T$ be a face-preserving Robinson theory.
Then for any finite tuple of variables $x$, the set $\cE^\qf_x(T)$ is $d$-closed in $\tS^\qf_x(T)$.

In particular, if $a_n\in M^x$ is a sequence in a model of $T$ converging to some $a\in M^x$, and if $\tp^\qf(a_n)\in \cE^\qf_x(T)$ for all $n$, then $\tp^\qf(a)\in \cE^\qf_x(T)$.
\end{lem}
\begin{proof}
Just as in \cite[Prop.~5.6(ii)]{BITaffine}.
\end{proof}

\begin{prop}
\label{prop:qf-extremal-generated}
Let $T$ be a face-preserving Robinson theory.
Let $a\in M^x$ be any tuple in a model $M\models T$.
If $\tp^\qf(a)\in\cE^\qf_x(T)$, then the substructure $\cl{\langle a\rangle}_\cL\subseteq M$ generated by $a$ is a qf-extremal model of $T$.
\end{prop}
\begin{proof}
By \autoref{lem:embedding-into-ec-model}, we may assume that $M\models_\aec T$.
We observe that if $b = \bar{t}(a)\in M^x$ is obtained from $a\in M^x$ by applying $\cL$-terms $t_i$ (and more generally, if the distance function $d(y,b)$ is quantifier-free definable over $a$) then the type $\tp^\qf(b/a)$ is an exposed point of $\tS^\qf_y(a)$, and in particular $\tp^\qf(b/a)\in \cE^\qf_y(a)$.
Therefore, if $\tp^\qf(a)\in \cE^\qf_x(T)$, then $\tp^\qf(b)\in\cE^\qf_y(T)$ by two applications of \autoref{prop:ExtremeTypeTwoSteps}.
Together with \autoref{lem:Ex(T)-partial-closed} (for passing to the closure), this shows $\cl{\langle a\rangle}_\cL$ is qf-extremal.
\end{proof}

\begin{cor}
\label{cor:qf-extremal-realization}
Let $T$ be a face-preserving Robinson theory and $x$ any tuple of variables.
Then every $p\in \cE^\qf_x(T)$ is realized in a qf-extremal model of~$T$.
\end{cor}
\begin{proof}
We may take any realization of $p$ and consider the structure it generates.
\end{proof}

\begin{prop}
\label{prop:chains-qf-ext}
Let $T$ be a face-preserving Robinson theory.
Let $(M_i:i\in I)$ be a chain of structures.
If each $M_i$ is a qf-extremal model of $T$, then so is $\cl{\bigcup_{i\in I}M_i}$. 
\end{prop}
\begin{proof}
Follows from \autoref{lem:Ex(T)-partial-closed}.
\end{proof}

\begin{prop}
\label{prop:qf-ext-amalgam-and-JE}
Let $T$ be a face-preserving Robinson theory.
\begin{enumerate}
\item\label{i:qf-ext-amalgam} The class of qf-extremal models of $T$ has the amalgamation property.

\item\label{i:qf-ext-JE} If moreover $T$ is irreducible, then the class of qf-extremal models of $T$ has the joint embedding property.
\end{enumerate}
\end{prop}
\begin{proof}
For \autoref{i:qf-ext-JE}, the proof of the first part of \cite[Prop.~5.15]{BITaffine} works verbatim, using \autoref{cor:qf-extremal-realization} to realize the type of an extreme joining in a qf-extremal model.
For \autoref{i:qf-ext-amalgam}, it is more straightforward to give a direct argument, essentially the same as for \autoref{i:qf-ext-JE}, rather than to \emph{reduce} it to \autoref{i:qf-ext-JE} as done in the second part of \cite[Prop.~5.15]{BITaffine}. (For that reduction, one would first need to replace $M$ and $N$ by a.e.c.\ qf-extremal extensions, whose existence we prove below in \autoref{prop:Existence-aec-qf-ext}.)
We leave the choice and the details to the reader.
\end{proof}

Next we prove fundamental properties of a.e.c.\ qf-extremal models.

\begin{lem}
\label{lem:charact-qf-ext-models}
Let $T$ be a face-preserving Robinson theory and let $M$ be a qf-extremal model of $T$.
The following are equivalent:
\begin{enumerate}
\item $M$ is an affinely existentially closed model of $T$.
\item Whenever $M\subseteq N$ with $N$ a qf-extremal model of $T$, we have $M\preceq^\aec N$.
\end{enumerate}
\end{lem}
\begin{proof}
For the non-trivial implication, suppose $M$ satisfies the second condition and let $M\subseteq N$ be any extension with $N\models T$.
Suppose also that $\varphi(a,b) < 0$ for some formula $\varphi\in\cL^{\qf,\aff}_{xy}$ and tuples $a\in M^x$, $b\in N^y$.
Up to adding dummy variables and extending the tuple $a$, we may assume that $a$ enumerates $M$.
Let $q = \tp^\qf(a)\in \cE^\qf_x(T)$, and consider the fiber
\begin{equation*}
F = \big(\pi^\qf_{x,y}\big)^{-1}(q) \subseteq \tS^\qf_{xy}(T).
\end{equation*}
Since $F$ is a closed face and $\tp^\qf(ab)\in F$, by Krein--Milman there is an extreme type $p\in\cE(F)\subseteq \cE_{xy}^\qf(T)$ such that $p(\varphi)<0$.
By \autoref{cor:qf-extremal-realization}, there is a realization $a'b'$ of $p$ in a qf-extremal model $N'$ of $T$.
Identifying $a$ and $a'$ we have $M\subseteq N'$, and therefore $M\preceq^\aec N'$ by our hypothesis on $M$.
In particular, $M\models \inf_y\varphi(a,y)<0$, as desired.
\end{proof}

\begin{prop}
\label{prop:Existence-aec-qf-ext}
Let $T$ be a face-preserving Robinson theory.
For every qf-extremal model $M\models T$, there is a qf-extremal model $N\models_\aec T$ with $M\subseteq N$.
\end{prop}
\begin{proof}
The standard chain argument works.
More precisely, let $(\varphi_\alpha : \alpha<\lambda)$ be a transfinite enumeration of all quantifier-free, affine formulas $\varphi_\alpha = \varphi_\alpha(x,a)$ with parameters $a$ from $M$.
We may assume $\lambda$ is a limit ordinal.
We start with $M_0 = M$ and successively define $M_{\alpha+1}$ as any extension $M'\supseteq M_\alpha$ satisfying both:
\begin{itemize}
\item $M'\models \inf_x\varphi_\alpha(x,a) \leq 0$,
\item $M'$ is a qf-extremal model of $T$;
\end{itemize}
or $M_{\alpha+1} = M_\alpha$ if there is no such extension.
At limits steps, we take the direct union, which remains a qf-extremal model of $T$ by \autoref{prop:chains-qf-ext}.
This yields a qf-extremal model $N_1 = M_\lambda$ of $T$, which has witnesses for all existential affine conditions with parameters from $M$ satisfied in qf-extremal extensions of $N_1$ that are models of $T$.
We then iterate the procedure starting with $N_1$.

Ultimately, we obtain a countable chain $(N_n : n\in\bN)$ of qf-extremal models.
Its direct union, $N$, is affinely existentially closed in every qf-extremal model of $T$ extending it. By \autoref{lem:charact-qf-ext-models}, $N\models_\aec T$.
\end{proof}

The following is an analogue of \cite[Prop.~16.22]{BITaffine}.

\begin{prop}
\label{prop:Mprec1N-qfextremal}
Let $T$ be an affine theory, and let $M$ and $N$ be qf-extremal models of $T$.
If $M\preceq^\aec N$, then $M\preceq^\ec N$.
\end{prop}
\begin{proof}
It is enough to see that whenever we have formulas $\varphi_i\in\cL^{\qf,\aff}_{xy}$ and tuples $a\in M^x$, $b\in N^y$ with $\varphi_i(a,b) < 0$ for all $i<n$, we can find $b'\in M^y$ with $\varphi_i(a,b')<0$ for all $i<n$.
Since $N$ is qf-extremal, $\tp^\qf(ab)\in\cE^\qf_{xy}(T)$, and since the map $\tp^\qf(c/a)\mapsto \tp^\qf(ac)$ is an affine embedding (\autoref{lem:types-with-parameters-vs-fibers}), we have $\tp^\qf(b/a) \in \cE^\qf_y(a)$.
Now, by \autoref{lem:S(T)-cco-realized} applied to the theory $D^\aff_\forall(a)$, $\tp^\qf(b/a)$ is in the closure of the types realized in $M$, and therefore there is $b'\in M^y$ with $\varphi_i(a,b')<0$ for all $i<n$, as desired.
\end{proof}

\begin{cor}
\label{cor:qf-aec-common-theory}
Let $T$ be a face-preserving, irreducible Robinson theory.
Then for all qf-extremal models $M,N\models_\aec T$, we have $\Th^\cont_\forall(M) = \Th^\cont_\forall(N)$.
\end{cor}
\begin{proof}
Follows from \autoref{prop:Mprec1N-qfextremal} and \autoref{prop:qf-ext-amalgam-and-JE}\autoref{i:qf-ext-JE}.
\end{proof}

\section{Quantifier-free-measurable fields}
\label{sec:qf-measurable-fields}

We recall some terminology from \cite[\textsection8]{BITaffine}. A \emph{field of structures} is any family of $\cL$-structures $M_\Omega = (M_\omega : \omega \in \Omega)$, and a \emph{section} of such a field is any function $f \colon \Omega \to \coprod_{\omega \in \Omega} M_\omega$ with $f(\omega) \in M_\omega$ for every $\omega\in\Omega$.
(If the language is many-sorted, a section must also specify a sort, and take all its values in that sort; for simplicity, in what follows we will assume that we work in a single-sorted language.)
A \emph{pointwise enumeration} of $M_\Omega$ is a family of sections $e_I = (e_i : i \in I)$ such that $\bigl\{ e_i(\omega) : i \in I \bigr\}$ is dense in $M_\omega$ for every $\omega \in \Omega$.

Given an $\cL$-structure $M$, a subset $A\subseteq M$, and a sublanguage $\cL_0\subseteq \cL$, we denote by $\cl{\langle A\rangle}_{\cL_0}$ the (closed) $\cL_0$-substructure generated by $A$.
Relaxing the definition of a pointwise enumeration, let us say that a family of sections $e_I$ of a field of $\cL$-structures $M_\Omega$ is a \emph{pointwise generating family} if for every $\omega\in\Omega$ we have:
\begin{equation*}
M_\omega = \cl{\langle \{ e_i(\omega) : i \in I \}\rangle}_\cL.
\end{equation*}

\begin{dfn}
Let $(\Omega,\cB,\mu)$ be a probability space, $M_\Omega$ a field of $\cL$-structures, and $e_I$ a pointwise generating family of $M_\Omega$.
We say $(M_\Omega,e_I)$ is a \emph{qf-measurable field of structures} if for every atomic $\cL$-formula $\varphi(x)$ (equivalently, every $\varphi\in\cL^\qf_x$) and every $x$-subtuple $\bar{e}$ of $e_I$, the function
\begin{equation*}
\omega\mapsto \varphi^{M_\omega}\big(\bar{e}(\omega)\big)
\end{equation*}
is $\mu$-measurable.
\end{dfn}

We recall that \emph{$\mu$-measurable} means measurable with respect to the $\sigma$-algebra $\cB_\mu$ generated by $\cB$ and the $\mu$-null sets.
One could consider only $\cB$-measurable functions (in which case $\mu$ would play no role in the definition), but we shall have no use for this finer version of the definition.

\begin{rmk}
\label{rmk:qf-measurable-vs-measurable}
Every measurable field of structures in the sense of \cite[Def.~8.2]{BITaffine} is qf-measurable in the sense we propose here.
In fact, our definition is more general than the one in \cite{BITaffine} in two ways, one essential and the other merely cosmetic.
The essential difference, which makes our framework \guillemotleft quantifier-free\guillemotright, is the dropping of condition (iii) from \cite[Def.~8.2]{BITaffine}.

Other differences in this and subsequent definitions, compared to those in \cite{BITaffine}, come from working with pointwise generating families rather than with pointwise enumerations.
The reason for this is just to render the presentation more natural in our setting, and to simplify the argument of the proof of \autoref{thm:decomposition-for-aec-models} in the next section.
On the other hand, given a pointwise generating family $e_I$ of a field $M_\Omega$, we may consider the associated family of sections $e'_J$ indexed by the set $J$ of $\cL$-terms of the form $j=t(x_{i_0},\dots,x_{n-1})$ for $i_0,\dots,i_{n-1}\in I$, and such that $e'_j(\omega) = t^{M_\omega}(e_{i_0}(\omega),\dots,e_{i_{n-1}}(\omega))$ for every $\omega\in\Omega$.
Then $e'_J$ is a pointwise enumeration of $M_\Omega$, and the pair $(M_\Omega,e'_J)$ satisfies conditions (i), (ii), and a quantifier-free analogue of (iii) from \cite[Def.~8.2]{BITaffine} (where one considers the collections $\fI^\qf(M_\Omega,e_I,\cL_0) = \{I_0\in \cP_{\aleph_0}(I) : \text{for $\mu$-a.e.\ } \omega \in \Omega,\ M_{\omega,I_0} \subseteq_{\cL_0} M_\omega\}$ in place of $\fI^\aff(M_\Omega,e_I,\cL_0)$).
Moreover, the \emph{measurable sections} and the \emph{direct integral} of the field $(M_\Omega,e_I)$, as defined below, coincide with those of $(M_\Omega,e'_J)$, as defined in \cite{BITaffine}.
\end{rmk}

\begin{dfn}
Let $(M_\Omega,e_I)$ be a qf-measurable field of $\cL$-structures.
\begin{enumerate}
\item Given a point $\omega\in\Omega$, a subset $I_0\subseteq I$ and a sublanguage $\cL_0\subseteq\cL$, we denote by
\begin{equation*}
M_{\omega,I_0,\cL_0} = \cl{\langle\{ e_i(\omega) : i \in I_0 \}\rangle}_{\cL_0}
\end{equation*}
the $\cL_0$-substructure of $M_\omega$ generated by the elements $e_i(\omega)$ with $i\in I_0$.

\item A section $f$ of the field $M_\Omega$ is \emph{measurable} if there exist a countable subset $I_0\subseteq I$ and a countable sublanguage $\cL_0\subseteq \cL$ such that the following two conditions hold:
\begin{itemize}
\item $f(\omega)\in M_{\omega,I_0,\cL_0}$ for every $\omega$ outside a $\mu$-null set,
\item for every $\cL_0$-term $t(x)$ and every $x$-subtuple $\bar{e}$ of $e_{I_0} = (e_i:i\in I_0)$, the function
\begin{equation*}
\omega\mapsto d^{M_\omega}\big(f(\omega),t(\bar{e}(\omega))\big)
\end{equation*}
is $\mu$-measurable.
\end{itemize}
\item We may identify two measurable sections that coincide outside a $\mu$-null set.
The collection of all measurable sections up to this identification will be denoted by $M_{\Omega,I}$.
\end{enumerate}
\end{dfn}

\begin{lem}
Let $(M_\Omega,e_I)$ be a qf-measurable field.
Then for every pair of measurable sections $f,g\in M_{\Omega,I}$, the function $\omega\mapsto d^{M_\omega}\big(f(\omega),g(\omega)\big)$ is $\mu$-measurable.
\end{lem}
\begin{proof}
By a straightforward adaptation of the argument in \cite[Lemma~8.4]{BITaffine}.
\end{proof}

From now on, we endow the set $M_{\Omega,I}$ of measurable sections with the metric:
\begin{equation*}
d(f,g) = \int_\Omega d^{M_\omega}\big(f(\omega),g(\omega)\big) d\mu(\omega).
\end{equation*}

\begin{dfn}
Let $(M_\Omega,e_I)$ be a qf-measurable field of $\cL$-structures.
A section $f$ of the field $M_\Omega$ is \emph{simple} if there is a finite measurable partition $\Omega = \bigsqcup_{j\in J} A_j$ and a collection of $\cL$-terms $t_j(x)$ with variables in $x=(x_i:i\in I)$ such that $f(\omega) = t_j(e_I(\omega))$ for every $\omega\in A_j$ and each $j\in J$.
\end{dfn}

\begin{lem}
\label{lem:simple-sections-are-dense}
Let $(M_\Omega,e_I)$ be a qf-measurable field of structures.
Then the metric space $(M_{\Omega,I},d)$ is complete, and the set of simple sections is dense.
\end{lem}
\begin{proof}
Similarly as in \cite[Lemma~8.5, Prop.~8.8]{BITaffine}.
\end{proof}

\begin{prop}
\label{prop:measurability-terms-and-formulas}
Let $(M_\Omega,e_I)$ be a qf-measurable field of structures over a probability space $(\Omega,\mu)$.
The following hold:
\begin{enumerate}
\item For every $\cL$-term $t(x)$ and tuple $f\in (M_{\Omega,I})^x$, the section $\omega\mapsto t^{M_\omega}\big(f(\omega)\big)$ is measurable.
\item For every quantifier-free formula $\varphi\in\cL^\qf_x$ and every tuple $f\in (M_{\Omega,I})^x$, the function
\begin{equation*}
\oset{\varphi(f)} \colon \omega \mapsto \varphi^{M_\omega}\bigl( f(\omega) \bigr)
\end{equation*}
is $\mu$-measurable, and thus belongs to $L^\infty(\Omega)$.
\end{enumerate}
\end{prop}
\begin{proof}
By essentially the same arguments as in \cite[Lemma~8.6, 8.7]{BITaffine}. 
\end{proof}

\begin{dfn}
\label{dfn:direct-integral}
Let $(M_\Omega,e_I)$ be a qf-measurable field of $\cL$-structures.
The \emph{direct integral} of the field $(M_\Omega,e_I)$ is the $\cL$-structure $K$ with underlying metric space $(M_{\Omega,I},d)$ and with interpretations given by
\begin{equation*}
P^K = \int_\Omega P^{M_\omega}\big(f(\omega)\big)d\mu(\omega),\quad F^K(f)(\omega) = F^{M_\omega}\big(f(\omega)\big),
\end{equation*}
for every predicate symbol $P\in\cL$, every function symbol $F\in\cL$, and every tuple $f$ from $K$ of the appropriate size.
We will denote the direct integral $K$ by
\begin{equation*}
M_{\Omega,I} = \int^\oplus_\Omega M_\omega d\mu.
\end{equation*}
\end{dfn}

One can check exactly as in \cite{BITaffine} that the preceding construction defines indeed an $\cL$-structure.

{\L}os's Theorem for Direct Integrals (see \cite[Thm.~8.11]{BITaffine}) is trivial in our quantifier-free setting (once measurability has been established, as per \autoref{prop:measurability-terms-and-formulas}), but still worth stating.

\begin{prop}
\label{prop:qf-Los}
Let $(M_\Omega,e_I)$ be a qf-measurable field.
Then for every $\varphi\in\cL^{\qf,\aff}_x$ and $f\in (M_{\Omega,I})^x$, we have:
\begin{equation*}
\varphi^{M_{\Omega,I}}(f) = \int_\Omega \varphi^{M_\omega}\big(f(\omega)\big)d\mu(\omega).
\end{equation*}
\end{prop}
\begin{proof}
By linearity of the integral and the fact that $\mu(\Omega) = 1$.
\end{proof}

\begin{cor}
\label{cor:Los-universal-theories}
Let $T$ be a universal affine theory, and let $(M_\Omega,e_I)$ be a qf-measurable field of models of $T$.
Then the direct integral $M_{\Omega,I}$ is a model of $T$.
\end{cor}

Two important particular cases of the direct integral construction are \emph{direct convex combinations}, denoted by $\bigoplus_\omega\lambda_\omega M_\omega$, and \emph{direct multiples}, denoted by $L^1(\Omega,M)$.
The former, which we already mentioned earlier, correspond to measurable fields over a finite or countably infinite set $\Omega$ in which all singletons are measurable, with measure $\mu(\{\omega\}) = \lambda_\omega$.
The latter correspond to \emph{constant} measurable fields, i.e., those with $M_\omega = M$ for all $\omega\in\Omega$ and with pointwise generating family $e_I$ given by constant sections; for instance, $e_i(\omega) = i$ for all $i\in I$ and $\omega\in\Omega$, where $I$ is a subset of $M$ with $\cl{\langle I\rangle}_\cL = M$.
Another distinguished particular case is when the index set $I$ and the language $\cL$ are countable.

In these three special cases, qf-measurable fields are \guillemotleft affinely\guillemotright\ measurable, i.e., measurable in the sense of \cite[Def.~8.2]{BITaffine} (up to the inessential differences discussed in \autoref{rmk:qf-measurable-vs-measurable}).
As pointed out in \cite[Rmk.~8.19]{BITaffine}, these fields are in fact \emph{elementarily measurable}, which implies that the functions of the form $\oset{\varphi(f)}\colon \omega\mapsto \varphi^{M_\omega}\big(f(\omega)\big)$, for any continuous logic formula $\varphi\in\cL^\cont_x$, are $\mu$-measurable as well.

We now introduce an appropriate notion of embedding (and isomorphism) of qf-measurable fields, analogous to that for affinely measurable fields from \cite[Def.~12.9]{BITaffine}.
See also \cite[Rmk.~12.13]{BITaffine}.

Following the convention used there, given two probability spaces $\Omega$, $\Xi$ and an embedding $\fs\colon \MALG(\Omega) \to \MALG(\Xi)$ between their measure algebras, we denote also by $\fs$ the induced embedding
\begin{equation*}
\fs\colon L^\infty(\Omega)\to L^\infty(\Xi)
\end{equation*}
of ordered unit vector lattices: the unique one satisfying $\fs(\chi_a) = \chi_{\fs(a)}$ for every $a\in\MALG(\Omega)$, where $\chi_a$ and $\chi_{\fs(a)}$ denote the corresponding characteristic functions.

\begin{dfn}
\label{dfn:field-embedding-and-iso}
Let $(M_\Omega,e_I)$ and $(N_\Xi,e'_J)$ be qf-measurable fields of $\cL$-structures over probability spaces $(\Omega, \cB, \mu)$ and $(\Xi,\cC,\nu)$.
An \emph{embedding} of $(M_\Omega,e_I)$ into $(N_\Xi,e'_J)$ is a pair $(\sigma,\fs)$ such that:
\begin{itemize}
\item $\sigma\colon M_{\Omega,I} \to N_{\Xi,J}$ is a map between the corresponding direct integrals,
\item $\fs \colon\MALG(\Omega) \to \MALG(\Xi)$ is an embedding between the measure algebras,
\item for every predicate symbol $P\in\cL$ and every tuple of measurable sections $f$ of $M_\Omega$ of the appropriate size, we have $\fs\oset{P(f)} = \oset{P(\sigma f)}$,
\item for every function symbol $F\in\cL$ and every tuple of measurable sections $f$ of $M_\Omega$ of the appropriate size, we have $\sigma F^{M_{\Omega,I}}(f) = F^{N_{\Xi,J}}(\sigma f)$.
\end{itemize}
If $\sigma$ and $\fs$ are bijective, then $(\sigma,\fs)$ is an \emph{isomorphism} of qf-measurable fields.
\end{dfn}

Clearly, the last two items of the definition together are equivalent to the condition that for every quantifier-free formula $\varphi\in\cL^\qf_x$ and every tuple $f\in (M_{\Omega,I})^x$, we have
\begin{equation*}
\fs\oset{\varphi(f)} = \oset{\varphi(\sigma f)}.
\end{equation*}
We also note that if $(\sigma,\fs)$ is an embedding of $(M_\Omega,e_I)$ into $(N_\Xi,e'_J)$, then $\sigma\colon M_{\Omega,I} \to N_{\Xi,J}$ is an embedding of $\cL$-structures.

We end this section with two mild technical conditions on qf-measurable fields that we shall need.
The first one is a complete analogue of \cite[Def.~8.22]{BITaffine}, and will be important in \autoref{thm:decomposition-uniqueness}.

\begin{dfn}
\label{dfn:non-degenerate}
Let $(M_\Omega,e_I)$ be a qf-measurable field of $\cL$-structures over a probability space $(\Omega,\mu)$, and let $f$ be a $\kappa$-tuple of measurable sections enumerating a dense subset of the direct integral $M_{\Omega,I}$.
Let $\theta\colon \Omega\to \tS^\qf_\kappa(\cL)$ denote the map $\omega\mapsto\tp^\qf\big(f(\omega)\big)$, where $\tS^\qf_\kappa(\cL)$ is the type space in $\kappa$ many variables over the empty $\cL$-theory, endowed with the $\sigma$-algebra of Baire sets (which renders $\theta$ $\mu$-measurable).
Let $\theta^*$ denote the induced dual embedding of measure algebras:
\begin{equation*}
\theta^*\colon \MALG\bigl( \tS^\qf_\kappa(\cL), \theta_* \mu \bigr) \to \MALG(\Omega,\mu).
\end{equation*}
We say then that the field $(M_\Omega,e_I)$ is \emph{non-degenerate} if $\theta^*$ is an isomorphism.
\end{dfn}

\begin{rmk}
One may see just as in \cite{BITaffine} that the surjectivity of the embedding $\theta^*$ does not depend on the choice of the dense tuple $f$.
\end{rmk}

\begin{rmk}
\label{rmk:non-degenerate-two-constants}
Under mild assumptions, every qf-measurable field of structures each of them having at least two elements is non-degenerate.
In particular, as is easy to see, if the language $\cL$ has at least two constants, then every qf-measurable field of $\cL$-structures in which these two constants are always distinct is non-degenerate.
This includes all examples of interest in ergodic theory.
See also \cite[Lemma~8.24]{BITaffine}.
\end{rmk}

The second technical condition is that the atoms of the measure space behave as in standard probability spaces, in the sense of the following definition.
This condition will be used in \autoref{lem:atoms-extremal-decomposition}.

\begin{dfn}
\label{dfn:point-like-atoms}
Let $(\Omega,\cB,\mu)$ be a probability space.
Given an atom $A\in\cB$, we will say $A$ \emph{concentrates} at a point $\omega\in\Omega$ if for every $B\subseteq A$, $B\in\cB$, we have $\mu(B) = \mu(A)$ if $\omega\in B$ and $\mu(B) = 0$ if $\omega\notin B$.
We say the space has \emph{point-like atoms} if every atom $A\in\cB$ concentrates at some point.
\end{dfn}

Note that in the previous definition, the singleton $\{\omega\}$ need not be measurable.

Our main source of non-standard probability spaces will be non-metrizable Choquet simplices with a boundary measure.
These satisfy the property of the definition.

\begin{lem}
\label{lem:point-like-atoms-Choquet-simplex}
Let $X$ be a Choquet simplex, let $\cB$ be the $\sigma$-algebra of Baire subsets of $X$, and let $\mu$ be a boundary probability measure on $X$.
Then every atom of $(X,\cB,\mu)$ concentrates at an extreme point of $X$.
In particular, $(X,\cB,\mu)$ has point-like atoms.
\end{lem}
\begin{proof}
Let $A\in\cB$ be an atom.
We consider the conditional probability measure $\mu_A$ given by $\mu_A(B) = \mu(A\cap B)/\mu(A)$ for every $B\in\cB$.
Then $\mu_A$ is a boundary measure as well (see \cite[Lemma~1.26]{BITaffine}).
We claim that its barycenter, $p = R(\mu_A)$, is an extreme point of $X$.
Otherwise, we may write $p = \lambda p_0 + (1-\lambda)p_1$ with $0<\lambda<1$ and $p_0\neq p_1$.
Let $\mu_0$, $\mu_1$ be the unique boundary measures on $X$ with $R(\mu_0)=p_0$, $R(\mu_1)=p_1$.
In particular, $\mu_A = \lambda\mu_0 + (1-\lambda)\mu_1$.
Since $p_0\neq p_1$, we must have $\mu_0\neq \mu_1$, and so there is $B\in\cB$ with $\mu_0(B)<\mu_1(B)$.
It follows that $\mu_A(B)>0$, and since $A$ is an atom, $\mu_A(B) = 1$.
But then $\mu_0(B)=1=\mu_1(B)$, a contradiction.

We deduce that $\mu_A$ is the Dirac measure on $p$, which is what we wanted.
\end{proof}

\section{Quantifier-free-simplicial theories and decomposable models}
\label{sec:qf-simplicial}

\begin{dfn}
Let $T$ be a consistent affine theory.
We say $T$ is \emph{qf-simplicial} if for every finite tuple $x$, the type space $\tS^\qf_x(T)$ is a Choquet simplex.
The condition then holds for arbitrary tuples $x$, using \cite[Lemma~1.28]{BITaffine}.
\end{dfn}

We observe that this property, as that of being face-preserving, only depends on the universal part $T_\forall$.

Given a qf-simplicial theory $T$ and a type $p\in\tS^\qf_x(T)$, let $\mu_p$ denote the unique boundary measure on $\tS^\qf_x(T)$ with barycenter $R(\mu_p) = p$.

\begin{lem}
\label{lem:push-bound-measu-qf-simplicial}
Let $T$ be a qf-simplicial, face-preserving theory.
Then for every $p\in \tS^\qf_{xy}(T)$ we have $\big(\pi^\qf_{x,y}\big)_*\mu_p = \mu_{\pi^\qf_{x,y}(p)}$.
\end{lem}
\begin{proof}
By \autoref{lem:boundary-measure-preservation-face-preserving}\autoref{item:bound-meas-pres:measures}.
\end{proof}

\begin{dfn}
\label{dfn:decomposable-models}
Let $T$ be an affine $\cL$-theory and $M$ be a model of $T$.
We will say the model $M$ is \emph{decomposable} if there is a non-degenerate qf-measurable field of $\cL$-structures $(M_\Omega,e_I)$ over a probability space $(\Omega,\mu)$ with point-like atoms such that each $M_\omega$ is a qf-extremal model of $T$, and
\begin{equation*}
M \cong \int^\oplus_\Omega M_\omega d\mu.
\end{equation*}
\end{dfn}

The following can be seen as a generalization of the Extremal Decomposition Theorem proved in \cite[Thm.~12.3]{BITaffine}.
The argument is essentially the same.

\begin{thm}
\label{thm:decomposition-for-aec-models}
Let $T$ be a qf-simplicial, face-preserving Robinson theory.
Then every affinely existentially closed model of $T$ is decomposable.
\end{thm}
\begin{proof}
Let $M\models_\aec T$, and let $a\in M^I$ be any tuple generating $M$, that is, with $M = \cl{\langle\{a_i:i\in I\}\rangle}_\cL$.
We fix a tuple of variables $x = (x_i:i\in I)$, and we consider the type $p=\tp^\qf(a)\in \tS^\qf_x(T)$ and the unique boundary measure $\mu\in \partial\cM\big(\tS^\qf_x(T)\big)$ with $p = R(\mu)$.

Let $\Omega = \cE^\qf_x(T)$.
For each $\omega\in \Omega$, let $e(\omega)$ be a realization of the type $\omega$ inside a model $M_\omega\models T$.
We may assume that $M_\omega = \cl{\langle \{ e(\omega)_i : i\in I \}\rangle}_\cL$.
By \autoref{prop:qf-extremal-generated}, $M_\omega$ is a qf-extremal model of $T$.
For each $i\in I$, the map $e_i\colon\omega\mapsto e(\omega)_i$ is a section of the field of $\cL$-structures $M_\Omega = (M_\omega:\omega\in\Omega)$, and the tuple $e = e_I = (e_i:i\in I)$ is a pointwise generating family of the field.

Let $\cB$ denote the $\sigma$-algebra of Baire subsets of $\tS^\qf_x(T)$, and consider the trace $\sigma$-algebra on $\Omega$,
\begin{equation*}
\cB_\Omega = \{B\cap \Omega : B\in\cB\}.
\end{equation*}
Recall that if $\tS^\qf_x(T)$ is non-metrizable, $\Omega$ need not be Baire (nor Borel).
However, since $\mu$ is a boundary measure, it \emph{concentrates} on $\Omega$ in the sense that $\mu(B)=0$ for every $B\in \cB$ disjoint from $\Omega$ (cf.~\cite[Thm.~1.24]{BITaffine} and the references therein).
Therefore, as per \cite[Lemma~1.17]{BITaffine},
the formula
\begin{equation*}
\mu_\Omega(B\cap \Omega) = \mu(B)
\end{equation*}
yields a well-defined probability measure on $(\Omega,\cB_\Omega)$.
By \autoref{lem:point-like-atoms-Choquet-simplex}, the probability space $(\Omega,\cB_\Omega,\mu_\Omega)$ has point-like atoms.

Given a formula $\varphi\in\cL^\qf_x$, the map
\begin{equation*}
\omega\mapsto \varphi^{M_\omega}\big(e(\omega)\big) = \omega(\varphi)
\end{equation*}
is continuous, and thus $\cB_\Omega$-measurable.
It follows that $(M_\Omega,e_I)$ is a qf-measurable field.
Moreover, $(M_\Omega,e_I)$ is non-degenerate.
Indeed, let $f\in (M_{\Omega,I})^\kappa$ be a tuple of measurable sections enumerating a dense subset of the direct integral and containing $e_I$ as a subtuple (via some injection $I\to \kappa$).
We consider the map $\theta\colon \Omega\to \tS^\qf_\kappa(\cL)$, $\omega\mapsto\tp^\qf\big(f(\omega)\big)$ as in \autoref{dfn:non-degenerate}, and the projection $\pi\colon\tS^\qf_\kappa(\cL) \to \tS^\qf_x(\cL)$ induced by the injection $I\to \kappa$.
The composition $\pi\circ\theta$ is simply the inclusion of $\Omega$ into $\tS^\qf_x(T)\subseteq \tS^\qf_x(\cL)$.
Therefore, $(\pi\circ\theta)_*\mu_\Omega = \mu$ and the dual embedding $(\pi\circ\theta)^*\colon \MALG\big(\tS^\qf_x(\cL),\mu\big)\to \MALG\big(\Omega,\mu_\Omega\big)$ is the isomorphism sending the class of $B$ to the class of $B\cap\Omega$.
We conclude that the intermediate embedding $\theta^*\colon \MALG\big(\tS^\qf_\kappa(\cL),\theta_*\mu_\Omega\big)\to \MALG\big(\Omega,\mu_\Omega\big)$ is also an isomorphism, as desired.

By \autoref{cor:Los-universal-theories}, the direct integral $M_{\Omega,I}$ is a model of $T$.
In addition, using \autoref{prop:qf-Los}, for every $\varphi\in\cL^{\qf,\aff}_x$ we have
\begin{equation*}
\varphi^M(a)
= p(\varphi)
= \int_\Omega \omega(\varphi) d\mu(\omega)
= \int_\Omega \varphi^{M_\omega}\big(e(\omega)\big) d\mu(\omega)
= \varphi^{M_{\Omega,I}}(e).
\end{equation*}
Therefore, the map $a\mapsto e$ induces an embedding of $M$ into the direct integral $M_{\Omega,I}$.
Let us identify $M$ with its image by this embedding, so that $M\subseteq M_{\Omega,I}$.
We want to show that this inclusion is an equality.

For this, by \autoref{lem:simple-sections-are-dense}, it is enough to show that every simple section of the field $(M_\Omega,e_I)$ belongs to $M$.
Let $b$ be such a simple section, say $b(\omega) = t_j(e(\omega))$ for all $\omega\in A_j$ and $j\in J$, for some finite partition $\Omega = \bigsqcup_{j\in J}A_j$ and $\cL$-terms $t_j(x)$.
Let $\tS^\qf_x(T) = \bigsqcup_{j\in J} A'_j$ be a partition into Baire sets with $A_j = A'_j\cap \Omega$ for each $j\in J$.
Now let $y,z$ be distinct single variables disjoint from $x$, and let $\theta^b\colon \tS^\qf_{xz}(T) \to \tS^\qf_{xyz}(T)$ be the measurable map defined by
\begin{equation*}
\theta^b\big(\tp^\qf(uv)\big) = \tp^\qf\big(ut_j(u)v\big)\quad\text{if}\quad\tp^\qf(u)\in A'_j,
\end{equation*}
for every $xz$-tuple $uv$ in a model of $T$.

Given a $x$-tuple $u$ in a model of $T$, let $\mu_u$ denote the unique boundary measure on $\tS^\qf_x(T)$ with barycenter $\tp^\qf(u)$, and similarly for tuples in other variables.

\begin{claim*}
For every $c\in N$ in an extension $M_{\Omega,I}\subseteq N$, $N\models T$, we have $\theta^b_*\mu_{ec} = \mu_{ebc}$.
\end{claim*}
\begin{proof}
This is proved exactly as in \cite[Lemma~12.2]{BITaffine}, using \autoref{lem:push-bound-measu-qf-simplicial}.
\renewcommand{\qedsymbol}{$\blacksquare$}
\end{proof}

Now let $M_{\Omega,I}\subseteq N$ be an extension with $N\models_\aec T$.
It follows from the Claim that for any $c,c'$ in a model of $T$ extending $N$, if $\tp^\qf(ec) = \tp^\qf(ec')$ then $\tp^\qf(ebc) = \tp^\qf(ebc')$.
Therefore, and since $T$ is Robinson, the restriction of $\pi^\qf_{zx,y}\colon \tS^\qf_{xyz}(T)\to \tS^\qf_{xz}(T)$ to the fiber $\big(\pi^\qf_{xy,z}\big)^{-1}\big(\tp^\qf(eb)\big)$ is injective, and thus an affine homeomorphism onto the fiber $\big(\pi^\qf_{x,z}\big)^{-1}\big(\tp^\qf(e)\big)$.

Recalling \autoref{lem:types-with-parameters-vs-fibers}, we obtain that the parameter restriction map $\tS^{\qf,\aff}_z(eb)\to\tS^{\qf,\aff}_z(e)$ (with the type spaces computed in $N$) is an affine homeomorphism.
As a consequence, for every $\epsilon>0$ there is $\psi\in\cL^{\qf,\aff}_{xz}$ such that 
\begin{equation*}
N\models \sup_z|d(b,z) - \psi(e,z)|\leq \epsilon.
\end{equation*}
In particular, $\psi(e,b)\leq \epsilon$, and since $e\in M$ and $M\models_\aec T$, there is $b'\in M$ such that $\psi(e,b')\leq 2\epsilon$, and thus $d(b,b')\leq 3\epsilon$.
We deduce that $b\in M$, which finishes the proof.
\end{proof}

\begin{lem}
\label{lem:boundary-meas-qf-extremal-field}
Let $T$ be an affine theory, and let $(M_\Omega,e_I)$ be a qf-measurable field over a probability space $(\Omega, \cB, \mu)$ such that every $M_\omega$ is a qf-extremal model of $T$.
Let $f$ be an $x$-tuple of measurable sections, and let $\theta\colon \Omega \rightarrow \tS^\qf_x(T)$ be the measurable map sending $\omega \mapsto \tp^\qf\bigl( f(\omega) \bigr)$, as in \autoref{dfn:non-degenerate}.
Then $\theta_* \mu$ is a boundary measure on $\tS^\qf_x(T)$.
\end{lem}
\begin{proof}
By two applications of Mokobodzki's characterization of boundary measures, exactly as in \cite[Lemma~12.15]{BITaffine}.
\end{proof}

The following result concerns the uniqueness of the extremal decomposition.

\begin{thm}
\label{thm:decomposition-uniqueness}
Let $T$ be a qf-simplicial theory.
Let $(M_\Omega,e_I)$ and $(N_\Xi,e'_J)$ be qf-measurable fields over probability spaces $(\Omega, \cB, \mu)$ and $(\Xi,\cC,\nu)$ such that all the structures $M_\omega$ and $N_\xi$ are qf-extremal models of $T$.
Suppose $(M_\Omega,e_I)$ is non-degenerate, and that we are given an $\cL$-embedding $\sigma\colon M_{\Omega,I} \rightarrow N_{\Xi,J}$  between the corresponding direct integrals.

Then there is a measure algebra embedding $\fs \colon\MALG(\Omega) \to \MALG(\Xi)$ such that $(\sigma,\fs)$ is an embedding of the measurable field $(M_\Omega,e_I)$ into $(N_\Xi,e'_J)$, in the sense of \autoref{dfn:field-embedding-and-iso}.
If $\sigma$ is bijective and $N_\Xi$ is non-degenerate as well, then $(\sigma,\fs)$ is an isomorphism of fields.
\end{thm}
\begin{proof}
The proof of \cite[Thm.~12.16]{BITaffine} works verbatim here, using \autoref{lem:boundary-meas-qf-extremal-field}.
\end{proof}

The preceding theorem ensures the following is well-defined, and implies the subsequent crucial corollary.

\begin{dfn}
Let $T$ be a qf-simplicial theory.
Let $M\models T$ be a decomposable model, say $M\cong\int^\oplus_\Omega M_\omega d\mu$ where $(M_\Omega,e_I)$ is as in \autoref{dfn:decomposable-models}.
We define the \emph{atomic weight} of $M$ by
\begin{equation*}
w(M) = \max\{\mu(a) : a\in\MALG(\Omega)\ \text{is an atom}\},
\end{equation*}
with $w(M)=0$ if $\MALG(\Omega)$ has no atoms.
\end{dfn}

\begin{cor}
\label{cor:weight-decreases}
Let $T$ be a qf-simplicial theory, and let $M,N$ be two decomposable models of $T$ with $M\subseteq N$.
Then $w(M) \geq w(N)$.
\end{cor}

We end this section with a basic lemma that we will use later.

\begin{lem}
\label{lem:atoms-extremal-decomposition}
Let $T$ be an affine theory, and let $M$ be a decomposable model with $w(M)>0$.
Then we have
\begin{equation*}
M \cong \lambda M_0 \oplus \lambda M_1 \oplus \dots \oplus \lambda M_{k-1} \oplus (1-k\lambda) K,
\end{equation*}
where $\lambda = w(M)$, each $M_i$ for $i<k$ is a qf-extremal model of $T$, and $K$ is a decomposable model of $T$ with $(1-k\lambda)w(K) < \lambda$.
\end{lem}
\begin{proof}
Fix a qf-extremal decomposition $M\cong \int^\oplus_\Omega M_\omega d\mu$ over a probability space $(\Omega,\cB,\mu)$ with point-like atoms.
Let $A_0,\dots, A_{k-1}\in \cB$ be representatives for all the atoms of $\MALG(\Omega)$ of maximal measure.
The models $M_i$ of the statement are then obtained as the direct integrals of the induced fields $(M_\omega:\omega\in A_i)$ over the normalized, restricted probability spaces $(A_i,\mu_{A_i})$, and the model $K$ is obtained as the direct integral of the induced field $(M_\omega:\omega\in B)$ over $(B,\mu_B)$, for $B = \Omega\setminus\bigcup_{i<k}A_i$ (when this set is non-null).
The latter field is qf-measurable, non-degenerate and has point-like atoms, so $K$ is indeed decomposable.
By construction, $(1-k\lambda)w(K) < w(M)$.

Finally, for each $i<k$, let $\omega_i\in A_i$ be as given by \autoref{dfn:point-like-atoms}.
It is easy to see that the map $M_i\to M_{\omega_i}$ sending the class of a section $f$ to the value $f(\omega_i)$ is then a well-defined isomorphism (even though $\{\omega_i\}$ need not be measurable -- this is in fact a particular case of \cite[Lemma~1.17]{BITaffine}).
Each $M_i$ is thus a qf-extremal model of $T$, as desired.
\end{proof}

\section{The convex realization property revisited}
\label{sec:convex-realization}

The following definition will only be used temporarily, since it will turn out to be equivalent to that of the \emph{convex realization property} from \cite[Def.~17.1]{BITaffine}, which we mentioned in \autoref{subsec:Tcr}.

\begin{dfn}\label{dfn:qfCR}
We will say an $\cL$-structure $M$ has the \emph{qf-convex realization property} if it satisfies all the conditions of the form:
  \begin{equation*}
    \sup_{x_0,\dots,x_{m-1},y} \qinf_z \bigvee_{i<n}\big|\varphi_i(z,y)-\sum_{j<m}\lambda_j\varphi_i(x_j,y)\big| = 0,
  \end{equation*}
  where $\varphi_i\in\cL^{\qf,\aff}_{xy}$ and $\lambda_j\geq 0$, $\sum_{j<m}\lambda_j=1$.
\end{dfn}

\begin{rmk}
\label{rmk:qf-crp-forall-exists}
This is a $\forall\exists$-continuous theory.
\end{rmk}

\begin{lem}
\label{lem:non-atomic-has-qfCR}
Let $(M_\Omega,e_I)$ be a qf-measurable field of $\cL$-structures over an atomless probability space $(\Omega,\mu)$.
Then the direct integral $\int_\Omega^\oplus M_\omega\ud\mu(\omega)$ has the qf-convex realization property.
\end{lem}
\begin{proof}
The proof of \cite[Lemma~17.2]{BITaffine} works verbatim.
Notice that the argument requires computing the integral of functions of the form $\omega\mapsto \varphi_i^{M_\omega}\big(a(\omega),b(\omega)\big)$ for measurable sections $a,b$ of the qf-measurable field.
Since the formulas $\varphi_i$ in \autoref{dfn:qfCR} are quantifier-free, these functions are indeed measurable, by \autoref{prop:measurability-terms-and-formulas}, and the integral makes sense.
\end{proof}

\begin{lem}
  \label{lem:qfCR-vs-aec-implies-ec}
  Let $M$ be an $\cL$-structure. The following are equivalent:
  \begin{enumerate}
  \item\label{i:qfCR} $M$ has the qf-convex realization property.
  \item\label{i:aec-implies-ec} For every $\cL$-structure $N$, if $M\preceq^\aec N$ then $M\preceq^\ec N$.
  \end{enumerate}
\end{lem}
\begin{proof}
The proof is the same as in \cite[Lemma~17.3]{BITaffine}.
For the implication \autoref{i:qfCR}$\Rightarrow$\autoref{i:aec-implies-ec} one uses \autoref{lem:S(T)-cco-realized} and \autoref{rmk:ec-vs-universal-diagram} (in place of \cite[Lemma~3.21]{BITaffine}), and for the implication \autoref{i:aec-implies-ec}$\Rightarrow$\autoref{i:qfCR} one uses \autoref{lem:non-atomic-has-qfCR}, \autoref{rmk:qf-crp-forall-exists}, and \autoref{lem:ec-vs-models-AE} (in place of \cite[Remark~16.14(ii)]{BITaffine}).
\end{proof}

\begin{lem}
\label{lem:extremal-model-with-CRP}
Let $T$ be an irreducible universal affine theory.
Let $M\models_\aec T$ have the qf-convex realization property.
Then for every tuple of variables $x$,
\begin{equation*}
\tS^\qf_x(T) = \cl{\{\tp^\qf(a) : a\in M^x\}}.
\end{equation*}
In particular, if $M$ is qf-extremal, then $\cE^\qf_x(T)$ is dense in $\tS^\qf_x(T)$.
\end{lem}
\begin{proof}
For every non-empty open set $U\subseteq\tS^\qf_x(T)$, we may realize any $p\in U$ in some extension $M\preceq^\aec N$.
Therefore, using \autoref{lem:qfCR-vs-aec-implies-ec}, we may find $a\in M^x$ with $\tp^\qf(a)\in U$.
(This is similar to \cite[Lemma~20.2]{BITaffine}.)
\end{proof}

\begin{prop}
\label{lem:qfCR-vs-CR}
The qf-convex realization property and the convex realization property are equivalent.
\end{prop}
\begin{proof}
Clearly, the convex realization property implies its quantifier-free version.
Conversely, suppose $M$ satisfies the conditions of \autoref{dfn:qfCR}.
Let us fix an atomless probability space $(\Omega,\mu)$.
We consider $N_0 = L^1(\Omega,M)$, which has the convex realization property by \cite[Lemma~17.2]{BITaffine}.
Since $M\preceq^\aff N_0$, \autoref{lem:qfCR-vs-aec-implies-ec} gives us $M\preceq^\ec N_0$.
Therefore, by \autoref{rmk:ec-vs-universal-diagram}, there is an extension $N_0\subseteq M_1$ such that $M\preceq^\cont M_1$.
Proceeding inductively, we build a chain:
\begin{equation*}
M = M_0\preceq^\ec N_0 \subseteq M_1\preceq^\ec N_1\subseteq \dots
\end{equation*}
where $N_n = L^1(\Omega,M_n)$ and $M_n\preceq^\cont M_{n+1}$ for every $n\in\bN$.
By \cite[Thm.~18.3]{BITaffine}, we have also $N_n\preceq^\cont N_{n+1}$ for every $n\in\bN$.

Now let $N = \cl{\bigcup_{n\in\bN}M_n} = \cl{\bigcup_{n\in\bN}N_n}$ be the limit of this chain.
Since $N_0\preceq^\cont N$, the model $N$ has the convex realization property.
But then so does $M$, because $M\preceq^\cont N$.
\end{proof}

\begin{cor}
\label{cor:crp-forall-exists-axiomatizable}
The convex realization property is $\forall\exists$-axiomatizable.
\end{cor}

\section{A general Bauer--Poulsen dichotomy}
\label{sec:dichotomy}

We start with several basic facts.

\begin{lem}
\label{lem:realizing-p-in-cl-Ex(T)}
Let $T$ be an irreducible universal affine theory.
Let $x$ be any tuple of variables and let $p\in \cl{\cE^\qf_x(T)}$.
Then for every $M\models_\aec T$ there is an elementary extension $M\preceq^\cont N$ realizing $p$.
\end{lem}
\begin{proof}
By \autoref{lem:irreducible-T-ec-model-Th-forall}, $p$ is the limit of types realized in $M$, and therefore $p$ is realized in some elementary extension of~$M$.
\end{proof}

\begin{lem}
\label{lem:models-ec-Dforall-implies-ThAE}
Let $M$ and $N$ be $\cL$-structures with $M\preceq^\ec N$.
If $N\models_\ec D^\cont_\forall(M)$, then $N\models \Th^\cont_{\forall\exists}(M)$.
\end{lem}
\begin{proof}
Since $D^\cont_\forall(M) = \big(D^\cont(M)\big)_\forall$, there is an extension $N\subseteq M'$ such that $M\preceq^\cont M'$.
If $N\models_\ec D^\cont_\forall(M)$, then $N\preceq^\ec M'$, which implies $N\models \Th^\cont_{\forall\exists}(M')$ by \autoref{lem:ec-vs-models-AE}.
Hence also $N\models \Th^\cont_{\forall\exists}(M)$, as desired.
\end{proof}

\begin{lem}
\label{lem:qf-ext-embeds-into-summand}
Let $T$ be a universal affine theory and $M\models T$ be a qf-extremal model.
If $M$ embeds into a convex combination of models of $T$, say $M\subseteq \bigoplus_{i<k}\lambda_i N_i$ with $\lambda_i>0$ for each $i<k$, then $M$ embeds into every $N_i$.
\end{lem}
\begin{proof}
Let $a\in M^x$ be an enumeration of $M$, and let $b_i\in N_i^x$ be tuples such that $a = \bigoplus_{i<k}\lambda_i b_i$.
Since $\tp^\qf(a)\in \cE^\qf_x(T)$, we must have $\tp^\qf(a) = \tp^\qf(b_i)$ for each $i<k$.
Therefore, the maps $a\mapsto b_i$ define embeddings of $M$ into each model $N_i$.
\end{proof}

\begin{lem}
\label{lem:aec-ec-convex-preservation}
For $i<k$, let $M_i, N_i$ be $\cL$-structures with $M_i\subseteq N_i$, and let $\lambda_i> 0$ be scalars with $\sum_{i<k}\lambda_i = 1$.
The following hold:
\begin{enumerate}
\item\label{i:aec-convex-bipreservation}
$\bigoplus_{i<k}\lambda_i M_i \preceq^\aec \bigoplus_{i<k}\lambda_i N_i$ if and only if $M_i\preceq^\aec N_i$ for every $i<k$.
\item\label{i:ec-convex-preservation}
If $M_i\preceq^\ec N_i$ for every $i<k$, then $\bigoplus_{i<k}\lambda_i M_i \preceq^\ec \bigoplus_{i<k}\lambda_i N_i$.
\end{enumerate}
\end{lem}
\begin{proof}
Let $M=\bigoplus_{i<k}\lambda_i M_i$ and $N=\bigoplus_{i<k}\lambda_i N_i$.

For \autoref{i:aec-convex-bipreservation}, it suffices to recall that for every formula of the form $\varphi(x)=\inf_y\psi(x,y)$ with $\psi\in\cL^{\qf,\aff}_{xy}$, and any tuples $a_i\in M^x$, we have $\varphi^M\big(\bigoplus_{i<k}\lambda_i a_i\big) = \sum_{i<k}\lambda_i\varphi^{M_i}(a_i)$, $\varphi^{N_i}(a_i)\leq \varphi^{M_i}(a_i)$ for each $i<k$, and $\varphi^N\big(\bigoplus_{i<k}\lambda_i a_i\big) = \sum_{i<k}\lambda_i\varphi^{N_i}(a_i)$.

For \autoref{i:ec-convex-preservation}, since $M_i\preceq^\ec N_i$, there are extensions $N_i\subseteq N_i'$ with $M_i\preceq^\cont N_i'$.
By \cite[Lemma~18.1]{BITaffine}, we have $M \preceq^\cont \bigoplus_{i<k}\lambda_i N_i'$.
Therefore, $M$ is existentially closed in the intermediate extension $N$.
\end{proof}

\begin{dfn}
Let $T$ be a qf-simplicial theory.
We will say $T$ is \emph{qf-Bauer} if for every finite tuple of variables $x$, $\tS^\qf_x(T)$  is a Bauer simplex.
We will say $T$ is \emph{qf-Poulsen} if for every finite tuple of variables $x$, $\cE^\qf_x(T)$ is dense in $\tS^\qf_x(T)$.

In either case, the condition holds as well for arbitrary tuples $x$, using \autoref{lem:inverse-limit-ext} in the Bauer case, and \cite[Lemma~1.31]{BITaffine} in the Poulsen case.
\end{dfn}

The following is our main model-theoretic result.
The proof is a more involved version of that of \cite[Thm.~20.8]{BITaffine}.

\begin{thm}
\label{thm:abstract-dichotomy}
Let $T$ be a qf-simplicial, face-preserving, irreducible Robinson theory satisfying the following property:
\begin{enumerate}
\item[(D)] For every qf-extremal model $M\models_\aec T$, every model of $\Th^\cont_{\forall\exists}(M)$ is decomposable, as a model of $T$, in the sense of \autoref{dfn:decomposable-models}.
\end{enumerate}
Then $T$ is Bauer or Poulsen, in the quantifier-free sense.
\end{thm}
\begin{proof}
For the purposes of this proof, we will say that an extension of structures $M\subseteq N$ is \emph{full}, and write $M\preceq^{\ec*} N$, if $N\models_\ec D^\cont_\forall(M)$.
Clearly, if $M\preceq^{\ec*} N$ then $M\preceq^\ec N$ and also, by \autoref{lem:models-ec-Dforall-implies-ThAE}, $N\models \Th^\cont_{\forall\exists}(M)$.
Moreover, whenever $M\preceq^\ec N$, we may find an extension $N\subseteq N'$ with $M\preceq^{\ec*} N'$.

Suppose $T$ is not qf-Bauer, and fix some type $p_0\in \cl{\cE^\qf_x(T)} \setminus \cE^\qf_x(T)$.
Let us also fix a qf-extremal model $M\models_\aec T$ (existence follows from \autoref{cor:qf-extremal-realization} and \autoref{prop:Existence-aec-qf-ext}).
By our hypothesis (D), every full extension $M\preceq^{\ec*} N$ is decomposable.

\begin{claim*}
For every full extension $M\preceq^{\ec*}N$, if $w(N)>0$, then there is a larger full extension $M\preceq^{\ec*} N'$, $N\subseteq N'$ (hence, $N\preceq^\ec N'$) with $w(N')<w(N)$.
\end{claim*}
\begin{proof}
Let $\lambda = w(N)>0$, and apply \autoref{lem:atoms-extremal-decomposition} to obtain, up to isomorphism:
\begin{equation*}
N = \lambda M_0\oplus \lambda M_1 \oplus \ldots \oplus \lambda M_{k-1} \oplus (1-k\lambda) K,
\end{equation*}
where each $M_i\models T$ is qf-extremal, $K$ is decomposable, and $(1-k\lambda)w(K) < \lambda$.
By \autoref{lem:qf-ext-embeds-into-summand}, $M$ embeds into each model $M_i$ (and into $K$, if $k\lambda<1$), and we may in fact assume that $M\subseteq M_i$.
These inclusions are compatible with $M\subseteq N$, in the sense that if we let 
\begin{equation*}
\tilde{M} = \lambda M \oplus \ldots \oplus \lambda M \oplus (1-k\lambda) K,
\end{equation*}
then the natural inclusions $M\subseteq \tilde{M} \subseteq N$ compose as expected.

Next we argue that $M_i\models_\aec T$ for every $i<k$.
Fix $i$, and suppose we have $M_i\subseteq M_i'$ with $M_i'$ qf-extremal.
Since $M\models_\aec T$, we must have $M\preceq^\aec M_i'$, and by \autoref{prop:Mprec1N-qfextremal}, $M\preceq^\ec M_i'$.
Let $M_j'=M_j$ for each $j\neq i$, and consider the model:
\begin{equation*}
M' = \lambda M_0' \oplus \ldots \oplus \lambda M_{k-1}' \oplus (1-k\lambda) K.
\end{equation*}
We have the sequence of inclusions:
\begin{equation*}
M\subseteq \tilde{M} \subseteq N \subseteq M'.
\end{equation*}
By \autoref{lem:aec-ec-convex-preservation}\autoref{i:ec-convex-preservation}, $\tilde{M}\preceq^\ec M'$.
Since $M\preceq^\ec N$, it follows that $M\preceq^\ec M'$.
Moreover, since the extension $M\subseteq N$ is full, we have $N\preceq^\ec M'$.
But then, by \autoref{lem:aec-ec-convex-preservation}\autoref{i:aec-convex-bipreservation}, we must have $M_i\preceq^\aec M_i'$.
Using \autoref{lem:charact-qf-ext-models}, we conclude that $M_i\models_\aec T$, as desired.

Now, by \autoref{lem:realizing-p-in-cl-Ex(T)}, each model $M_i$ admits an elementary extension $M_i\preceq^\cont \tilde{M}_i$ realizing the non-extreme type $p_0$.
By our hypothesis (D) and since $M_i\models_\aec T$, the models $\tilde{M}_i$ are decomposable.
As they are non-extremal, we have $w(\tilde{M}_i)<1$.
Letting
\begin{equation*}
\tilde{N} = \lambda \tilde{M}_0 \oplus \ldots \oplus \lambda \tilde{M}_{k-1} \oplus (1-k\lambda) K,
\end{equation*}
we obtain an extension $N\subseteq \tilde{N}$ with $w(\tilde{N})<w(N)$.
Moreover, by \autoref{lem:aec-ec-convex-preservation}\autoref{i:ec-convex-preservation}, we have $N\preceq^\ec \tilde{N}$ (in fact, $N\preceq^\cont \tilde{N}$), and we may therefore find an extension $\tilde{N}\subseteq N'$ such that $M\preceq^{\ec*} N'$.
It follows that $N\subseteq N'$ and, by \autoref{cor:weight-decreases}, that $w(N')\leq w(\tilde{N}) < w(N)$.
This proves the Claim.
\renewcommand{\qedsymbol}{$\blacksquare$}
\end{proof}

We deduce from the Claim that there is a full extension $M\preceq^{\ec*} N$ with $w(N)=0$. Otherwise, using \autoref{lem:ec-chains} and \autoref{cor:weight-decreases}, we could build a transfinite chain $(N_\alpha)_{\alpha<\omega_1}$ of full extensions of $M$ with $0 < w(N_\beta) < w(N_\alpha)$ for every $\alpha<\beta<\omega_1$, which is impossible.
Finally, since $w(N)= 0$ and recalling \autoref{lem:non-atomic-has-qfCR}, $N$ has the convex realization property.
It follows from the fact that $M$ is existentially closed in $N$, and by \autoref{lem:ec-vs-models-AE}, that $M$ has the convex realization property as well.
Therefore, by \autoref{lem:extremal-model-with-CRP}, $T$ is qf-Poulsen, as desired.
\end{proof}

\begin{rmk}
It can be seen a posteriori (i.e., using the dichotomy) that, under the hypotheses of \autoref{thm:abstract-dichotomy}, every two qf-extremal models $M,N\models_\aec T$ satisfy $\Th^\cont_{\forall\exists}(M) = \Th^\cont_{\forall\exists}(N)$.
\end{rmk}

\begin{rmk}
The hypothesis (D) of the statement of the theorem can be replaced by the following in principle weaker (but more technical) condition: for every qf-extremal model $M\models_\aec T$, every full extension $M\preceq^{\ec*} N$ is decomposable.
\end{rmk}

The following particular case of \autoref{thm:abstract-dichotomy} is equivalent to \cite[Thm.~20.8]{BITaffine}, via the standard coding of first-order logic (affine in this case) into the framework of Robinson theories.

\begin{cor}
\label{cor:dichotomy-model-completion}
Let $T$ be a qf-simplicial, face-preserving, irreducible Robinson theory.
Suppose $T$ admits an affine model completion.
Then $T$ is Bauer or Poulsen, in the quantifier-free sense.
\end{cor}
\begin{proof}
If $T$ has an affine model completion, then for any $M\models_\aec T$, every model of $\Th^\aff_{\forall\exists}(M)$ is affinely existentially closed.
Thus the hypothesis (D) of \autoref{thm:abstract-dichotomy} follows from \autoref{thm:decomposition-for-aec-models}.
\end{proof}

\section{Particular case: the Glasner--Weiss dichotomy}
\label{sec:Glasner-Weiss}

The Glasner--Weiss dichotomy for the simplices $\cM_\inv(2^G)$ can also be derived easily from \autoref{thm:abstract-dichotomy}.
We spell out the details here, as they provide a preview of the arguments in the following sections.

We recall that probability algebras can be axiomatized in affine logic in the language $\cL_\PrA=\{\mu,\cup,\cap,\neg,\bZero,\bOne\}$; see \cite[\textsection25.1]{BITaffine} for details and basic properties.
Given a group $G$, we may consider the expanded language $\cL_G = \cL_\PrA \cup \{g : g\in G\}$, where we see each $g\in G$ as a new unary function symbol.
The theory $\PMP_G$ is the universal affine $\cL_G$-theory containing the axioms of probability algebras together with axioms stating that the interpretations of the new function symbols define an action of $G$ by automorphisms of the $\cL_\PrA$-reduct (i.e., the underlying probability algebra).
See \cite[\textsection28]{BITaffine}.

The theory $\PMP_G$ is a Robinson theory.
Indeed, probability algebras can be canonically amalgamated (independently over the base), and this construction extends to the amalgamation of arbitrary probability measure-preserving $G$-systems (\emph{relatively independent joinings}).
Since all models share the trivial substructure $\{\bZero,\bOne\}$, $\PMP_G$ is moreover irreducible.

Given a set $I$ and a tuple of variables $x = (x_i:i\in I)$, we denote by $2^x$ the compact product space $2^I$.
We let $\proj_{x,y}\colon 2^{xy}\to 2^x$ denote the canonical projection.

\begin{prop}
\label{prop:PMPG-qf}
Let $G$ be a group.
For every tuple of variables $x$ (of arbitrary length) we have a canonical affine homeomorphism:
\begin{equation*}
\Phi^G_x\colon \tS_x^\qf(\PMP_G) \to \cM_\inv\big((2^x)^G\big).
\end{equation*}
In particular, the extreme types correspond to the ergodic measures.

Moreover, the maps $\Phi^G_x$ commute with variable restriction, in the following sense:
\begin{equation*}
\Phi^G_x\circ \pi^\qf_{x,y} = (\proj_{x,y})_*\circ \Phi^G_{xy}.
\end{equation*}
\end{prop}
\begin{proof}
The first part is precisely \cite[Prop.~28.2(i)]{BITaffine} (the countability assumption on the group is not needed).
The moreover part follows easily from the construction in the proof.
\end{proof}

\begin{cor}
\label{cor:PMPG-basics}
For every group $G$, the theory $\PMP_G$ is a qf-simplicial, face-preserving, irreducible Robinson theory.
\end{cor}
\begin{proof}
Follows from \autoref{prop:PMPG-qf} and \autoref{prop:cMinv(X)-properties}.
\end{proof}

\begin{thm}
\label{thm:Glasner-Weiss-revisited}
Let $G$ be a countable group.
Then for every non-empty compact metrizable space $K$, $\cM_\inv\big(K^G\big)$ is a Bauer simplex or the Poulsen simplex.
\end{thm}
\begin{proof}
For the case $K = 2^{\aleph_0}$, by \autoref{prop:PMPG-qf}, it is enough to check the hypotheses of \autoref{thm:abstract-dichotomy} for the theory $T = \PMP_G$.
On the other hand, by the previous corollary, we are only left to check property (D) from the statement of \autoref{thm:abstract-dichotomy}.
However, this property holds for the theory $\PMP_G$ in a strong form: \emph{every} model is decomposable.
Indeed, this is essentially the Ergodic Decomposition Theorem, in the model-theoretic (and more general) form given in \cite[Cor.~28.4]{BITaffine}.

Now, given a compact metrizable space $K\neq\emptyset$, let $\pi\colon 2^{\aleph_0}\to K$ be a continuous surjection (cf.~\cite[(4.18)]{Kechris1995}) and let $\pi^G\colon (2^{\aleph_0})^G\to K^G$ be the induced product map, which is a continuous, $G$-invariant surjection.
The induced pushforward map
\begin{equation*}
\pi^G_{*\inv}\colon \cM_\inv\big((2^{\aleph_0})^G\big)\to \cM_\inv\big(K^G\big)
\end{equation*}
is also surjective.
Indeed, one may find a universally measurable function $\sigma\colon K\to 2^{\aleph_0}$ with $\pi\circ\sigma = \id_K$ (cf.~\cite[(18.3), (21.10)]{Kechris1995}). Then any $\mu\in \cM_\inv\big(K^G\big)$ can be obtained as $\mu = \pi^G_{*\inv}\big(\sigma^G_{*\inv}(\mu)\big)$.
Since $\pi^G_{*\inv}$ is also continuous, it sends closed sets to closed sets and dense sets to dense sets.
Finally, since $\pi^G_{*\inv}$ preserves the extreme points, $\cM_\inv\big(K^G\big)$ must be Bauer or Poulsen, just as $\cM_\inv\big((2^{\aleph_0})^G\big)$.
\end{proof}

\section{The affine theory $\PMP_{G/H}$}
\label{sec:homogeneous}

Throughout this section we fix a group $G$ and a subgroup $H\leq G$.
We do not make any cardinality assumptions, although we will only be interested in the case where $G/H$ is countable.
In the next section we shall consider the group topology on $G$ induced from the permutation group $\Sym(G/H)$, but here we see $G$ and $H$ just as discrete groups.

We consider the language $\cL_G$ and the theory $\PMP_G$ introduced in the preceding section.
The unique sort of the language $\cL_G$ will be denoted by $S_G$.
We enrich the language $\cL_G$ to obtain a two-sorted language,
\begin{equation*}
\cL_{G,H} = \cL_G \cup \cL^H_\PrA \cup \{i\}.
\end{equation*}
Here, $\cL^H_\PrA$ is a copy of the language $\cL_\PrA$ on a new sort $S_H$.
The symbol $i$ is a unary function symbol from the sort $S_H$ into the sort $S_G$.
We then consider the universal affine $\cL_{G,H}$-theory $\PMP_{G,H}$ consisting of the following axioms:
\begin{itemize}
\item the sort $S_G$ satisfies $\PMP_G$;
\item the sort $S_H$ is a probability algebra (as an $\cL^H_\PrA$-structure);
\item the map $\iota\colon S_H\to S_G$ is an embedding of probability algebras;
\item for every $x\in S_H$ and every $h\in H$, $h\iota(x) = \iota(x)$.
\end{itemize}
In other words, the models $M_{G,H} = (M_G,M_H)\models \PMP_{G,H}$ consist of a probability measure-preserving $G$-system $M_G$ together with a distinguished measure subalgebra $\iota(M_H)\subseteq \Fix_H(M_G)$.
In what follows, we identify $\iota(M_H) = M_H$.

Let us now consider the set $\Phi$ of all atomic $\cL_{G,H}$-formulas $\varphi$ of the form
\begin{equation*}
\varphi(x) = \mu(t(x)),
\end{equation*}
where $x$ is a finite tuple of variables of the sort $S_H$ and $t\colon (S_H)^x\to S_G$ is an $\cL_{G,H}$-term with free variables in $x$ and taking values in the sort $S_G$.
For each such atomic formula $\varphi\in\Phi$, let $P_\varphi$ be a new predicate symbol with the same arity, bound and uniform continuity modulus as $\varphi$.
We consider the expanded language
\begin{equation*}
\cL_{G,H}^+ = \cL_{G,H} \cup \{P_\varphi : \varphi\in\Phi\},
\end{equation*}
and the $\cL_{G,H}^+$-theory $\PMP_{G,H}^+$ that expands $\PMP_{G,H}$ with axioms asserting that $P_\varphi(x) = \varphi(x)$ for all $\varphi\in\Phi$.
In other words, $\PMP_{G,H}^+$ is a partial, atomic Morleyization of $\PMP_{G,H}$.

Finally, let us consider the one-sorted language:
\begin{equation*}
\cL_{G/H} = \cL^H_\PrA \cup \{P_\varphi : \varphi\in\Phi\}.
\end{equation*}
The following is our main theory of interest.

\begin{dfn}
The theory $\PMP_{G/H}$ is the universal affine reduct of $\PMP_{G,H}^+$ to the language $\cL_{G/H}$.
In symbols, $\PMP_{G/H} = \big(\PMP_{G,H}^+|_{\cL_{{G/H}}}\big)_{\forall^\aff}$.
\end{dfn}

Every model $M_{G,H} = (M_G,M_H)\models \PMP_{G,H}$ can be expanded in a unique way to a model of the definitional expansion $\PMP_{G,H}^+$.
With the induced structure, the sort $M_H$ becomes then a model of $\PMP_{G/H}$.
Conversely, we have the following.

\begin{lem}
\label{lem:M-induces-MGH}
For every model $M\models \PMP_{G/H}$ there is a uniquely determined, up to isomorphism, model $M_{G,H} = (M_G,M_H)\models \PMP_{G,H}$ with $M_H = M$ (as $\cL_{G/H}$-structures) and $M_G = \cl{\langle M_H\rangle}_{\cL_G}$.
\end{lem}
\begin{proof}
Consider the set $M_0$ consisting of pairs $(t,a)$ where $t\colon (S_H)^x\to S_G$ is an $\cL_{G,H}$-term and $a\in M^x$ is a tuple from $M$.
We may endow $M_0$ with the pseudometric
\begin{equation*}
d_0\big((t,a),(t',a')\big) = P_{\mu(t(x)\triangle t'(x'))}^M(a,a'),
\end{equation*}
where the variables of the tuple $x'$ are first renamed, if necessary, to ensure they are disjoint from those of $x$.
The fact that this defines a pseudometric is indeed encoded in the theory $\PMP_{G/H}$.
We then define $M_G = (\widehat{M_0/d_0}, d)$ as the completion of the quotient of $M_0$ by the relation of being at $d_0$-distance $0$.
Let $[t,a]$ denote the class of the pair $(t,a)$.

We turn $M_G$ into an $\cL_G$-structure in the obvious way, by defining, for instance:
\begin{itemize}
\item $[t(x),a]\cap [t'(x'),a'] = [t(x)\cap t'(x'),aa']$ (for disjoint $x$ and $x'$),
\item $g([t,a]) = [gt,a]$,
\item $\mu([t,a]) = P_{\mu(t(x))}^M(a)$.
\end{itemize}
Again, the axioms of $\PMP_{G/H}$ ensure that these operations and predicate are well-defined and uniformly continuous on $M_0/d_0$, thus extending to the completion, and that the resulting $\cL_G$-structure is a model of $\PMP_G$.
Finally, we may also define, for every $a\in M$ and any single variable $x$ of the sort $S_H$:
\begin{itemize}
\item $\iota(a) = [\iota(x),a]$.
\end{itemize}
We obtain an $\cL_{G,H}$-structure $M_{G,H} = (M_G,M)$ which is indeed a model of $\PMP_{G,H}$, and such that the $\cL_{G/H}$-structure induced on the sort $S_H$ is precisely~$M$.
Moreover, by construction, $M_G = \cl{\langle M\rangle}_{\cL_G}$.

The uniqueness part is clear.
\end{proof}

\begin{rmk}
It is easy to see, using the preceding lemma and \cite[Thm.~19.2]{BITaffine}, that the full continuous logic reduct of the universal affine theory $\PMP_{G,H}^+$ to the language $\cL_{G/H}$ is again universal and affine.
In other words, $\PMP_{G/H}$ is just $\PMP_{G,H}^+|_{\cL_{{G/H}}}$.
\end{rmk}

\begin{dfn}
\label{dfn:MGH}
Given a model $M\models\PMP_{G/H}$, we will denote by
\begin{equation*}
M^{G,H} = (M^G,M)
\end{equation*}
the model of $\PMP_{G,H}$ \emph{generated by $M$}, as given by \autoref{lem:M-induces-MGH}.
In particular, $M^G$ is a model of $\PMP_G$ and $M\subseteq\Fix_H(M^G)$.
\end{dfn}

\begin{lem}
\label{lem:PMP-GH-extensions}
Let $M\subseteq N$ be an extension of models of $\PMP_{G/H}$.
Then $M^{G,H}$ can be identified with a uniquely determined $\cL_{G,H}$-substructure of $N^{G,H}$.
\end{lem}
\begin{proof}
Clear.
\end{proof}

\begin{prop}
\label{prop:PMP-GH-Robinson}
The theory $\PMP_{G/H}$ is irreducible and Robinson.
\end{prop}
\begin{proof}
By definition, $\PMP_{G/H}$ is universal.
We note that the trivial probability algebra $\{\bZero,\bOne\}$ can be expanded in a unique way to a model of $\PMP_{G/H}$, and that as such it is a common substructure of all models of $\PMP_{G/H}$.
Therefore, we only need to show that $\PMP_{G/H}$ has the amalgamation property.

By \autoref{lem:PMP-GH-extensions}, and using that $\PMP_{G,H}^+$ is a quantifier-free definitional expansion of $\PMP_{G,H}$, it is enough to show that $\PMP_{G,H}$ has the amalgamation property.
Now, since $\PMP_G$ is Robinson, given $M_{G,H}=(M_G,M_H)$ and $N_{G,H}=(N_G,N_H)$ two models of $\PMP_{G,H}$ with a common substructure $K_{G,H}=(K_G,K_H)$, we may amalgamate $M_G$ and $N_G$ over $K_G$ into a model $L$ of $\PMP_G$.
Letting $L_{G,H}=(L,\Fix_H(L))$, we have $L_{G,H}\models \PMP_{G,H}$, and $L_{G,H}$ amalgamates $M_{G,H}$ and $N_{G,H}$ over $K_{G,H}$.
\end{proof}

\begin{prop}
\label{prop:PMP-GH-qf}
For every tuple of variables $x = (x_i)_{i\in I}$ (of arbitrary length) we have a canonical affine homeomorphism:
\begin{equation*}
\Phi^{G/H}_x\colon \tS_x^\qf(\PMP_{G/H}) \to \cM_\inv\big((2^x)^{G/H}\big).
\end{equation*}
In particular, the extreme types correspond to the ergodic measures.

Moreover, the maps $\Phi^{G/H}_x$ commute with variable restriction:
\begin{equation*}
\Phi^{G/H}_x\circ \pi^\qf_{x,y} = (\proj_{x,y})_*\circ \Phi^{G/H}_{xy}.
\end{equation*}
\end{prop}
\begin{proof}
The proof of \autoref{prop:PMPG-qf} (that is, the argument in \cite[Prop.~28.2(i)]{BITaffine}) adapts easily to our setting.
More precisely, given a model $M\models\PMP_{G/H}$, we may realize the $G$-system $M^G$ generated by $M$ as an action by measurable maps on a probability space $(\Omega_M,\mu_M)$ (e.g., the Stone space of $M^G$).
Given the type $p = \tp^\qf(a)$ of a tuple $a\in M^x$, we define the measure $\Phi^{G/H}_x(p)$ as the pushforward of $\mu_M$ by the $G$-equivariant map:
\begin{equation*}
\Psi_{a} \colon \Omega_M \to (2^x)^{G/H}, \quad \Psi_a(\omega)(gH)(x_i) = 1 \iff \omega \in ga_i,
\end{equation*}
where we identify the elements $a_i\in M$ with some measurable representatives $a_i\subseteq \Omega$.
The fact that $M\subseteq\Fix_H(M^G)$ ensures that $\Psi_{a}$ is well-defined.
The good definition of $\Phi^{G/H}_x$, as well as the fact it is continuous, affine and bijective, are seen exactly as in \cite[Prop.~28.2(i)]{BITaffine}.
In particular, for surjectivity, every measure $\nu\in \cM_\inv\big((2^x)^{G/H}\big)$ induces a model $N_\nu\models \PMP_G$ associated to the measure-preserving system $G\actson \big((2^x)^{G/H},\nu\big)$.
We may then consider the model $M_\nu = \Fix_H(N_\nu)$ of $\PMP_{G/H}$ and the tuple $a\in M_\nu^x$ given by
\begin{equation*}
a_i = \big\{z\in (2^x)^{G/H} : z(H)(x_i) = 1\big\}.
\end{equation*}
Note that $N_\nu = (M_\nu)^G$, since $N_\nu = \cl{\langle a\rangle}_{\cL_G}$. If $p$ is the quantifier-free type of $a$, we have $\Phi^{G/H}_x(p) = \nu$.
\end{proof}

\begin{rmk}
The type space $\tS_x^\qf(\PMP_{G/H})$ can be identified naturally with a subset of $\tS_x^\qf(\PMP_G)$, by taking the type of $a\in M^x$ to the type of $\iota(a)\in (M^G)^x$.
Also, the collection $\cM_\inv\big((2^x)^{G/H}\big)$ can be identified naturally with a subset of $\cM_\inv\big((2^x)^G\big)$, via the canonical embedding $(2^x)^{G/H} \to (2^x)^G$.
The map $\Phi^{G/H}_x$ is then simply the restriction of the map $\Phi^G_x$ of \autoref{prop:PMPG-qf} to these two subsets.
\end{rmk}

\begin{cor}
\label{cor:PMP-GH-basic-properties}
The theory $\PMP_{G/H}$ is a qf-simplicial, face-preserving, irreducible Robinson theory.
\end{cor}
\begin{proof}
By \autoref{prop:PMP-GH-Robinson}, \autoref{prop:PMP-GH-qf}, and \autoref{prop:cMinv(X)-properties}.
\end{proof}

The following proposition is a counterpart to \autoref{prop:PMP-GH-qf} at the level of the structures.
For every $\cL_{G,H}$-term $t(x)$ with variables of the sort $S_H$ and with target sort $S_G$, we will denote by $A_t\subseteq (2^x)^{G/H}$ the clopen set
\begin{equation*}
A_t = \{z\in (2^x)^{G/H} : \tilde{t}(z) = 1\},
\end{equation*}
where we see $(2^x)^{G/H}$ as a (product) Boolean algebra with the canonically induced $G$-action, and $\tilde{t}\colon (2^x)^{G/H}\to 2$ is the Boolean homomorphism induced by $t$ in the natural way.
For instance, if $t(x) = g\iota(x_i)\cap (1\setminus\iota(x_j))$, then $\tilde{t}(z) = z(gH)(x_i)\cap (1\setminus z(H)(x_j))$ as computed in the Boolean algebra $2$, and thus
\begin{equation*}
A_t = \{z \in (2^x)^{G/H} : z(gH)(x_i) = 1\ \&\ z(H)(x_j)=0\}.
\end{equation*}

\begin{prop}
\label{prop:MGa-vs-mu-a}
Let $M$ be a model of $\PMP_{G/H}$ and $a\in M^x$ be any tuple.
Denote by $M^G_a = \cl{\langle a\rangle}_{\cL_G}$ the $\cL_G$-substructure of $M^G$ generated by $a$.
Let
\begin{equation*}
\mu_a = \Phi^{G/H}_x\big(\tp^\qf(a)\big) \in \cM_\inv\big((2^x)^{G/H}\big)
\end{equation*}
be the measure given by \autoref{prop:PMP-GH-qf}, and let $M_{\mu_a}$ be its corresponding measure algebra.
We see $M_{\mu_a}$ as a model of $\PMP_G$ with the natural $G$-action.

We then have an $\cL_G$-isomorphism
\begin{equation*}
\Theta_a\colon M^G_a\cong M_{\mu_a}
\end{equation*}
satisfying that for every $\cL_{G,H}$-term of the form $t\colon (S_H)^x\to S_G$,
\begin{equation*}
\Theta_a\big(t(a)\big) = A_t.
\end{equation*}
In particular, $\tp^\qf(a)$ is an extreme type if and only if $M^G_a$ is an ergodic $G$-system.
\end{prop}
\begin{proof}
Follows from the construction in the proof of \autoref{prop:PMP-GH-qf}.
\end{proof}

\begin{cor}
\label{cor:PMP-GH-extremal-models}
A model $M$ of $\PMP_{G/H}$ is qf-extremal if and only if $M^G$ is ergodic.
\end{cor}
\begin{proof}
By \autoref{prop:MGa-vs-mu-a}, applied to any enumeration $a\in M^x$ of $M$.
\end{proof}

\section{A Bauer--Poulsen dichotomy for permutation groups}
\label{sec:pmp}

For this final section, we fix a countable set $\cS$ and a permutation group $G\leq \Sym(\cS)$.
We assume $G$ acts transitively on $\cS$, and we denote by $H$ the stabilizer of a point $s_0\in \cS$ in $G$ (the particular choice of point is irrelevant).
In particular, we may identify $\cS = G/H$ as $G$-sets.

The full symmetric group $\Sym(\cS)$ carries the topology of pointwise convergence, which makes it a Polish group.
For our purposes, it would be natural to assume $G$ is closed in $\Sym(\cS)$.
However, in the exchangeability literature, it is common to formulate the relevant invariance property in terms of a \emph{countable} group of transformations (if the underlying group of the problem is mentioned at all).
One of the aims of the present work is to emphasize the importance of this topology for exchangeability theory.

We denote by $\cl{G}$ and $\cl{H}$ the closures of $G$ and $H$ in $\Sym(\cS)$, which are therefore Polish groups themselves.
Note that $\cl{H}$ is the stabilizer of $s_0$ inside $\cl{G}$.
In particular, $\cl{H}$ is an open subgroup of $\cl{G}$, and again we have $\cS = \cl{G}/\cl{H}$ as $\cl{G}$-sets.
We also note that a probability measure $\mu$ on $2^\cS$ is $G$-invariant if and only if it is $\cl{G}$-invariant, so the notation $\cM_\inv(2^\cS)$ is unambiguous.

\begin{rmk}
\label{rmk:PMP-GH-vs-clGclH}
The theories $\PMP_{G/H}$ and $\PMP_{\cl{G}/\cl{H}}$ are interdefinable in a trivial manner (every atomic $\cL_{\cl{G}/\cl{H}}$-formula is equivalent to an atomic $\cL_{G/H}$-formula).
\end{rmk}

In \autoref{sec:Glasner-Weiss}, we saw how the Glasner--Weiss Theorem for the simplex $\cM_\inv(2^G)$ (for countable $G$) can be deduced from the model-theoretic Bauer--Poulsen dichotomy established in \autoref{sec:dichotomy}, applied to the theory $\PMP_G$.
In that case, the key hypothesis (D) of \autoref{thm:abstract-dichotomy} holds unconditionally and in the strongest possible manner.
Our goal in this section is to apply \autoref{thm:abstract-dichotomy} to the theory $\PMP_{G/H}$.
We will see that $\PMP_{G/H}$ satisfies property (D) provided that the pair of Polish groups $(\cl{G},\cl{H})$ has Property (T).

\begin{dfn}
A model $M\models \PMP_{G/H}$ is \emph{full} if $M = \Fix_H(M^G)$, where $M^G$ is the model of $\PMP_G$ generated by $M$, as in \autoref{dfn:MGH}.
\end{dfn}

For the rest of the section and by an abuse of notation, when writing $\cL_{G/H}$-formulas we may replace the occurrences of basic predicates of the form $P_\varphi$ by the corresponding $\cL_{G,H}$-formulas $\varphi$.
Moreover, for further simplicity, we may omit the symbol $\iota$.
Thus, for instance, we will write $\inf_y\mu(gx\triangle y)$ instead of $\inf_y P_{\mu(g\iota(x)\triangle \iota(y))}$.

\begin{lem}
\label{lem:aec-implies-full}
Every affinely existentially closed model of $\PMP_{G/H}$ is full.
\end{lem}
\begin{proof}
Let $M\models_\aec \PMP_{G/H}$, and consider $N = \Fix_H(M^G)$.
The pair $(M^G,N)$ forms a model of $\PMP_{G,H}$ (and of $\PMP_{G,H}^+$), and with the induced structure $N$ becomes a model of $\cL_{G/H}$.
Moreover, $M$ is an $\cL_{G/H}$-substructure of $N$.

Let $b\in N$ and $\epsilon>0$.
Since $M^G$ is generated by $M$, we may find an $\cL_{G,H}$-term $t\colon (S_H)^x\to S_G$ and a tuple $a\in M^x$ such that
\begin{equation*}
N\models \mu\big(t(a)\triangle b\big) < \epsilon.
\end{equation*}
Since $M$ is affinely existentially closed in $N$, the same condition is satisfied for some $b'\in M$.
It follows that $d(b,b') < 2\epsilon$, showing that $M$ is dense in $N$.
We conclude that $M = N$, as desired.
\end{proof}

\begin{lem}
\label{lem:relative-(T)+full}
Assume the pair $(\cl{G},\cl{H})$ has Property (T).
Then for every $\eta>0$ there is $K>0$ with the following property: for every $\cL_{G,H}$-term of the form $t\colon (S_H)^x\to S_G$ there is a finite set $F\subseteq G$ such that for every full model $M\models\PMP_{G/H}$ and every $a\in M^x$, we have:
\begin{equation}
\label{eq:relative-(T)+full}
\inf_{b\in M} \mu\big(t(a)\triangle b\big) \leq K\max_{g\in F}\mu\big(t(a)\triangle gt(a)\big) + \eta.
\end{equation}
\end{lem}
\begin{proof}
Let $\delta = \sqrt{\eta}$, and choose $Q\subseteq \cl{G}$ and $\epsilon>0$ satisfying the conditions of \autoref{dfn:relative-(T)} with respect to $\delta$.
We let then $K = 4/\epsilon^2$.

Fix an $\cL_{G,H}$-term $t(x)$ as in the statement, and consider the clopen set $A_t\subseteq (2^x)^{G/H}$ as defined before \autoref{prop:MGa-vs-mu-a}.
Let $F_0\subseteq Q$ be a finite set satisfying the property of \autoref{lem:relative-(T)}\autoref{i:relative-(T)-measures} for the pair $(\cl{G},\cl{H})$, with respect to $A_t$.
Since $A_t$ depends on finitely many coordinates, for every $g\in F_0$ there is $h\in G$ with $gA_t=hA_t$.
Therefore, we may find a finite set $F\subseteq G$ satisfying the inequality \autoref{eq:relative-(T)-measures} from \autoref{lem:relative-(T)}.
Let $M$ be a full model of $\PMP_{G/H}$,  and consider any tuple $a\in M^x$.
Letting $\mu_a = \Phi^{G/H}_x(\tp^\qf(a))$ be the invariant measure on $(2^x)^{G/H}$ given by \autoref{prop:PMP-GH-qf}, we have, by choice of $F$,
\begin{equation*}
\inf_{[B]_{\mu_a}\in\Fix_H(M_{\mu_a})} \mu_a(A_t\triangle B) \leq K\max_{g\in F}\mu_a(A_t\triangle gA_t) + \eta.
\end{equation*}
On the other hand, by \autoref{prop:MGa-vs-mu-a},
\begin{equation*}
\mu_a(A_t\triangle gA_t) = \mu\big(t(a)\triangle gt(a)\big)
\end{equation*}
for every $g\in G$, and
\begin{equation*}
\inf_{[B]_{\mu_a}\in\Fix_H(M_{\mu_a})} \mu_a(A_t\triangle B) = \inf_{b\in\Fix_H(M^G_a)} \mu(t(a)\triangle b).
\end{equation*}
We conclude by observing that since $M$ is full, $\Fix_H(M_a^G) \subseteq M$.
\end{proof}

\begin{dfn}
We will denote by
\begin{equation*}
\PMP_{G/H}^\full = \bigcap \{\Th^\cont_{\forall\exists}(M) : M\models \PMP_{G/H},\ M\ \text{full}\}
\end{equation*}
the common $\forall\exists$-continuous theory of the full models of $\PMP_{G/H}$.
\end{dfn}

A model of $\PMP_{G/H}^\full$ need not be full.
However, we have the following.

\begin{lem}
\label{lem:Tfull-Fix-cap-MG}
Assume the pair $(\cl{G},\cl{H})$ has Property (T).
Then for every model $M\models \PMP_{G/H}^\full$ we have $\Fix_G(M^G)\subseteq M$.
\end{lem}
\begin{proof}
Let $M\models \PMP_{G/H}^\full$ and $b\in \Fix_G(M^G)$.
Fix an arbitrary $\eta>0$, and let $K>1$ satisfy the property of \autoref{lem:relative-(T)+full} with respect to $\eta$.
Since $M^G = \cl{\langle M\rangle}_{\cL_G}$, we may find an $\cL_{G,H}$-term $t\colon (S_H)^x\to S_G$ and a tuple $a\in M^x$ such that $d(t(a),b) < \eta/K$.
Since $b$ is $G$-invariant, for every $g\in G$ we have:
\begin{equation*}
\mu\big(t(a)\triangle gt(a)\big) < 2\eta/K.
\end{equation*}
Now let $F\subseteq G$ be a finite set satisfying the inequality \autoref{eq:relative-(T)+full} of \autoref{lem:relative-(T)+full} for $x$-tuples in full models of $\PMP_{G/H}$.
We observe that this inequality, quantified universally over $a$, can be expressed by a $\forall\exists$-continuous $\cL_{G/H}$-condition.
Therefore, \autoref{eq:relative-(T)+full} also holds in $M$, for our given tuple $a$.
It follows that there is $b'\in M$ with
\begin{equation*}
\mu\big(t(a)\triangle b'\big) < 3\eta,
\end{equation*}
which in turn implies $d(b,b') < 4\eta$.
We conclude that $b\in M$, as desired.
\end{proof}

The following proposition and corollary show how the main hypothesis of \autoref{thm:abstract-dichotomy} for the theory $\PMP_{G/H}$ follows from relative Property (T).
The proof of the proposition requires a number of technical reminders and verifications concerning direct integrals.
The core of the argument, however, is a simple application of the preceding lemma.

\begin{prop}
\label{prop:Tfull-decomposability}
Assume the pair $(\cl{G},\cl{H})$ has Property (T).
Then every model of $\PMP_{G/H}^\full$ is a decomposable model of $\PMP_{G/H}$.
\end{prop}
\begin{proof}
By \autoref{rmk:PMP-GH-vs-clGclH}, up to passing to a dense subgroup of $G$ (in the topology inherited from $\Sym(\cS)$) whose intersection with $H$ be dense in $H$, we may assume that $G$ is countable.

Let $M\models \PMP_{G/H}^\full$, and consider the associated model $M^G\models \PMP_G$.
As recalled earlier, all models of $\PMP_G$ are decomposable.
More precisely, the affine theory $\PMP_G$ is simplicial by \cite[Thm.~28.3]{BITaffine} (the proof of this result uses $G$ is countable), and therefore, by the Extremal Decomposition Theorem (\cite[Thm.~12.3]{BITaffine}), $M^G$ is isomorphic to the direct integral
\begin{equation*}
N_{\Omega,J} = \int^\oplus_\Omega N_\omega d\mu_\Omega(\omega)
\end{equation*}
of a measurable field $(N_\Omega,e_J)$ of extremal models of $\PMP_G$.
Moreover, also by \cite[Thm.~28.3]{BITaffine}, the extremal models of $\PMP_G$, in the sense of \cite{BITaffine}, are precisely the ergodic $G$-systems (in particular, they coincide with the qf-extremal models).

Let us recall the construction of the measurable field $(N_\Omega,e_J)$ in more detail, as given in \cite[\textsection10]{BITaffine}. (We already considered and adapted this construction in the proof of \autoref{thm:decomposition-for-aec-models}.)
First we fix any tuple $a\in (M^G)^J$ enumerating the model $M^G$ or a dense subset thereof.
In particular, we may choose $a$ so that for some set of indices $I\subseteq J$, the subtuple $a|_I$ enumerates $M$.
We then fix a tuple of variables $x = (x_j)_{j\in J}$ and let 
\begin{equation*}
\Omega = \cE^\fM_x(\PMP_G)\subseteq \tS^\aff_x(\PMP_G)
\end{equation*}
be the set of types of tuples that enumerate (dense subsets of) ergodic $G$-systems.
For each $\omega\in \Omega$, we let $e_J(\omega) = (e_j(\omega))_{j\in J}$ be a realization of $\omega$ in an ergodic model $N_\omega$ with $N_\omega = \cl{\{e_j(\omega):j\in J\}}$.
The tuple $e_J$ is a pointwise enumeration of the field $N_\Omega = (N_\omega : \omega\in \Omega)$.

On the other hand, let $\mu$ denote the unique boundary measure on the simplex $\tS^\aff_x(\PMP_G)$ whose barycenter is $\tp^\aff(a)$.
We see it as defined on $\cB$, the $\sigma$-algebra of Baire subsets of $\tS^\aff_x(\PMP_G)$.
Then $\mu$ \emph{concentrates} on $\Omega$ (in the sense of \cite[Def.~1.15]{BITaffine}), and thereby defines a probability measure $\mu_\Omega$ on $\Omega$, endowed with the $\sigma$-algebra $B_\Omega = \{B\cap\Omega : B\in\cB\}$, given by $\mu_\Omega(B\cap\Omega) = \mu(B)$.
It follows from \cite[Prop.~3.31(i)]{BITaffine} and \autoref{lem:point-like-atoms-Choquet-simplex} that the probability space $(\Omega,\cB_\Omega,\mu_\Omega)$ has point-like atoms.
Equipping $\Omega$ with this measure, the pair $(N_\Omega,e_J)$ becomes a measurable field of models of $\PMP_G$, and the map $a\mapsto e_J$ extends uniquely to an isomorphism $M^G\cong N_{\Omega,J}$.

In order to produce a decomposition for $M$, let us first modify the tuple $e_I = e_J|_I$ slightly.
For every $i\in I$, let
\begin{equation*}
Z_i = \bigcap_{h\in H} \left\{\omega\in\Omega : d^{N_\omega}\big(e_i(\omega),he_i(\omega)\big) = 0\right\}.
\end{equation*}
Since $H$ is countable, the sets $Z_i$ are measurable.
Moreover, since
\begin{equation*}
0 = d^{M^G}(a_i,h(a_i)) = d^{N_{\Omega,J}}(e_i,he_i) = \int_\Omega d^{N_\omega}\big(e_i(\omega),he_i(\omega)\big) d\mu_\Omega(\omega)
\end{equation*}
for every $i\in I$ and $h\in H$, we have $\mu_\Omega(Z_i) = 1$.
We will then replace $e_i$ by the section:
\begin{equation*}
e'_i(\omega) =
\begin{cases*}
e_i(\omega) & if $\omega\in Z_i$ \\
\bZero & if $\omega\notin Z_i$.
\end{cases*}
\end{equation*}

Now, for each $\omega\in \Omega$, we let
\begin{equation*}
M_\omega = \cl{\langle \{e'_i(\omega) : i\in I\} \rangle}_{\cL_\PrA}\subseteq N_\omega
\end{equation*}
be the probability algebra generated by $e'_I(\omega)$.
By construction, $M_\omega\subseteq \Fix_H(N_\omega)$ for every $\omega\in\Omega$.
In particular, the pairs $(N_\omega,M_\omega)$ form models of $\PMP_{G,H}$, and with the induced structure we may see each $M_\omega$ as a model of $\PMP_{G/H}$.
Note that since the $G$-systems $N_\omega$ are ergodic, so are the generated subsystems $(M_\omega)^G \subseteq N_\omega$.
Therefore, by \autoref{cor:PMP-GH-extremal-models}, each $M_\omega$ is a qf-extremal model of $\PMP_{G/H}$.

The tuple of sections $e'_I$ is a pointwise generating family for the field of $\cL_{G/H}$-structures $M_\Omega = (M_\omega : \omega\in \Omega)$, and the pair $(M_\Omega,e'_I)$ is a qf-measurable field.
As per \autoref{rmk:non-degenerate-two-constants}, it is also non-degenerate.
Moreover, by construction and recalling \autoref{prop:qf-Los}, we have the following closely related facts:
\begin{itemize}
\item $M$ is $\cL_{G/H}$-embedded in the direct integral $M_{\Omega,I} = \int^\oplus_\Omega M_\omega d\mu_\Omega(\omega)$, via the map $\zeta\colon a|_I\mapsto e'_I$;
\item the direct integral $M_{\Omega,I}$ is a probability measure subalgebra of $N_{\Omega,J}\cong M^G$;
\item the pair $(M^G,M)$ is $\cL_{G,H}$-embedded in $(N_{\Omega,J},M_{\Omega,I})$.
\end{itemize}
To conclude the proof, we will show that the embedding of $\zeta\colon M\to M_{\Omega,I}$ is surjective.
By \autoref{lem:simple-sections-are-dense}, we just need to prove that every simple section $b$ in $M_{\Omega,I}$ belongs to $\zeta(M)$.
Let $b$ be given by $b(\omega) = t_k(e'_I(\omega))$ for $\omega\in C_k$ and $k<n$, for some finite measurable partition $\Omega = \bigsqcup_{k<n}C_k$ and $\cL_{\PrA}$-terms $t_k(x|_I)$.

For each $k<n$, let $\tilde{c}_k\in M_{\Omega,I}\subseteq N_{\Omega,J}$ be the simple section defined by
\begin{equation*}
\tilde{c}_k(\omega) =
\begin{cases*}
\bOne & if $\omega\in C_k$ \\
\bZero & if $\omega\notin C_k$.
\end{cases*}
\end{equation*}
Since the sections $\tilde{c}_k$ are pointwise $G$-invariant, they are $G$-invariant in the direct integral, i.e., $\tilde{c}_k\in \Fix_G(N_{\Omega,J})$ for every $k<n$.
Therefore, since the pair $(\cl{G},\cl{H})$ has Property (T) and $M\models \PMP^\full_{G/H}$, by \autoref{lem:Tfull-Fix-cap-MG} we have
\begin{equation*}
\tilde{c}_k\in \zeta(M)
\end{equation*}
for every $k<n$.
Say $\tilde{c}_k = \zeta (c_k)$ with $c_k\in M$.

Finally, we consider the element
\begin{equation*}
c = \bigcup_{k<n} c_k\cap t_k(a|_I) \in M,
\end{equation*}
and we claim that $\zeta(c) = b$.
Indeed, for every $\omega\in\Omega$ and $\ell<n$, we have
\begin{equation*}
\zeta(c)(\omega) = \bigcup_{k<n} \tilde{c}_k(\omega)\cap t_k(e'_I(\omega)) = t_\ell(e'_I(\omega))\quad\text{if and only if}\quad \omega\in C_\ell,
\end{equation*}
just as $b(\omega)$.
This shows that $M\cong \int^\oplus_\Omega M_\omega d\mu_\Omega(\omega)$, and hence that $M$ is a decomposable model of $\PMP_{G/H}$.
\end{proof}

\begin{cor}
\label{cor:PMPGH-property-D}
Assume the pair $(\cl{G},\cl{H})$ has Property (T).
Let $M$ be an affinely existentially closed model of $\PMP_{G/H}$.
Then every model of $\Th^\cont_{\forall\exists}(M)$ is a decomposable model of $\PMP_{G/H}$, in the sense of \autoref{dfn:decomposable-models}.
\end{cor}
\begin{proof}
By \autoref{lem:aec-implies-full} and \autoref{prop:Tfull-decomposability}.
\end{proof}

\begin{rmk}
In the preceding corollary, we need not suppose that $M$ is qf-extremal.
\end{rmk}

\begin{rmk}
If the pair $(\cl{G},\cl{H})$ has \emph{strong} Property (T), in the sense that the set $Q$ in \autoref{dfn:relative-(T)} can always be taken finite, then in the preceding arguments it is enough to consider the $\forall\exists$-\emph{affine} theory of the full models of $\PMP_{G/H}$.
(The crucial point is that in \autoref{lem:Tfull-Fix-cap-MG}, the size of the set $F$ does no longer depend on the term $t$.)
In particular, in that case, the conclusion of \autoref{cor:PMPGH-property-D} holds for all models of $\Th^\aff_{\forall\exists}(M)$.
\end{rmk}

The following is our main theorem concerning simplices of invariant measures.

\begin{thm}
\label{thm:main-final}
Let $G$ be a transitive permutation group on a countable set $\cS$, and let $H$ be the stabilizer of a point of $\cS$.
Suppose $\cl{H}$ has relative Property (T) in the Polish group $\cl{G}$.
Then for every compact metrizable space $K$ with $|K|>1$, the Choquet simplex $\cM_\inv(K^\cS)$ is Bauer if and only if $\cl{G}$ has Property (T), and is Poulsen otherwise.
\end{thm}
\begin{proof}
By \autoref{cor:PMP-GH-basic-properties} and \autoref{cor:PMPGH-property-D}, the theory $\PMP_{G/H}$ satisfies all the hypotheses of our abstract dichotomy result, \autoref{thm:abstract-dichotomy}.
Therefore, $\PMP_{G/H}$ is either a qf-Bauer or a qf-Poulsen theory.
By the description of the quantifier-free type spaces in \autoref{prop:PMP-GH-qf}, this means that the compact convex sets $\cM_\inv\big((2^x)^\cS\big)$, for countable tuples $x$, are either all Bauer simplices or are all the Poulsen simplex.

Now, if $\cl{G}$ has Property (T), then $\cM_\inv(K^\cS)$ is a Bauer simplex for any compact metrizable $K$, by \autoref{prop:prop-T-implies-Bauer}.
Otherwise, by \autoref{prop:prop-T-iff-Bauer}, the simplex $\cM_\inv(2^\cS)$ is not Bauer, and it follows from our preceding conclusion that $\cM_\inv\big((2^{\aleph_0})^\cS\big)$ is the Poulsen simplex.
But then, exactly as in the proof of \autoref{thm:Glasner-Weiss-revisited}, this implies that $\cM_\inv(K^\cS)$ is the Poulsen simplex for every compact metrizable space $K$ with $|K|>1$, as desired.
\end{proof}

We end with an application of the preceding theorem to a family of examples whose status was left open in \cite[\textsection5.5]{Austin2008} and which serves as a counterpart to the cases discussed in \autoref{rmk:oligomorphic-groups-exchangeability-theory}.

Let $\Sym_0(\bN)$ be the group of finitely supported permutations of $\bN$, acting on $(\bZ/m\bZ)^{\oplus\bN}$ by permuting coordinates.

\begin{cor}
\label{cor:finer-grained-Poulsen}
Let $m>1$.
Consider the standard action of the semidirect product $G = (\bZ/m\bZ)^{\oplus\bN}\rtimes \Sym_0(\bN)$ on $\cS = (\bZ/m\bZ)^{\oplus\bN}$.
Then for every compact metrizable space $K$ with $|K|>1$, the set of $G$-invariant measures on $K^\cS$ is the Poulsen simplex.
\end{cor}
\begin{proof}
The group $H = \Sym_0(\bN)$ is the stabilizer of $0\in \cS$ in $G$.
Seen as a subgroup of $\Sym(\cS)$, its closure $\cl{H}$ ($\cong \Sym(\bN)$) has Property (T) (\cite[Thm.~6.7]{TsankovUnitary}).
Therefore, by \autoref{thm:main-final}, it is enough to show that $\cl{G}$ ($\cong (\bZ/m\bZ)^{\oplus\bN}\rtimes \Sym(\bN)$) does not.

For this, let $Q\subseteq\cl{G}$ be any compact subset.
Then there is $N\in\bN$ such that $a_n=0$ for every $(a,\sigma)\in Q$ and $n\geq N$.
Choose $0<p<1$ and consider the unitary representation $\cl{G}\actson^\pi L^2\big(2^\bN,\mu_p\big)$ where $\mu_p$ is the product measure $(p,1-p)^\bN$, and the action is given by
\begin{equation*}
\big(\pi(a,\sigma)\cdot f\big)\big((x_n)_n\big) = \exp\Big({\textstyle \frac{2\pi i}{m}}{\sum_n a_nx_n}\Big) f\big((x_{\sigma(n)})_n\big).
\end{equation*}
Since $m>1$, this representation does not have invariant unit vectors.
Let $f\in L^2(2^\bN,\mu_p)$ be the characteristic function of the set $A_N=\{(x_n)_n:\forall n<N,x_n=0\}$, normalized to become a unit vector.
Then given $\epsilon>0$, if $p$ is chosen sufficiently close to $1$ (so that $\mu_p(A_N)$ is close to 1), $f$ is a $(Q,\epsilon)$-invariant unit vector.
\end{proof}

\appendix

\section{Open maps and face-preserving maps}
\label{sec:appendix:open-and-face-preserving}

In this appendix, we identify and study open, face-preserving maps in the category of compact convex sets as the appropriate counterparts to open maps in the category of compact Hausdorff spaces.
In doing so we complement, and in a sense complete, a classical theorem of Vesterstr{\o}m \cite{Vesterstrom}, and we revisit and highlight several results of Villadsen \cite{Villadsen}.

The following general definition will play a crucial role.

\begin{dfn}
Let $(A,\leq_A)$ and $(B,\leq_B)$ be partially ordered sets, and let $\beta\colon A\to B$ and $\alpha\colon B\to A$ be order-preserving maps.
We say that $\alpha$ is an \emph{upper adjoint} of $\beta$, and that $\beta$ is a \emph{lower adjoint} of $\alpha$, if for every $a\in A$ and $b\in B$ we have:\begin{equation*}
\beta(a) \leq_B b \iff a\leq_A \alpha(b).\footnote{In order theory, the pair $(\alpha,\beta)$ is called a \emph{Galois connection}.}
\end{equation*}
\end{dfn}

In the category of compact Hausdorff spaces, a continuous surjective map $\pi\colon X\to Y$ is open if and only if the dual map $\pi^*\colon C(Y)\to C(X)$
admits an upper adjoint (see \autoref{prop:open-maps} below).
We will thus study the existence of upper adjoints for the Kadison dual maps
\begin{equation*}
\pi^{*\aff}\colon \cA(Y)\to \cA(X)
\end{equation*}
arising from affine continuous surjections $X\to Y$ between compact convex sets.

Note that in general, an order-preserving map $\beta\colon A\to B$ admits at most one upper adjoint.
Also, if $\beta$ is a positive linear map between ordered vector spaces, and $\alpha\colon B\to A$ is an upper adjoint for $\beta$, then the map $\gamma\colon B\to A$ defined by $\gamma(b) = -\alpha(-b)$, is a lower adjoint for $\beta$.
Similarly, in the converse direction.
Therefore, for some of the maps we are interested in, the existence of an upper adjoint and the existence of a lower adjoint are equivalent properties.

We start by introducing a notation that we shall use throughout.

\begin{dfn}
\label{dfn:upper-adjoint}
Let $\pi\colon X\to Y$ be a surjective map between two sets.
For every bounded function $f\colon X\to \bR$, we define
\begin{equation*}
\check{f}^\pi\colon Y\to \bR,\quad \check{f}^\pi(y) = \inf f(\pi^{-1}(y)) = \inf\{f(x) : \pi(x)=y\}.
\end{equation*}
\end{dfn}

\begin{rmk}
\label{rmk:theta_f}
On $X$ we have $\check{f}^\pi\circ\pi \leq f$, and if $f_0\leq f_1$ on $X$ then $\check{f_0}^\pi \leq \check{f_1}^\pi$ on $Y$. Moreover, for every pair of bounded functions $f\colon X\to \bR$, $g\colon Y\to\bR$, we have
\begin{equation}\label{eq:check-f-pi}
g\circ\pi \leq f \iff g\leq \check{f}^\pi.
\end{equation}
In other words, the map $f\mapsto\check{f}^\pi$ is the upper adjoint of the precomposition map $\ell^\infty(Y)\to\ell^\infty(X)$, $g\mapsto g\circ\pi$.
\end{rmk}

For a compact Hausdorff space $X$, we will denote by $\ell^\infty_\lsc(X)$ the set of bounded, lower semicontinuous functions $f\colon X\to \bR$.
If $X$ is a compact convex set, we will also use the notation $\cA_\lsc(X)$ for the set of functions $f\in\ell^\infty_\lsc(X)$ that are moreover affine.

We recall the following basic facts.

\begin{lem}
\label{lem:lsc-sup}
\begin{enumerate}[wide]
\item\label{item:lsc-sup} Let $X$ be a compact Hausdorff space. Then for every $f\in\ell^\infty_\lsc(X)$ and~$x\in X$,
\begin{equation*}
f(x) = \sup\{g(x) : g\in C(X),\, g\leq f\}.
\end{equation*}
In particular, since the set $\{g\in C(X) : g\leq f\}$ is directed upward, $f$ is the pointwise limit of an increasing net from $C(X)$.
\item\label{item:lsc-aff-sup} Let $X$ be a compact convex set. Then for every convex $f\in \ell^\infty_\lsc(X)$ and $x\in X$,
\begin{equation*}
f(x) = \sup\{g(x) : g\in\cA(X),\, g\leq f\}.
\end{equation*}
If moreover $f\in\cA_\lsc(X)$, then the set $\{g\in\cA(X) : g\leq f\}$ is directed upward, and therefore $f$ is the pointwise limit of an increasing net from $\cA(X)$.
\end{enumerate}
\end{lem}
\begin{proof}
For \autoref{item:lsc-sup}, using Urysohn's Lemma, for every $x\in X$ and real number $r<f(x)$ we may find $g\in C(X)$ such that $g \leq r$ on $X$, $g(x) = r$, and $g = \inf f$ on $f^{-1}(\bR_{\leq r})$. In particular, $g\leq f$, and the equality of the statement holds.
The rest is clear.

For \autoref{item:lsc-aff-sup}, see \cite[Prop.~I.1.2, Cor.~I.1.4]{Alfsen1971}.
\end{proof}

\begin{lem}
\label{lem:theta-continuity}
Let $\pi\colon X\to Y$ be a continuous surjection between compact Hausdorff spaces.
Then for every $f\in\ell^\infty_\lsc(X)$ we have $\check{f}^\pi\in \ell^\infty_\lsc(Y)$, and for every $y\in Y$,
\begin{equation}\label{eq:lower-semicont-sup}
\check{f}^\pi(y) = \sup\{g(y) : g\in C(Y),\, g\circ\pi\leq f\}.
\end{equation}
If moreover $f$ is continuous and $\pi$ is open, then $\check{f}^\pi$ is continuous.
\end{lem}
\begin{proof}
For every $r\in\bR$ we have $(\check{f}^\pi)^{-1}(\bR_{\leq r}) = \pi\bigl(f^{-1}(\bR_{\leq r})\bigr)$, which is closed, so $\check{f}^\pi$ is lower semicontinuous.
By \autoref{eq:check-f-pi}, the right-hand side in \autoref{eq:lower-semicont-sup} can be written as $\sup\{g(y) : g\in C(Y),\, g\leq\check{f}^\pi\}$.
Then \autoref{eq:lower-semicont-sup} reduces to the equality of \autoref{lem:lsc-sup}\autoref{item:lsc-sup}.

Finally, $(\check{f}^\pi)^{-1}(\bR_{<r}) = \pi\bigl(f^{-1}(\bR_{<r})\bigr)$, which is open if $f$ is continuous and $\pi$ is open. Therefore, in that case, $\check{f}^\pi$ is continuous.
\end{proof}

\begin{prop}
\label{prop:open-maps}
Let $\pi\colon X\to Y$ be a continuous surjection between compact Hausdorff spaces.
Then the following are equivalent:
\begin{enumerate}
\item $\pi$ is open.
\item\label{i:prop:open-maps:theta} For every $f\in C(X)$ we have $\check{f}^\pi \in C(Y)$.
\item\label{i:prop:open-maps:adjoint} The dual map $\pi^*\colon C(Y) \to C(X)$ admits an upper adjoint.
\end{enumerate}
Moreover, if these hold then the upper adjoint of $\pi^*$ is given by the map $f\mapsto \check{f}^\pi$.
\end{prop}
\begin{proof}
\begin{cycprf}
\item By \autoref{lem:theta-continuity}.

\item By \autoref{rmk:theta_f}.

\item[\impprev] Suppose $\pi^*$ admits an upper adjoint $\alpha\colon C(X) \to C(Y)$.
Then by \autoref{lem:theta-continuity}, for every $f\in C(X)$ and $y\in Y$,
\begin{equation*}
\check{f}^\pi(y) = \sup\{g(y) : g\in C(Y),\, g\leq \alpha(f)\} = \alpha(f)(y).
\end{equation*}
Therefore, $\check{f}^\pi = \alpha(f) \in C(Y)$.

\item[\impprev] Let $U\subseteq X$ be an open subset.
By Urysohn's Lemma, for every $x\in U$ we may find $f_x\in C(X)$ such that $f_x=0$ on $X\setminus U$ and $f_x(x) < 0$.
Then $U = \bigcup_{x\in U} f_x^{-1}(\bR_{<0})$, and
\begin{equation*}
\pi(U) = \bigcup_{x\in U} \pi\bigl(f_x^{-1}(\bR_{<0})\bigr) = \bigcup_{x\in U} \big(\check{f_x}^\pi\big)^{-1}(\bR_{<0}).
\end{equation*}
If the functions $\check{f_x}^\pi$ are continuous, then $\pi(U)$ is open.
\end{cycprf}
\end{proof}

Before turning our attention to the category of compact convex sets, we state the result of Vesterstr{\o}m mentioned at the beginning of the section, which mixes both categories.
See \cite[Thm.~2.1]{Vesterstrom}.
We include his nice proof here, for the convenience of the reader.

\begin{thm}
\label{thm:Vesterstrom}
Let $\pi\colon X\to Y$ be an affine, continuous surjection between compact convex sets.
Then the following are equivalent:
\begin{enumerate}
\item $\pi$ is open.
\item For every $f\in\cA(X)$ we have $\check{f}^\pi\in C(Y)$.
\end{enumerate}
\end{thm}
\begin{proof}
The downward implication follows from \autoref{lem:theta-continuity}.
For the converse, it is enough to show that for every open set $U = \bigcap_{i<n} \{x\in X : f_i(x) <0 \}$ defined by some $f_i\in \cA(X)$, $i<n$, the set $\pi(U)$ is open in $Y$.
Fix such functions $f_i$, and for each $y\in Y$ consider the compact convex set
\begin{equation*}
C(y) = \left\{\bigl(f_i(x)\bigr)_{i<n} : x\in \pi^{-1}(y)\right\} \subseteq \bR^n.
\end{equation*}
Thus, $y\in \pi(U)$ if and only if $C(y)\cap (\bR_{<0})^n \neq \emptyset$.
On the other hand, by the Hyperplane Separation Theorem, $C(y)\cap (\bR_{<0})^n = \emptyset$ if and only if there exists $v\in (\bR_{\geq 0})^n$ with $\|v\| = 1$ such that
\begin{equation*}
\bigl\langle v,\bigl(f_i(x)\bigr)_{i<n}\bigr\rangle = \sum_{i<n}v_i f_i(x) \geq 0
\end{equation*}
for all $i<n$ and $x\in \pi^{-1}(y)$.
Therefore, letting $f_v = \sum_{i<n}v_if_i \in \cA(X)$, we have
\begin{equation*}
Y\setminus \pi(U) = \proj_Y \left\{(v,y) \in (\bR_{\geq 0})^n \times Y: \|v\| = 1,\, \check{f_v}^\pi(y)\geq 0\right\}.
\end{equation*}
We observe that under the assumption that the functions $\check{f_v}^\pi$ are continuous, the map $(v,y) \mapsto \check{f_v}^\pi(y)$ is continuous.
Indeed, since in general $\|\check{f}^\pi - \check{h}^\pi\|_\infty \leq \|f-h\|_\infty$, we have
\begin{equation*}
\bigl|\check{f_v}^\pi(y) - \check{f_{v'}}^\pi(y')\bigr| \leq \bigl|\check{f_v}^\pi(y) - \check{f_v}^\pi(y')\bigr| + \|v-v'\|_1\cdot \max_{i<n}\|f_i\|_\infty.
\end{equation*}

Hence $Y\setminus \pi(U)$ is compact and $\pi(U)$ is open, as desired.
\end{proof}

\begin{dfn}
\label{dfn:convex-lifting}
Let $\pi\colon X\to Y$ be a function between compact convex sets.
We will say $\pi$ has the \emph{convex lifting property} if whenever the image of an element $x\in X$ can be written as a convex combination
\begin{equation*}
\pi(x) = \lambda y_0 + (1-\lambda)y_1,\quad y_0,y_1\in Y,\ 0< \lambda< 1,
\end{equation*}
there exist $x_0,x_1\in X$ such that
\begin{equation*}
x = \lambda x_0 + (1-\lambda)x_1,\ \pi(x_0)=y_0\ \text{and}\ \pi(x_1)=y_1.
\end{equation*}
\end{dfn}

\begin{rmk}
\label{rmk:CLP-induction}
By induction, if $\pi$ has the convex lifting property and $\pi(x) = \sum_{i<n}\lambda_i y_i$ is a finite convex combination of elements $y_i\in Y$ and coefficients $\lambda_i>0$, then there exist $x_i\in X$ such that $x = \sum_{i<n}\lambda_i x_i$ and $\pi(x_i) = y_i$ for each $i<n$.
\end{rmk}

\begin{lem}
\label{lem:theta-affineness}
Let $\pi\colon X\to Y$ be an affine surjection between compact convex sets.
If $f\in\ell^\infty_\lsc(X)$ is convex, then so is $\check{f}^\pi$.
If moreover $f$ is affine and $\pi$ has the convex lifting property, then $\check{f}^\pi$ is affine.
\end{lem}
\begin{proof}
If $y = \lambda y_0 + (1-\lambda) y_1$ for $y,y_0,y_1\in Y$ and $\lambda\in [0,1]$, then for every $\epsilon>0$ we may find $x_0,x_1\in X$ such that $f(x_i) \leq \check{f}^\pi(y_i) + \epsilon$ and $\pi(x_i) = y_i$ for $i=0,1$.
Therefore $\pi(\lambda x_0 + (1-\lambda)x_1) = y$, and
\begin{align*}
\check{f}^\pi(y) & \leq f(\lambda x_0 + (1-\lambda)x_1) \leq \lambda f(x_0) + (1-\lambda)f(x_1) \\
& \leq \lambda \check{f}^\pi(y_0) + (1-\lambda)\check{f}^\pi(y_1) + \epsilon.
\end{align*}
We deduce that $\check{f}^\pi$ is convex.
If moreover $f$ is affine and $\pi$ has the convex lifting property, let $x\in X$ be such that $f(x) \leq \check{f}^\pi(y) + \epsilon$ and $\pi(x) = y$, and choose $z_0,z_1\in X$ with $x=\lambda z_0 + (1-\lambda)z_1$ and $\pi(z_i) = y_i$.
Then we have
\begin{align*}
\check{f}^\pi(y) +\epsilon & \geq f(\lambda z_0 + (1-\lambda)z_1) = \lambda f(z_0) + (1-\lambda)f(z_1) \\
& \geq \lambda \check{f}^\pi(y_0) + (1-\lambda)\check{f}^\pi(y_1).
\end{align*}
This shows that $\check{f}^\pi$ is also concave, and therefore affine.
\end{proof}

\begin{lem}\label{lem:check-f-equivalent-def-aff}
Let $\pi\colon X\to Y$ be an affine, continuous surjection between compact convex sets.
Let $f\in\ell^\infty_\lsc(X)$ be a convex function.
Then for every $y\in Y$,
\begin{equation*}
\check{f}^\pi(y) = \sup\{g(y) : g\in \cA(Y),\, g\circ\pi\leq f\}.
\end{equation*}
\end{lem}
\begin{proof}
By \autoref{lem:theta-continuity} and \autoref{lem:theta-affineness}, $\check{f}^\pi$ is a lower semicontinuous convex function.
As such, it is the pointwise supremum of the affine continuous functions below it (\autoref{lem:lsc-sup}\autoref{item:lsc-aff-sup}).
The equality of the statement then follows from \autoref{eq:check-f-pi}.
\end{proof}

Clearly, a map with the convex lifting property is face-preserving.
For continuous affine maps, the two properties are equivalent, as observed by Villadsen.

\begin{lem}
\label{lem:Villadsen}
Let $\pi\colon X\to Y$ be an affine, continuous map between compact convex sets.
Then the following are equivalent:
\begin{enumerate}
\item $\pi$ has the convex lifting property.
\item $\pi$ is face-preserving.
\end{enumerate}
\end{lem}
\begin{proof}
For the non-obvious implication see \cite[Lemma~1.1]{Villadsen}.
\end{proof}

\begin{rmk}
If $\pi\colon X\to Y$ is a face-preserving map then, in particular, $\pi(\cE(X))\subseteq \cE(Y)$.
The converse inclusion holds for every affine continuous surjection.
Therefore, if $\pi$ is an affine, continuous, face-preserving surjection, we have the equality:
\begin{equation*}
\pi\bigl(\cE(X)\bigr) = \cE(Y).
\end{equation*}
Preserving the extreme points is in general weaker than preserving all faces, as exemplified by an appropriate projection of a square pyramid onto its base.
However, it is equivalent when $X$ is a metrizable Choquet simplex (see \autoref{prop:ext-preserving-simplex}).
\end{rmk}

\begin{thm}
\label{thm:convex-lifting-maps}
Let $\pi\colon X\to Y$ be an affine, continuous surjection between compact convex sets.
Then the following are equivalent:
\begin{enumerate}
\item\label{item:convex-lifting:lifting-prop}
$\pi$ is face-preserving.

\item\label{item:convex-lifting:theta}
For every $f\in \cA_\lsc(X)$ we have $\check{f}^\pi \in \cA_\lsc(Y)$.

\item\label{item:convex-lifting:adjoint}
The dual map $\pi^{*\aff,\lsc}\colon \cA_\lsc(Y) \to \cA_\lsc(X)$ admits an upper adjoint.

\item\label{item:convex-lifting:Aff-to-Aff-lsc}
For every $f\in \cA(X)$ we have $\check{f}^\pi \in \cA_\lsc(Y)$.
\end{enumerate}
Moreover, if these hold then the upper adjoint of $\pi^{*\aff,\lsc}$ is given by the map $f\mapsto \check{f}^\pi$.
\end{thm}
\begin{proof}
\begin{cycprf}
\item By \autoref{lem:theta-continuity}, \autoref{lem:Villadsen} and \autoref{lem:theta-affineness}.

\item By \autoref{rmk:theta_f}.

\item Let $\alpha\colon \cA_\lsc(X)\to \cA_\lsc(Y)$ be an upper adjoint for $\pi^{*\aff,\lsc}$.
By \autoref{lem:check-f-equivalent-def-aff} and \autoref{lem:lsc-sup}\autoref{item:lsc-aff-sup}, for every $f\in \cA_\lsc(X)$ (in particular, for $f\in \cA(X)$) and every $y\in Y$ we have
\begin{equation*}
\check{f}^\pi(y) = \sup\{g(y) : g\in\cA(Y),\, g\leq\alpha(f)\} = \alpha(f)(y).
\end{equation*}
It follows that $\check{f}^\pi \in \cA_\lsc(Y)$.

\item[\impfirst] 
We prove the convex lifting property for $\pi$.
Let $x\in X$ and $y_0,y_1\in Y$ be such that $\pi(x) = \lambda y_0 + (1-\lambda)y_1$ for some $0<\lambda<1$.
Consider the maps $\ell_0,\ell_1\colon \cA(X) \to \bR$ given by:
\begin{equation*}
\ell_0(f) = \lambda\check{f}^\pi(y_0),\quad \ell_1(f) = f(x) - (1-\lambda)\check{f}^\pi(y_1).
\end{equation*}
We observe that $-\ell_0$ and $\ell_1$ are sublinear functionals.
Moreover, since by assumption $\check{f}^\pi$ is affine for every $f\in\cA(X)$, and since $\check{f}^\pi\big(\pi(x)\big)\leq f(x)$, we have
\begin{equation*}
\ell_0(f) \leq \ell_1(f)
\end{equation*}
for all $f\in \cA(X)$.
By the sandwich form of the Hahn--Banach Theorem (see \cite[Cor.~6]{Simons}), there is a linear functional $p_0\colon \cA(X)\to \bR$ with
\begin{equation}
\label{eq:sandwich-convex-lifting}
\lambda^{-1}\ell_0 \leq p_0 \leq \lambda^{-1}\ell_1.
\end{equation}
It follows immediately that $p_0$ is positive and unital, and in particular, bounded.
That is, $p_0$ is a state of $\cA(X)$, and hence there is $x_0\in X$ such that $p_0 = \hat{x}_0$, i.e., $p_0(f) = f(x_0)$ for every $f\in\cA(X)$.

On the other hand, let $p_1\colon \cA(X)\to \bR$ be the linear functional given by $p_1(f) = (1-\lambda)^{-1}\big(f(x) - \lambda p_0(f)\big)$. From the second inequality in \autoref{eq:sandwich-convex-lifting}, for every $f\in\cA(X)$ we have
\begin{equation*}
p_1(f)\geq \check{f}^\pi(y_1).
\end{equation*}
It follows that $p_1$ is positive and unital, and so there is $x_1\in X$ such that $p_1 = \hat{x}_1$.
Therefore, $f(x) = \lambda f(x_0) + (1-\lambda) f(x_1)$ for every $f\in\cA(X)$, i.e., $x = \lambda x_0 + (1-\lambda)x_1$.

Finally, since $\check{f}^\pi = g$ whenever $g\in\cA(Y)$ and $f = g\circ\pi$, the inequalities \autoref{eq:sandwich-convex-lifting} yield $g(\pi(x_0)) = g(y_0)$ for every $g\in\cA(Y)$. We deduce that $\pi(x_0) = y_0$ and $\pi(x_1) = y_1$, completing the proof.
\end{cycprf}
\end{proof}

We now give a version of the convex lifting property for measures. This property was established for variable restriction maps between affine type spaces in \cite[Lemma~3.6]{BITaffine}. See also \autoref{prop:strong-CLP} below for a more general form.

\begin{rmk}
\label{rmk:pushforward-image}
If $\pi\colon X\to Y$ is a continuous map between compact Hausdorff spaces, the image of the pushforward map $\pi_*\colon \cM(X)\to \cM(Y)$ is precisely $\pi_*\big(\cM(X)\big) = \{\mu\in\cM(Y) : \mu(\pi(X))=1\} \cong \cM\big(\pi(X)\big)$.
\end{rmk}

\begin{prop}
\label{prop:CLP-with-measures}
Let $\pi\colon X\to Y$ be an affine, continuous map between compact convex sets.
Then the following are equivalent:
\begin{enumerate}
\item $\pi$ is face-preserving.
\item\label{item:CLP-measure-lifting} For every $x\in X$ and every probability measure $\mu\in\cM_{\pi(x)}(Y)$, there exists $\nu\in\cM_x(X)$ such that $\pi_*(\nu) = \mu$.
\end{enumerate}
Moreover, in \autoref{item:CLP-measure-lifting}, if $\mu$ is a boundary measure, then $\nu$ can be taken to be a boundary measure as well.
\end{prop}
\begin{proof}
\begin{cycprf}
\item Let $x$ and $\mu$ be as in \autoref{item:CLP-measure-lifting}.
Since $\pi(X)$ is a closed face of $Y$, we must have $\mu(\pi(X)) = 1$ (see, e.g., \cite[1.18]{BITaffine}).
Hence, by \autoref{rmk:pushforward-image}, we may assume without loss of generality that $\pi$ is surjective.
Let $(\mu_i : i\in I)\subseteq \co(Y) \subseteq \cM(Y)$ be a net of simple probability  measures converging to $\mu$.
By \cite[Prop.~I.2.3]{Alfsen1971}, we may choose the $\mu_i$ so that $R(\mu_i) = \pi(x)$ for every $i\in I$.
Using \autoref{lem:Villadsen} and the form of the convex lifting property given by \autoref{rmk:CLP-induction}, and using that $\pi$ is surjective, we may find simple measures $\nu_i\in \co(X)\subseteq \cM(X)$ with $\pi_*(\nu_i) = \mu_i$ and $R(\nu_i) = x$.
Up to passing to a subnet, we may assume that the $\nu_i$ converge to some $\nu\in\cM(X)$.
Therefore, by continuity, $\pi_*(\nu) = \mu$ and $R(\nu) = x$, as desired.

If moreover $\mu\in \partial\cM_{\pi(x)}(Y)$, take any $\nu'\in \partial\cM_x(X)$ with $\nu'\succeq \nu$.
Then for every continuous convex function $f\colon Y\to\bR$ we have $\pi_*(\nu')(f) = \nu'(f\circ\pi) \geq \nu(f\circ\pi) = \mu(f)$.
In other words, $\pi_*(\nu')\succeq \mu$, and by maximality of $\mu$, $\pi_*(\nu') = \mu$.

\item[\impfirst] If $\pi(x) = \sum_{i<n}\lambda_i y_i$ for pairwise distinct $y_i\in Y$, we may consider $\mu = \sum_{i<n}\lambda_i\delta_{y_i}$ and choose a lifting $\nu\in\cM_x(X)$ as given by \autoref{item:CLP-measure-lifting}.
Let $X_i=\pi^{-1}(y_i)$ (hence $\nu(X_i) = \lambda_i$) and let $\nu_i$ denote the conditional measure induced by $X_i$, i.e., $\nu_i(A)=\lambda_i^{-1}\nu(A\cap X_i)$ for every Borel subset $A\subseteq X$.
Note that $\pi_*(\nu_i) = \delta_{y_i}$.
If $x_i = R(\nu_i)$, then $x=\sum_{i<n}\lambda_i x_i$ and $\pi(x_i) = R(\pi_*(\nu_i)) = y_i$. We conclude that $\pi$ has the convex lifting property, and hence is face-preserving.
\end{cycprf}
\end{proof}

Next, we review some fundamental properties of face-preserving and extreme point-preserving maps, essentially due to Villadsen.

Recall first that given a bounded function $f\colon X\to \bR$, the \emph{lower envelope} of $f$ is the function $\check{f}\colon X\to\bR$ defined by
\begin{equation*}
\check{f}(x) = \sup\{g(x) : g\in\cA(X),\, g\leq f\}.
\end{equation*}
(By \autoref{lem:check-f-equivalent-def-aff}, for $f\in \ell^\infty_\lsc(X)$ the lower envelope of $f$ is the same as $\check{f}^{\operatorname{id}}$.)

\begin{lem}
\label{lem:boundary-measure-preservation}
Let $X$ be a metrizable compact convex set, and let $\pi\colon X\to Y$ be an affine, continuous map into another compact convex set with $\pi\big(\cE(X)\big) \subseteq \cE(Y)$.
Then for every $\mu\in \partial\cM(X)$ we have $\pi_*(\mu)\in \partial\cM(Y)$.
\end{lem}
\begin{proof}
Let $\mu\in \partial\cM(X)$ and $f\in C(Y)$.
Since $X$ is metrizable and $\mu$ is a boundary measure, $\cE(X)$ is measurable with $\mu(\cE(X)) = 1$.
On the other hand, by \cite[Prop.~I.4.1]{Alfsen1971}, for every extreme point $y\in\cE(Y)$ we have $\check{f}(y) = f(y)$.
Hence, since $\pi\big(\cE(X)\big) \subseteq \cE(Y)$, we have
\begin{equation*}
\pi_*(\mu)(\check{f}) = \int_{\cE(X)}\check{f}(\pi(x)) d\mu = \int_{\cE(X)} f(\pi(x)) d\mu = \pi_*(\mu)(f).
\end{equation*}
We conclude that $\pi_*(\mu)$ is a boundary measure using \cite[Prop.~I.4.5]{Alfsen1971}.
\end{proof}

\begin{rmk}
\label{rmk:missing-hypothesis-Villadsen}
I am not sure whether the preceding lemma holds in the general non-metrizable case (which we will not need).
The argument contained in the proof of \cite[Prop.~1.3]{Villadsen} seems to have a gap in that case.
\end{rmk}

\begin{prop}
\label{prop:M(X)-convex-lifting}
Let $\pi\colon X\to Y$ be a continuous map between compact Hausdorff spaces.
Then the pushforward map $\pi_*\colon \cM(X) \to \cM(Y)$ has the convex lifting property.
\end{prop}
\begin{proof}
Suppose $\pi_*(\nu) = \lambda\mu_0 + (1-\lambda)\mu_1$ for $\nu\in\cM(X)$, $\mu_0,\mu_1\in \cM(Y)$ and $0<\lambda<1$.
Then $\mu_0$ and $\mu_1$ are absolutely continuous with respect to $\pi_*(\nu)$.
Let $h_0,h_1\in L^1(Y,\pi_*(\nu))$ be the corresponding Radon--Nikodym derivatives, noting that $\lambda h_0 + (1-\lambda)h_1 = 1$.
Consider the measures $\nu_0,\nu_1\in \cM(X)$ defined by $\nu_i(A) = \int_A (h_i\circ\pi) d\nu$ for every Borel set $A\subseteq X$.
Then $\nu = \lambda\nu_0 + (1-\lambda)\nu_1$.
Moreover, for every Borel set $B\subseteq Y$, $\pi_*(\nu_i)(B) = \int_{\pi^{-1}(B)}(h_i\circ\pi)d\nu = \int_B h_i d(\pi_*\nu) = \mu_i(B)$.
That is, $\pi_*(\nu_0)=\mu_0$ and $\pi_*(\nu_1)=\mu_1$, as desired.
\end{proof}

\begin{rmk}
One may also explain the preceding proposition by duality, using \autoref{thm:convex-lifting-maps}.
Indeed, as one can show, the map
\begin{equation*}
\ell^\infty_\lsc(X) \to \cA_\lsc\big(\cM(X)\big),\quad f\mapsto \big(f'\colon \mu\to \mu(f)\big),
\end{equation*}
is an isomorphism of ordered convex cones.
Therefore, if $\pi\colon X\to Y$ is a continuous surjection (the non-surjective case follows easily), the map $\cA_\lsc\big(\cM(Y)\big) \to \cA_\lsc\big(\cM(X)\big)$, $f\mapsto f\circ\pi_*$ can be identified with the map
\begin{equation*}
\ell^\infty_\lsc(Y) \to \ell^\infty_\lsc(X),\quad f\mapsto f\circ\pi.
\end{equation*}
By \autoref{rmk:theta_f} and \autoref{lem:theta-continuity}, the latter map has an upper adjoint, and therefore $\pi_*$ is face-preserving.
\end{rmk}

The following nice criterion is precisely \cite[Prop.~1.3]{Villadsen}, up to the added metrizability assumption (see \autoref{rmk:missing-hypothesis-Villadsen}).

\begin{prop}
\label{prop:ext-preserving-simplex}
Let $X$ be a metrizable compact convex set, and let $\pi\colon X \to Y$ be an affine, continuous map preserving the extreme points, i.e., $\pi\big(\cE(X)\big) \subseteq \cE(Y)$.
Assume $Y$ is a Choquet simplex.
Then $\pi$ is face-preserving.
\end{prop}
\begin{proof}
We prove $\pi$ has the convex lifting property.
Suppose $\pi(x) = \lambda y_0 + (1-\lambda)y_1$ for some $x\in X$, $y_0,y_1\in Y$ and $0< \lambda< 1$.
Let $\mu_0$, $\mu_1$ and $\mu$ be the unique boundary measures on $Y$ representing $y_0$, $y_1$ and $\pi(x)$, respectively.
In particular, $\mu = \lambda\mu_0 + (1-\lambda)\mu_1$.
On the other hand, take any $\nu\in \partial\cM_x(X)$.
By \autoref{lem:boundary-measure-preservation}, $\pi_*(\nu)$ is a boundary measure.
Since $R(\pi_*(\nu)) = \pi(x)$, it must be $\pi_*(\nu) = \mu$.
Now, by \autoref{prop:M(X)-convex-lifting} (and \autoref{lem:Villadsen}), we can find $\nu_0,\nu_1\in\cM(X)$ with $\pi_*(\nu_0) = \mu_0$, $\pi_*(\nu_1) = \mu_1$ and $\nu = \lambda\nu_0 + (1-\lambda)\nu_1$.
Letting $x_0 = R(\nu_0)$, $x_1 = R(\nu_1)$, we have $\pi(x_0) = y_0$, $\pi(x_1) = y_1$ and $x = \lambda x_0 + (1-\lambda)x_1$, as desired.
\end{proof}

The following lemma shows that for face-preserving maps, \autoref{lem:boundary-measure-preservation} holds with no metrizability assumption.

\begin{lem}
\label{lem:boundary-measure-preservation-face-preserving}
Let $\pi\colon X\to Y$ be an affine, continuous, face-preserving map between compact convex sets.
Then the following hold:
\begin{enumerate}
\item
\label{item:bound-meas-pres:envelope}
For every $f\in C(Y)$, $\check{f}\circ\pi = \widecheck{f\circ\pi}$.
\item
\label{item:bound-meas-pres:measures}
For every $\mu\in \partial\cM(X)$ we have $\pi_*(\mu)\in \partial\cM(Y)$.
\end{enumerate}
\end{lem}
\begin{proof}
The first point is precisely \cite[Lemma~1.8]{Villadsen}.
The second is a consequence of the first, exactly as in the proof of \cite[Prop.~3.24]{BITaffine}.
\end{proof}

Finally, the following is a more general form of \cite[Prop.~1.9]{Villadsen}.

\begin{prop}
Let $\pi\colon X\to Y$ be an affine, continuous, face-preserving function between compact convex sets.
Then for every $x\in X$, the pushforward map $\pi_*\colon\cM(X)\to \cM(Y)$ restricts to an affine surjection
\begin{equation*}
\pi_*^{\partial\cM_x} \colon \partial\cM_x(X)\to \partial\cM_{\pi(x)}(Y).
\end{equation*}
In particular, if $\pi$ is surjective and $X$ is a Choquet simplex, then so also is $Y$.
\end{prop}
\begin{proof}
By \autoref{lem:boundary-measure-preservation-face-preserving}, the restriction $\pi_*^{\partial\cM_x}$ is well-defined.
On the other hand, by \autoref{prop:CLP-with-measures}, $\pi_*^{\partial\cM_x}$ is surjective.

If $\pi$ is surjective and $X$ is a Choquet simplex, then $\partial\cM_y(Y)$ must be a singleton for every $y\in Y$. Therefore, $Y$ is a Choquet simplex.
\end{proof}

We now combine our previous analyses of open and face-preserving maps.

\begin{thm}
\label{thm:affinely-open-maps}
Let $\pi\colon X\to Y$ be an affine, continuous surjection between compact convex sets.
Then the following are equivalent:
\begin{enumerate}
\item\label{item:open-and-convex-lifting:lifting-prop}
$\pi$ is open and face-preserving.

\item\label{item:open-and-convex-lifting:theta}
For every $f\in \cA(X)$ we have $\check{f}^\pi \in \cA(Y)$.

\item\label{item:open-and-convex-lifting:adjoint}
The dual map $\pi^{*\aff}\colon \cA(Y) \to \cA(X)$ admits an upper adjoint.
\end{enumerate}
Moreover, if these hold then the upper adjoint of $\pi^{*\aff}$ is given by the map $f\mapsto \check{f}^\pi$.
\end{thm}
\begin{proof}
\begin{cycprf}
\item By \autoref{lem:theta-continuity} and \autoref{lem:theta-affineness}.

\item By \autoref{rmk:theta_f}.

\item[\impprev] If $\alpha\colon \cA(X)\to \cA(Y)$ is an upper adjoint for $\pi^{*\aff}$, it follows from \autoref{lem:check-f-equivalent-def-aff} that $\check{f}^\pi = \alpha(f)\in\cA(Y)$ for every $f\in\cA(X)$.

\item[\impprev] By \autoref{thm:Vesterstrom}, $\pi$ is open, and by \autoref{thm:convex-lifting-maps}\autoref{item:convex-lifting:Aff-to-Aff-lsc}, $\pi$ has the convex lifting property.
\end{cycprf}
\end{proof}

It seems worth mentioning the following result of Ditor and Eifler \cite{DitorEifler} (seemingly obtained independently by Vesterstr{\o}m, see \cite[Thm.~4.2]{Vesterstrom}), which we can derive immediately from the preceding results.

\begin{cor}
Let $\pi\colon X\to Y$ be a continuous surjection between compact Hausdorff spaces, and let $\pi_*\colon \cM(X) \to \cM(Y)$ be the corresponding pushforward map.
Then $\pi$ is open if and only if $\pi_*$ is open.
\end{cor}
\begin{proof}
Up to composing with the canonical vector lattice isomorphisms $C(Y)\cong \cA\bigl(\cM(Y)\bigr)$ and $C(X) \cong \cA\bigl(\cM(X)\bigr)$, we may identify the dual map $\pi^*\colon C(Y) \to C(X)$ with the dual map
\begin{equation*}
(\pi_*)^{*\aff}\colon \cA\bigl(\cM(Y)\bigr) \to \cA\bigl(\cM(X)\bigr).
\end{equation*}
In particular, $\pi^*$ has an upper adjoint if and only if so does $(\pi_*)^{*\aff}$.
Thus, the result follows from \autoref{prop:open-maps}, \autoref{thm:affinely-open-maps} and \autoref{prop:M(X)-convex-lifting}.
\end{proof}

We end this appendix with a general version of the convex lifting property satisfied by open, face-preserving maps.

\begin{prop}
\label{prop:strong-CLP}
Let $\pi\colon X\to Y$ be an affine, continuous, open, face-preserving surjection between compact convex sets.
Let
\begin{equation*}
      \begin{tikzcd}
        X' \arrow{d}{\rho_X} \arrow{r}{\pi'} & Y' \arrow{d}{\rho_Y} \\
        X \arrow{r}{\pi} & Y
      \end{tikzcd}
\end{equation*}
be a commutative diagram of affine, continuous maps with the property that for every closed convex subset $C\subseteq X$ such that $\pi(C) = Y$, we have $\pi'\big(\rho_X^{-1}(C)\big) = Y'$.
Then for every $x\in X$ and every $y'\in Y'$, if $\pi(x) = \rho_Y(y')$ then there exists $x'\in X'$ with $\rho_X(x') = x$ and $\pi'(x') = y'$.
\end{prop}
\begin{proof}
The argument is essentially the same as in \cite[Lemma~3.6(i)]{BITaffine}.
Suppose $\pi(x) = \rho_Y(y')$, and let $C' = \pi'^{-1}(y')$ be the fiber of $\pi'$ over $y'$.
We want to show that $x\in \rho_X(C')$.
Otherwise, by the Hahn--Banach Separation Theorem, there is $f\in \cA(X)$ with $f(x) < 0$ and $f\big(\rho_X(x')\big) \geq 0$ for every $x' \in C'$.
Now let
\begin{equation*}
    C = \bigl\{ z \in X : f(z) = \check{f}^\pi\big(\pi(z)\big) \bigr\}.
\end{equation*}
By \autoref{thm:affinely-open-maps}, $\check{f}^\pi$ is continuous and affine, and therefore $C$ is closed and convex.
Moreover, $\pi(C) = Y$ because the infimum in the definition of $\check{f}^\pi$ is attained.
By our assumption on the commuting diagram, $\pi'(\rho_X^{-1}(C)) = Y'$.
In particular, there is $x' \in \rho_X^{-1}(C)\cap C'$.
Denoting $z = \rho_X(x')$, we have $\pi(z) = \rho_Y\big(\pi'(x')\big) = \rho_Y(y') = \pi(x)$. But then
\begin{equation*}
f(x) < 0 \leq f(z) = \check{f}^\pi\big(\pi(x)\big) \leq f(x),
\end{equation*}
a contradiction.
\end{proof}

\begin{rmk}\label{rmk:strong-CLP-to-CLP}
The convex lifting property corresponds to the diagram of \autoref{prop:strong-CLP} with $X' = X\times X$, $\rho_X(x_0,x_1) = \lambda x_0 + (1-\lambda)x_1$, similarly for $Y'$ and $\rho_Y$, and $\pi' = \pi\times\pi$. In turn, the lifting property for measures (\autoref{prop:CLP-with-measures}) corresponds to the diagram with $X'=\cM(X)$, $Y'=\cM(Y)$, $\pi' = \pi_*$ the pushforward map, and $\rho_X$, $\rho_Y$ the barycenter maps.
\end{rmk}

\section{Variable restriction maps and model completions}
\label{sec:appendix:model-completions}

We start by recalling the following basic fact.

\begin{lem}
\label{lem:pi-cont-open}
Let $T$ be a continuous logic theory.
Then for every pair of tuples of variables $x$ and $y$, the variable restriction map
\begin{equation*}
\pi^\cont_{x,y}\colon \tS^\cont_{xy}(T)\to \tS^\cont_x(T)
\end{equation*}
is continuous, surjective and open.
\end{lem}
\begin{proof}
Continuity and surjectivity are obvious.
The sets
\begin{equation*}
\llbracket\varphi<0\rrbracket = \{p\in\tS^\cont_{xy}(T) : p(\varphi)<0\}
\end{equation*}
for $\varphi\in\cL^\cont_{xy}$ form a basis of open neighborhoods of $\tS^\cont_{xy}(T)$, and the sets $\pi^\cont_{x,y}(\llbracket\varphi<0\rrbracket) = \llbracket\inf_y\varphi(x,y)<0\rrbracket$ are open in $\tS^\cont_x(T)$.
Hence $\pi^\cont_{x,y}$ is open.
\end{proof}

From the viewpoint developed in \autoref{sec:appendix:open-and-face-preserving}, and specifically, by \autoref{prop:open-maps}, the previous lemma can also be derived from the self-evident fact that the quantification map
\begin{equation*}
\cL^\cont_{xy}(T) \to \cL^\cont_x(T),\quad \varphi(x,y)\mapsto \inf_y\varphi(x,y)
\end{equation*}
is the upper adjoint of the inclusion $\cL^\cont_x(T) \subseteq \cL^\cont_{xy}(T)$.
This alternative argument will be worthwhile in the affine setting, as we shall see in \autoref{lem:pi-aff-open}.

We first state a criterion for the existence of model completions in continuous logic.
This may be considered folklore.
The special case of the theory $\PMP_G$ was recently put forward, and exploited remarkably, in \cite{GST2025}.

\begin{prop}
\label{prop:model-completion-main-cont}
Let $T$ be a Robinson theory in continuous logic.
Then the following are equivalent:
\begin{enumerate}
\item\label{i:main:T*} $T$ admits a model completion.
\item\label{i:main:piqf-open} For every pair of finite tuples of variables $x$ and $y$, the variable restriction map $\pi^\qf_{x,y} \colon \tS^\qf_{xy}(T) \to \tS^\qf_x(T)$ is open.
\end{enumerate}
\end{prop}
\begin{proof}
\begin{cycprf}
\item If $T^*$ is a model completion of $T$, then the canonical projections $\tS^\cont_x(T^*)\to \tS^\qf_x(T^*)$ are homeomorphisms, because $T^*$ has quantifier elimination.
On the other hand, by \autoref{rmk:Sqf(Tforall)} and since $T$ is the universal part of $T^*$, we can identify $\tS^\qf_x(T^*) = \tS^\qf_x(T)$.
These homeomorphisms commute with the variable restriction maps:
\begin{equation*}
      \begin{tikzcd}
        \tS^\cont_{xy}(T^*) \arrow{d}{\cong} \arrow{r}{\pi^\cont_{x,y}} & \tS^\cont_x(T^*) \arrow{d}{\cong} \\
        \tS^\qf_{xy}(T) \arrow{r}{\pi^\qf_{x,y}} & \tS^\qf_x(T)
      \end{tikzcd}
\end{equation*}
As the maps $\pi^\cont_{x,y}$ are open (\autoref{lem:pi-cont-open}), so are the maps $\pi^\qf_{x,y}$.

\item[\impfirst] For each formula $\varphi\in\cL^\qf_{xy}$ let us consider the function
\begin{equation*}
\alpha_\varphi = \check{\varphi}^{\pi^\qf_{x,y}} \colon \tS^\qf_x(T) \to \bR,\quad \alpha_\varphi(p) = \inf\bigl\{q(\varphi) : q\in \tS^\qf_{xy}(T),\, \pi^\qf_{x,y}(q) = p\bigr\}.
\end{equation*}
By \autoref{i:main:piqf-open} and \autoref{lem:theta-continuity}, the functions $\alpha_\varphi$ are continuous.
We may thus see them as quantifier-free definable predicates on models of $T$. By definition, we have:
\begin{equation*}
\label{eq:theta-leq-inf-phi}
T \models \sup_x \bigl(\alpha_\varphi(x) - \inf_y\varphi(x,y)\bigr) \leq 0.
\end{equation*}
Now let us consider the theory $T^*$ consisting of $T$ together with all the conditions of the form
\begin{equation}
\label{eq:inf-phi-leq-theta}
\sup_x \bigl(\inf_y \varphi(x,y) - \alpha_\varphi(x) \bigr) \leq 0,
\end{equation}
with $\varphi\in\cL^\qf_{xy}$.
Modulo $T^*$, every formula of the form $\inf_y\varphi(x,y)$ with $\varphi$ quantifier-free is equivalent to the quantifier-free definable predicate $\alpha_\varphi$.
It follows that $T^*$  has quantifier elimination.

To conclude that $T^*$ is the model completion of $T$, it remains to prove that every model $M\models T$ embeds into a model of $T^*$.
For each formula $\varphi\in\cL^\qf_{xy}$ and tuple $a\in M^x$, we may by compactness find a model $M_{\varphi,a}\models T$ containing the substructure $A$ generated by $a$, and a tuple $b\in M^y$ such that
\begin{equation*}
\alpha_\varphi(a) = \varphi(a,b).
\end{equation*}
Since $T$ has the amalgamation property, we may amalgamate $M$ and the structures $M_{\varphi,a}$ (as $\varphi$ and $a$ vary) over $A$ to get a model $M_1\models T$.
We may iterate this construction to get a countable chain $M \subseteq M_1\subseteq M_2\subseteq\dots$ of models of $T$ with the property that for every $\varphi\in\cL^\qf_{xy}$ and $a\in M_n^x$,
\begin{equation*}
M_{n+1} \models \inf_y\varphi(a,y) \leq \alpha_\varphi(a).
\end{equation*}
The direct limit $N = \cl{\bigcup_{n\in\bN}M_n}$ is a model of $T^*$ extending $M$, as desired. 
\end{cycprf}
\end{proof}

We now turn to affine logic.

\begin{lem}
\label{lem:pi-aff-open}
Let $T$ be an affine theory.
Then for every pair of tuples of variables $x$ and $y$, the variable restriction map
\begin{equation*}
\pi^\aff_{x,y}\colon \tS^\aff_{xy}(T)\to \tS^\aff_x(T)
\end{equation*}
is affine, continuous, surjective, open and face-preserving.
\end{lem}
\begin{proof}
The quantification map
\begin{equation*}
\cL^\aff_{xy}(T) \to \cL^\aff_x(T),\quad \varphi(x,y)\mapsto \inf_y\varphi(x,y)
\end{equation*}
is the upper adjoint of the inclusion $\cL^\aff_x(T) \subseteq \cL^\aff_{xy}(T)$.
It is a contraction, and extends by continuity to an upper adjoint for the dual map $(\pi^\aff_{x,y})^*\colon \cA(\tS^\aff_x(T)) \to \cA(\tS^\aff_{xy}(T))$ (i.e., the inclusion of the space of affine definable predicates in the variables $x$ into the space of affine definable predicates in the variables $xy$).
Therefore, by \autoref{thm:affinely-open-maps}, $\pi^\aff_{x,y}$ is open and face-preserving.
\end{proof}

The fact that the variable restriction maps $\pi_{x,y}^\aff\colon \tS^\aff_{xy}(T)\to \tS^\aff_x(T)$ have the convex lifting property was first observed in \cite[Lemma~3.6(iii)]{BITaffine}.
From that fact, it was deduced that their restriction to the corresponding extreme type spaces,
\begin{equation*}
\pi^\cE_{x,y} \colon \cE_{xy}(T)\to \cE_x(T),
\end{equation*}
are well-defined, continuous, surjective, open maps (\cite[Cor.~5.5]{BITaffine}).
The question of whether the maps $\pi^\aff_{x,y}$ themselves are open did not arise explicitly in~\cite{BITaffine}.\footnote{We may remark, however, that the openness of the maps $\pi^\aff_{x,y}$ does follow from our results in \cite{BITaffine}, and more precisely from the theory of \emph{convex realization completions} (which also plays a role below, in \autoref{thm:affine-vs-continuous-model-completion}).
Indeed, by \cite[Thm.~17.7]{BITaffine}, the affine part maps $\rho^\aff_x\colon \tS^\cont_x(T_\crc) \to \tS^\aff_x(T)$ are homeomorphisms, and since they commute with the appropriate variable restriction maps, the openness of the maps $\pi^\aff_{x,y}$ reduces to the openness of the maps $\pi^\cont_{x,y}$.}
Note that the sets of the form $\llbracket\varphi<0\rrbracket$ do not form a basis of the whole affine type spaces $\tS^\aff_x(T)$, and therefore the basic argument of \autoref{lem:pi-cont-open} does not go through in this setting.
Instead, one needs Vesterstr{\o}m's argument (\autoref{thm:Vesterstrom}).

We now prove an affine version of \autoref{prop:model-completion-main-cont}.

\begin{thm}
\label{thm:model-completion-main-aff}
Let $T$ be an affine Robinson theory.
The following are equivalent:
\begin{enumerate}
\item\label{i:main:T*aff} $T$ admits an affine model completion, $T^{*\aff}$.
\item\label{i:main:piqf-open-lifting} For every pair of finite tuples of variables $x$ and $y$, the variable restriction map $\pi^\qf_{x,y} \colon \tS^\qf_{xy}(T) \to \tS^\qf_x(T)$ is open and face-preserving.
\end{enumerate}
Moreover, if these conditions hold, then $T^{*\aff}$ is a face of $T$.
In particular, if $T$ is a simplicial theory, then so is $T^{*\aff}$.
\end{thm}
\begin{proof}
\begin{cycprf}
\item If the affine model completion $T^{*\aff}$ exists, we may identify $\tS^\qf_x(T)$ with $\tS^\aff_x(T^{*\aff})$.
Hence, the conclusion follows from \autoref{lem:pi-aff-open}.

\item[\impfirst] We proceed as in \impref{i:main:piqf-open}{i:main:T*} of \autoref{prop:model-completion-main-cont}, but considering only affine quantifier-free formulas $\varphi\in\cL^{\qf,\aff}_{xy}$. 
Indeed, by \autoref{i:main:piqf-open-lifting} and \autoref{thm:affinely-open-maps}, the associated functions
\begin{equation*}
\alpha_\varphi \colon \tS^\qf_x(T) \to \bR,\quad \alpha_\varphi(p) = \inf\bigl\{q(\varphi) : q\in \tS^\qf_{xy}(T),\, \pi^\qf_{x,y}(q) = p\bigr\},
\end{equation*}
are continuous and affine, and can thus be seen as quantifier-free affine definable predicates on models of $T$. We have by construction
\begin{equation}
\label{eq:theta-leq-inf-phi-affine}
T \models \sup_x \bigl(\alpha_\varphi(x) - \inf_y\varphi(x,y)\bigr) \leq 0,
\end{equation}
and we may consider the $\cL$-theory $T^{*\aff}\supseteq T$ obtained by adding the affine conditions
\begin{equation}
\label{eq:inf-phi-leq-theta-affine}
\sup_x \bigl(\inf_y \varphi(x,y) - \alpha_\varphi(x) \bigr) \leq 0
\end{equation}
for each $\varphi\in\cL^{\qf,\aff}_{xy}$. Clearly,  $T^{*\aff}$ has affine quantifier elimination. On the other hand, every model of $T$ embeds into a model of $T^{*\aff}$, by the same amalgamation argument as in the continuous case. Therefore, $T^{*\aff}$ is the affine model completion of $T$.
\smallskip

For the moreover part, since $T$ satisfies the conditions in \autoref{eq:theta-leq-inf-phi-affine}, each of the conditions in \autoref{eq:inf-phi-leq-theta-affine} defines an (exposed) face of $\tS^\aff_0(T)$.
Their intersection is the theory $T^{*\aff}$, which is thus a face of $T$.
As a consequence, if the theory $T$ is simplicial, then so is $T^{*\aff}$, by \cite[Prop.~11.4]{BITaffine}.
\end{cycprf}
\end{proof}

From \autoref{prop:model-completion-main-cont} and \autoref{thm:model-completion-main-aff}, we deduce the following.

\begin{cor}
Let $G$ be any group.
Then the theory $\PMP_G$ has a model completion in continuous logic if and only if it has a model completion in affine logic.
\end{cor}
\begin{proof}
Since $\PMP_G$ is face-preserving (\autoref{cor:PMPG-basics}), this is indeed a consequence of the preceding criteria.
\end{proof}

In fact, the previous corollary holds for \emph{any} irreducible, affine Robinson theory.

In order to see this, let us first recall a notion introduced in \cite[\textsection17]{BITaffine}.
Given an affine theory $T$, the \emph{convex realization completion} of $T$, denoted by $T_\crc$, is the continuous logic theory obtained from $T$ by adding the axiom scheme of the convex realization property.
The theory $T_\crc$ can be equivalently defined as the common continuous theory of the direct integrals of models of $T$ over atomless probability spaces, or as the common continuous theory of the affinely $\aleph_0$-saturated models of~$T$.

It is always the case that $(T_\crc)_\aff = T$, and one may therefore consider the maps
\begin{equation*}
\rho^\aff_x\colon \tS^\cont_x(T_\crc) \to \tS^\aff_x(T)
\end{equation*}
sending a continuous type to its affine part.
One of the main properties of the convex realization completion is that these maps are homeomorphisms.
Another important property is that if $T$ is affinely complete, then $T_\crc$ is complete in continuous logic.
See \cite[Thm.~17.7]{BITaffine}.

We also have the following basic fact.

\begin{lem}
\label{lem:Tcr-forall}
Let $T$ be an affine theory.
Then $(T_\crc)_\forall = T_\forall$.
\end{lem}
\begin{proof}
Clearly, $T_\forall\subseteq (T_\crc)_\forall$. Conversely, if $M\models T_\forall$ then $M$ embeds into a model of $T$, which in turn embeds into a model of $T_\crc$ (e.g., by taking any direct multiple over an atomless probability space). Hence the other inclusion holds as well.
\end{proof}

\begin{thm}
\label{thm:affine-vs-continuous-model-completion}
Let $T$ be an irreducible, affine Robinson theory.
Then the following are equivalent:
\begin{enumerate}
\item $T$ admits a continuous logic model completion, $T^*$.
\item $T$ admits an affine model completion, $T^{*\aff}$.
\end{enumerate}
Moreover, if these hold, then $T^* = (T^{*\aff})_\crc$ and $T^{*\aff} = (T^*)_\aff$.
\end{thm}
\begin{proof}
\begin{cycprf}
\item Let us show that $(T^*)_\aff$ is the affine model completion of $T$.
First, by \autoref{lem:forall-aff}, $\bigl((T^*)_\aff\bigr)_\forall = (T^*)_{\forall^\aff} = \bigl((T^*)_\forall\bigr)_{\forall^\aff} = T$.

Let $M\models T^*$. Since any model of $T$ embeds into a model of $T_\crc$, and since $M\models_\ec T$, it follows from \autoref{lem:ec-vs-models-AE} and \autoref{cor:crp-forall-exists-axiomatizable} that $M$ has the convex realization property.
This shows that $T^* \models ((T^*)_\aff)_\crc$.
On the other hand, since $T^*$ is complete in continuous logic (\autoref{rmk:irreducible-model-completion}), so is $((T^*)_\aff)_\crc$, by \cite[Thm.~17.7(iv)]{BITaffine}.
Therefore, $T^* = ((T^*)_\aff)_\crc$.

Now, since $T^*$ eliminates quantifiers in continuous logic, given a tuple of variables $x$, the composition of the canonical projections
\begin{equation*}
\tS^\cont_x(T^*) = \tS^\cont_x\bigl(((T^*)_\aff)_\crc\bigr) \to
\tS^\aff_x\bigl((T^*)_\aff\bigr) \to \tS^\qf_x(T)
\end{equation*}
is a homeomorphism.
By \cite[Thm.~17.7(iii)]{BITaffine}, the first projection is also a homeomorphism, and therefore so is the second projection.
Hence $(T^*)_\aff$ has affine quantifier elimination, and it is thus the affine model completion of $T$.

\item[\impfirst] The implication follows from the criteria given earlier, but let us show directly that $(T^{*\aff})_\crc$ is the continuous model completion of $T$.

By \autoref{lem:Tcr-forall}, $\bigl((T^{*\aff})_\crc\bigr)_\forall = (T^{*\aff})_\forall = T$.
On the other hand, the first of the canonical projections
\begin{equation*}
\tS^\cont_x\bigl((T^{*\aff})_\crc\bigr) \to \tS^\aff_x(T^{*\aff}) \to \tS^\qf_x(T)
\end{equation*}
is a homeomorphism, again by \cite[Thm.~17.7(iii)]{BITaffine}.
Since $T^{*\aff}$ has affine quantifier elimination, the second projection is a homeomorphism as well, and so is their composition. Hence $(T^{*\aff})_\crc$ has quantifier elimination in continuous logic, and therefore it is the model completion of $T$.
\end{cycprf}
\end{proof}

\begin{cor}
Let $T$ be an irreducible, affine Robinson theory admitting continuous and affine model completions, $T^*$ and $T^{*\aff}$.
If $T$ is qf-Poulsen, then $T^* = T^{*\aff}$.
\end{cor}
\begin{proof}
If $T$ is qf-Poulsen then $T^{*\aff}$ is Poulsen in the sense of \cite{BITaffine}.
Therefore, using \cite[Thm.~20.4]{BITaffine}, $T^{*\aff} = (T^{*\aff})_\crc = T^*$.
\end{proof}

\begin{cor}
Let $T$ be an irreducible, affine Robinson theory.
If the variable projection maps $\pi^\qf_{x,y}\colon \tS^\qf_{xy}(T)\to \tS^\qf_x(T)$ are open, then they are face-preserving.

If moreover $T$ is qf-simplicial, then it is qf-Bauer or qf-Poulsen.
\end{cor}
\begin{proof}
By \autoref{prop:model-completion-main-cont}, \autoref{thm:affine-vs-continuous-model-completion}, \autoref{thm:model-completion-main-aff}, and \autoref{cor:dichotomy-model-completion}.
\end{proof}

\vspace{0.2em}

\bibliographystyle{amsalpha}
\bibliography{affinely-ec-models}

\providecommand{\bysame}{\leavevmode\hbox to3em{\hrulefill}\thinspace}
\providecommand{\MR}{\relax\ifhmode\unskip\space\fi MR }
\providecommand{\MRhref}[2]{%
  \href{http://www.ams.org/mathscinet-getitem?mr=#1}{#2}
}
\providecommand{\href}[2]{#2}
\begin{thebibliography}{BBHU08}

\bibitem[Ald85]{Aldous1985}
D.~J. Aldous, \emph{Exchangeability and related topics}, \'Ecole d'\'et\'e{} de
  probabilit\'es de {S}aint-{F}lour, {XIII}---1983, Lecture Notes in Math.,
  vol. 1117, Springer, Berlin, 1985, pp.~1--198. \MR{883646}

\bibitem[Alf71]{Alfsen1971}
E.~M. Alfsen, \emph{Compact convex sets and boundary integrals},
  Springer-Verlag, New York-Heidelberg, 1971, Ergebnisse der Mathematik und
  ihrer Grenzgebiete, Band 57. \MR{0445271}

\bibitem[Aus08]{Austin2008}
T.~Austin, \emph{On the geometry of a class of invariant measures and a problem
  of {A}ldous}, preprint \texttt{arXiv:0808.2268}, 2008.

\bibitem[AW13]{AbertWeiss}
M.~Ab\'ert and B.~Weiss, \emph{Bernoulli actions are weakly contained in any
  free action}, Ergodic Theory Dynam. Systems \textbf{33} (2013), no.~2,
  323--333. \MR{3035287}

\bibitem[Bag14]{Bagheri2014}
S.-M. Bagheri, \emph{Linear model theory for {L}ipschitz structures}, Arch.
  Math. Logic \textbf{53} (2014), no.~7-8, 897--927. \MR{3271367}

\bibitem[Bag24]{Bagheri2024p}
\bysame, \emph{Extreme types and extremal models}, Ann. Pure Appl. Logic
  \textbf{175} (2024), no.~7, Paper No. 103451, 18. \MR{4735719}

\bibitem[BBHU08]{BBHU}
I.~{Ben Yaacov}, A.~Berenstein, C.~W. Henson, and A.~Usvyatsov, \emph{Model
  theory for metric structures}, Model theory with applications to algebra and
  analysis. {V}ol. 2, London Math. Soc. Lecture Note Ser., vol. 350, Cambridge
  Univ. Press, Cambridge, 2008, pp.~315--427.

\bibitem[BGK15]{BowenLamplighter2015}
L.~Bowen, R.~Grigorchuk, and R.~Kravchenko, \emph{Invariant random subgroups of
  lamplighter groups}, Israel J. Math. \textbf{207} (2015), no.~2, 763--782.
  \MR{3359717}

\bibitem[BH]{Berenstein2018p}
A.~Berenstein and C.~W. Henson, \emph{Model theory of probability spaces with
  some countable groups of automorphisms}, unpublished.

\bibitem[BHI24]{BerHenIba}
A.~Berenstein, C.~W. Henson, and T.~Ibarluc\'ia, \emph{Existentially closed
  measure-preserving actions of free groups}, Fund. Math. \textbf{264} (2024),
  no.~3, 241--282. \MR{4731992}

\bibitem[BIT]{BITaffine}
I.~{Ben Yaacov}, T.~Ibarluc{\'i}a, and T.~Tsankov, \emph{Extremal models and
  direct integrals in affine logic}, preprint, arXiv:2407.13344.

\bibitem[Bow15]{BowenIRS2015}
L.~Bowen, \emph{Invariant random subgroups of the free group}, Groups Geom.
  Dyn. \textbf{9} (2015), no.~3, 891--916. \MR{3420547}

\bibitem[BS14]{Bagheri2014a}
S.-M. Bagheri and R.~Safari, \emph{Preservation theorems in linear continuous
  logic}, MLQ Math. Log. Q. \textbf{60} (2014), no.~3, 168--176. \MR{3207206}

\bibitem[BU10]{BenYaacov2010}
I.~{Ben Yaacov} and A.~Usvyatsov, \emph{Continuous first order logic and local
  stability}, Trans. Amer. Math. Soc. \textbf{362} (2010), no.~10, 5213--5259.
  \MR{2657678}

\bibitem[BV93]{BekkaValette}
B.~Bekka and A.~Valette, \emph{Kazhdan's property {$({\rm T})$} and amenable
  representations}, Math. Z. \textbf{212} (1993), no.~2, 293--299. \MR{1202813}

\bibitem[CS23]{Candela2023}
P.~Candela and B.~Szegedy, \emph{Nilspace factors for general uniformity
  seminorms, cubic exchangeability and limits}, Mem. Amer. Math. Soc.
  \textbf{287} (2023), no.~1425, v+101. \MR{4608886}

\bibitem[CW80]{Connes1980}
A.~Connes and B.~Weiss, \emph{Property {${\rm T}$} and asymptotically invariant
  sequences}, Israel J. Math. \textbf{37} (1980), no.~3, 209--210.

\bibitem[DE72]{DitorEifler}
S.~Z. Ditor and L.~Q. Eifler, \emph{Some open mapping theorems for measures},
  Trans. Amer. Math. Soc. \textbf{164} (1972), 287--293. \MR{477729}

\bibitem[Dow91]{Downarowicz}
T.~Downarowicz, \emph{The {C}hoquet simplex of invariant measures for minimal
  flows}, Israel J. Math. \textbf{74} (1991), no.~2-3, 241--256. \MR{1135237}

\bibitem[ET16]{EvaTsa}
D.~M. Evans and T.~Tsankov, \emph{Free actions of free groups on countable
  structures and property ({T})}, Fund. Math. \textbf{232} (2016), no.~1,
  49--63. \MR{3417738}

\bibitem[Gir19]{Giraud2019p}
A.~Giraud, \emph{Hyperfinite measure-preserving actions of countable groups and
  their model theory}, preprint \texttt{arXiv:1910.07985}, 2019.

\bibitem[GIV26]{GaoCharacters2026}
D.~Gao, A.~Ioana, and I.~Vigdorovich, \emph{Characters of surface groups},
  preprint \texttt{arXiv:2605.02242}, 2026.

\bibitem[Gla03]{Glasner2003}
E.~Glasner, \emph{Ergodic theory via joinings}, Mathematical Surveys and
  Monographs, vol. 101, American Mathematical Society, Providence, RI, 2003.

\bibitem[GLM24]{GlasnerYair2024}
Y.~Glasner, Y.~F. Lin, and T.~Meyerovitch, \emph{Extensions of invariant random
  orders on groups}, Groups Geom. Dyn. \textbf{18} (2024), no.~4, 1377--1401.
  \MR{4797633}

\bibitem[GST25]{GST2025}
I.~Goldbring, B.~Seward, and R.~{Tucker-Drob}, \emph{Existentially closed
  measure-preserving actions of approximately treeable groups}, preprint
  \texttt{arXiv:2507.03195}, 2025.

\bibitem[GW97]{Glasner1997}
E.~Glasner and B.~Weiss, \emph{Kazhdan's property {T} and the geometry of the
  collection of invariant measures}, Geom. Funct. Anal. \textbf{7} (1997),
  no.~5, 917--935.

\bibitem[GW05]{Glasner2005}
\bysame, \emph{Spatial and non-spatial actions of {P}olish groups}, Ergodic
  Theory Dynam. Systems \textbf{25} (2005), no.~5, 1521--1538. \MR{2173431}

\bibitem[Hay75]{Haydon1975}
R.~Haydon, \emph{A new proof that every {P}olish space is the extreme boundary
  of a simplex}, Bull. London Math. Soc. \textbf{7} (1975), 97--100.
  \MR{358312}

\bibitem[Hru19]{HruPatterns}
E.~Hrushovski, \emph{Definability patterns and their symmetries}, preprint
  \texttt{arXiv:1911.01129}, 2019.

\bibitem[Iba21]{Iba2021}
T.~Ibarluc\'{\i}a, \emph{Infinite-dimensional {P}olish groups and {P}roperty
  ({T})}, Invent. Math. \textbf{223} (2021), no.~2, 725--757. \MR{4209862}

\bibitem[ISV25]{IoanaTrace2025}
A.~Ioana, P.~Spaas, and I.~Vigdorovich, \emph{Trace spaces of full free product
  {$C^*$}-algebras}, Compos. Math. \textbf{161} (2025), no.~11, 2947--2989.
  \MR{5012844}

\bibitem[IT21]{IbarluciaTsankov}
T.~Ibarluc\'{\i}a and T.~Tsankov, \emph{A model-theoretic approach to rigidity
  of strongly ergodic, distal actions}, Ann. Sci. \'{E}c. Norm. Sup\'{e}r. (4)
  \textbf{54} (2021), no.~3, 751--777. \MR{4311098}

\bibitem[Jol05]{Jolissaint2005}
P.~Jolissaint, \emph{On property ({T}) for pairs of topological groups},
  Enseign. Math. (2) \textbf{51} (2005), no.~1-2, 31--45. \MR{2154620}

\bibitem[Kal05]{Kallenberg2005}
O.~Kallenberg, \emph{Probabilistic symmetries and invariance principles},
  Probability and its Applications (New York), Springer, New York, 2005.
  \MR{2161313}

\bibitem[Kec95]{Kechris1995}
A.~S. Kechris, \emph{Classical descriptive set theory}, Graduate Texts in
  Mathematics, vol. 156, Springer-Verlag, New York, 1995.

\bibitem[Kec10]{KechrisGlobal}
\bysame, \emph{Global aspects of ergodic group actions}, Mathematical Surveys
  and Monographs, vol. 160, American Mathematical Society, Providence, RI,
  2010. \MR{2583950}

\bibitem[LOS78]{Lindenstrauss1978}
J.~Lindenstrauss, G.~Olsen, and Y.~Sternfeld, \emph{The {P}oulsen simplex},
  Ann. Inst. Fourier (Grenoble) \textbf{28} (1978), no.~1, vi, 91--114.

\bibitem[Mac62]{Mackey1962}
G.~W. Mackey, \emph{Point realizations of transformation groups}, Illinois J.
  Math. \textbf{6} (1962), 327--335. \MR{143874}

\bibitem[Phe01]{Phelps2001}
R.~R. Phelps, \emph{Lectures on {C}hoquet's theorem}, second ed., Lecture Notes
  in Mathematics, vol. 1757, Springer-Verlag, Berlin, 2001.

\bibitem[Pou61]{Poulsen1961}
E.~T. Poulsen, \emph{A simplex with dense extreme points}, Ann. Inst. Fourier
  (Grenoble) \textbf{11} (1961), 83--87, XIV. \MR{123903}

\bibitem[Ram71]{Ramsay1971}
A.~Ramsay, \emph{Virtual groups and group actions}, Advances in Math.
  \textbf{6} (1971), 253--322. \MR{281876}

\bibitem[Sim68]{Simons}
S.~Simons, \emph{Extended and sandwich versions of the {H}ahn-{B}anach
  theorem}, J. Math. Anal. Appl. \textbf{21} (1968), 112--122. \MR{222601}

\bibitem[Slu26]{Slutsky2026}
R.~Slutsky, \emph{Invariant trace simplices and relative property ({T})},
  preprint \texttt{arXiv:2604.24738}, 2026.

\bibitem[Tsa12]{TsankovUnitary}
T.~Tsankov, \emph{Unitary representations of oligomorphic groups}, Geom. Funct.
  Anal. \textbf{22} (2012), no.~2, 528--555. \MR{2929072}

\bibitem[Ves73]{Vesterstrom}
J.~Vesterstr{\o}m, \emph{On open maps, compact convex sets, and operator
  algebras}, J. London Math. Soc. (2) \textbf{6} (1973), 289--297. \MR{315464}

\bibitem[Vil95]{Villadsen}
J.~Villadsen, \emph{The range of the {E}lliott invariant}, J. Reine Angew.
  Math. \textbf{462} (1995), 31--55. \MR{1329901}

\end{thebibliography}

\end{document}